%% file: secant_arxiv.tex
\documentclass{amsart}

\usepackage[foot]{amsaddr}
\usepackage{xr}
\usepackage{lipsum}
\usepackage{hyperref}
\usepackage[mathscr]{euscript}
\usepackage{graphicx}
\usepackage{epstopdf}
\usepackage{tikz}
\usepackage{onimage}
\usepackage{amsfonts}
\usepackage{amsmath}
\usepackage{amssymb}
\usepackage{amsthm}
\usepackage{amsopn}
\usepackage{latexsym}
\usepackage{bbm}
\usepackage[caption=false]{subfig}
\ifpdf
  \DeclareGraphicsExtensions{.eps,.pdf,.png,.jpg}
\else
  \DeclareGraphicsExtensions{.eps}
\fi

\theoremstyle{definition}
\newtheorem{definition}{Definition}[section]
\newtheorem{remark}{Remark}[section]

\theoremstyle{plain}
\newtheorem{proposition}[definition]{Proposition}
\newtheorem{theorem}[definition]{Theorem}
\newtheorem{lemma}[definition]{Lemma}

\makeatletter
\renewcommand*\env@matrix[1][*\c@MaxMatrixCols c]{%
  \hskip -\arraycolsep
  \let\@ifnextchar\new@ifnextchar
  \array{#1}}
\makeatother

\DeclareMathOperator*{\minimize}{\min\!imize\enskip}
\DeclareMathOperator*{\maximize}{\max\!imize\enskip}
\DeclareMathOperator{\diam}{diam}

\newcommand{\mat}[1]{{\mathbf{#1}}}
\newcommand{\vect}[1]{{\mathbf{#1}}}

\newcommand{\calX}{\mathcal{X}}
\newcommand{\scrM}{\mathscr{M}}
\newcommand{\scrS}{\mathscr{S}}
\newcommand{\xx}{\vect{x}}

\newcommand{\PhiS}{\boldsymbol{\Phi}_\scrS}
\newcommand{\mS}{\vect{m}_\scrS}
\newcommand{\bbone}{\mathbbm{1}}

\newenvironment{acknowledgements}
{\subsection*{Acknowledgements}}

\begin{document}

\title[Nonlinear Feature Selection Using Secants]{Inadequacy of Linear Methods for Minimal Sensor Placement and Feature Selection in Nonlinear Systems; a New Approach Using Secants}

\author{Samuel E. Otto}
\address{Department of Mechanical and Aerospace Engineering, Princeton
  University, Princeton, NJ 08544}
\email{sotto@princeton.edu}

\author{Clarence W. Rowley}

\thanks{This research was supported by the Army Research Office under grant
  number W911NF-17-1-0512.  S.E.O. was supported by a National Science Foundation Graduate Research Fellowship Program under Grant No. DGE-2039656.
Any opinions, findings, and conclusions or recommendations expressed in this
material are those of the authors and do not necessarily reflect the views of
the National Science Foundation.}


\begin{abstract}
  Sensor placement and feature selection are critical steps in engineering, modeling, and data science that share a common mathematical theme: the selected measurements should enable solution of an inverse problem.
  Most real-world systems of interest are nonlinear, yet the majority of available techniques for feature selection and sensor placement rely on assumptions of linearity or simple statistical models.
  We show that when these assumptions are violated, standard techniques can lead to costly over-sensing without guaranteeing that the desired information can be recovered from the measurements.
  In order to remedy these problems, we introduce a novel data-driven approach for sensor placement and feature selection for a general type of nonlinear inverse problem based on the information contained in secant vectors between data points.
  Using the secant-based approach, we develop three efficient greedy algorithms that each provide different types of robust, near-minimal reconstruction guarantees.
  We demonstrate them on two problems where linear techniques consistently fail: sensor placement to reconstruct a fluid flow formed by a complicated shock-mixing layer interaction and selecting fundamental manifold learning coordinates on a torus.
\end{abstract}

\maketitle

\input{secant_paper}

\bibliographystyle{abbrv}
\bibliography{References}

\appendix

\input{secant_appendices}

\end{document}

%% file: secant_paper.tex

\section{Introduction}
\label{intro}

Reconstructing the state of complex systems like fluid flows, chemical processes, and biological networks from measurements taken by a few carefully chosen sensors is a crucial task for controlling, forecasting, and building simplified models of these systems.
In this setting it is important to be able to reconstruct the relevant information about the system using the smallest total number of measurements which includes minimizing the number of sensors to reduce cost, and using the shortest possible measurement histories to shorten response time.
Feature selection in statistics and machine learning is a related task where one tries to find a small subset of measured variables (features) in the available data that allow one to reliably predict a quantity of interest.

Nonlinear reconstruction can yield large improvements over linear reconstruction when the sensors or features are carefully selected \cite{Guyon2003introduction}.
Successful nonlinear reconstruction techniques include neural networks \cite{Nair2019integrating},\cite{Nair2019leveraging}, deep nonlinear state estimators \cite{Hosseinyalamdary2018deep}, \cite{Krishnan2017structured}, and convex $\ell^1$ minimization to reveal sparse coefficients in learned libraries \cite{Wright2008robust}, \cite{Callaham2019robust}.
The need for nonlinear representation and reconstruction is also recognized in the reduced-order-modeling community where it is called ``nonlinear Galerkin'' approximation \cite{Lee2020model}, \cite{Rega2005dimension}, \cite{Marion1989nonlinear}.
These methods are necessary because in many systems of interest, the state is found to lie near a low-dimensional underlying manifold that is curved in such a way that it is not contained in any low-dimensional subspace \cite{Ohlberger2016reduced}.
We will show that the best possible linear reconstruction accuracy is fundamentally limited by the number of measurements (features) and the fraction of the variance that is captured in the principal subspace \cite{Hotelling1933analysis} of that dimension.
In essence, any linear representation in a subspace is ``too loose'' and demands an excessive number of measurements to even have a hope of accurately reconstructing the state using linear functions.
Nonlinear reconstruction is much more powerful, as Whitney's celebrated embedding theorem (Theorem.~5, \cite{Whitney1944self}) shows that states on any $r$-dimensional smooth manifold can be reconstructed using $2 r$ carefully chosen measurements.
If the measurements must be linear functions of the state on a compact sub-manifold of $\mathbb{R}^n$ then $2r+1$ can be found \cite{Whitney1936differentiable}.

With many measurements available from our sensors (though not necessarily ones that achieve Whitney's results), the problem that remains is to properly choose them so that nonlinear reconstruction is possible and robust to noise.
While nonlinear reconstruction has proved to be extremely advantageous, the overwhelming majority of sensor placement and feature selection methods rely on measures of linear or Gaussian reconstruction accuracy as an optimization criteria.
Such methods include techniques based on sampling modal bases \cite{Yildirim2009efficient}, \cite{Manohar2018data}, \cite{Chaturantabut2010nonlinear}, \cite{Drmac2016new}, \cite{Businger1965linear}, linear dynamical system models \cite{Mons2017optimal}, \cite{Dhingra2014admm}, \cite{Summers2015submodularity}, \cite{Summers2016submodularity}, \cite{tzoumas2016sensor}, Bayesian and maximum likelihood optimality in linear inverse problems \cite{Chaloner1995bayesian}, \cite{Joshi2008sensor}, \cite{Shamaiah2010greedy}, information-theoretic criteria under Gaussian or other simple statistical models \cite{Krause2008near}, \cite{Caselton1984optimal}, \cite{Caselton1992quality}, \cite{Shewry1987maximum}, \cite{Sebastiani2000maximum}, and sparse linear approximation in dictionaries using LASSO \cite{Tibshirani1996regression}, \cite{Yuan2006model} or orthogonal matching pursuit \cite{Pati1993orthogonal}, \cite{Tropp2005simultaneous}.
We provide an overview of a representative collection of these methods that we shall use as a basis for comparison in Section~\ref{sec:linear_Gaussian}. 

We show that relying on these linear, Gaussian techniques to identify sensors that will be used for nonlinear reconstruction can lead to costly over-sensing when the underlying manifold is low-dimensional, but the data do not lie in an equally low-dimensional subspace.
This effect is most pronounced when the most energetic (highest variance) components of the data are actually functions of less-energetic components, but not vice versa.
In such cases, the linear techniques are consistently ``tricked'' into sensing the most energetic components while failing to capture the important less energetic ones that can actually be used for minimal reconstruction.
These situations are not merely academic, and they actually abound in physics and in data science.
As we shall discuss in Section~\ref{sec:problems_with_linear}, the problem appears in mixing layer fluid flows and in the presence of shock waves, which are both ubiquitous in aerodynamics.
The presence of important low-energy sub-harmonic frequencies is also generic behavior after a period-doubling bifurcation, which is a common route to chaos, for instance in ecosystem collapse \cite{tzuk2019period} and cardiac arrhythmia \cite{Quail2015predicting}.
In data science, the problem is most pronounced when we try to select fundamental nonlinear embedding coordinates for a data set using manifold learning techniques like kernel PCA \cite{Scholkopf1998}, Laplacian eigenmaps \cite{Belkin2003}, diffusion maps \cite{Coifman2006}, and Isomap \cite{Tenenbaum2000global} as we shall discuss in Section~\ref{subsec:manifold_learning}.

In order to address the limitations of linear, Gaussian methods for sensor placement and feature selection demonstrated in the first half of the paper, we develop a novel data-driven approach based on consideration of secant vectors between states in Section~\ref{sec:secants}.
Related secant-based approaches have been pioneered by \cite{Broomhead2005dimensionality}, \cite{Jamshidi2007towards}, \cite{Hegde2015numax}, \cite{Sun2017sparse} for the purpose of finding optimal embedding subspaces.
While their considerations of secants lead to continuous optimization problems over subspaces, our considerations of secants lead to combinatorial optimization problems over sets of sensors.
We develop three different secant-based objectives together with greedy algorithms that each provide different types of robust, near-minimal reconstruction guarantees for very general types of nonlinear inverse problems.
The guarantees stem from the underlying geometric information that is captured by secants and encoded in our optimization objectives.
Moreover, the objectives we consider each have the celebrated diminishing returns property called \emph{submodularity}, allowing us to leverage the classical results of G. L. Nemhauser and L. A. Wolsey et al. \cite{Nemhauser1978analysis}, \cite{Wolsey1982analysis} to guarantee the performance of efficient greedy algorithms for sensor placement.
We also leverage concentration of measure results in order to prove performance guarantees when the secants are randomly down-sampled, enabling computational scalability to very large data sets.
Each of these techniques demonstrates greatly improved performance compared to a large collection of linear techniques on a canonical shock-mixing layer flow problem \cite{Yee1999low} as well as for selecting fundamental manifold learning coordinates.

\section{Background on Linear, Gaussian, Techniques}
\label{sec:linear_Gaussian}

The predominant sensor placement, feature selection, and experimental design techniques available today rely on linear and/or Gaussian assumptions about the underlying data: that is, that the data live in a low-dimensional subspace and/or have a Gaussian distribution.
Under these assumptions, it becomes easy to quantify the performance of sensors, features, or experiments, using a variety of information theoretic, Bayesian, maximum likelihood, or other optimization criteria.
A comprehensive review is beyond the scope of this paper, and of course we do not claim that linear methods always fail.
Rather, we argue that because the underlying linear, Gaussian assumptions are violated in many real-world problems, we cannot expect them to find small collections of sensors that enable nonlinear reconstruction of the desired quantities.
We shall briefly review the collection of linear techniques that we shall compare to throughout this work and that we think are representative of the current literature.

\subsection{(Group) LASSO}
The Least Absolute Shrinkage and Selection Operator (LASSO) method introduced by R. Tibshirani \cite{Tibshirani1996regression} is a highly successful technique for feature selection in machine learning that has found additional applications in compressive sampling recovery \cite{Candes2009near} and system identification \cite{Brunton2016discovering}.
A generalization by M. Yuan and Y. Lin \cite{Yuan2006model} called group LASSO is especially relevant for sensor placement since it allows measurements to be selected in groups that might come from the same sensor at different instants of time.
Suppose we are given a collection of data consisting of available measurements $\vect{m}_j(\xx_i)$, $j=1,\ldots,M$ along with relevant quantities $\vect{g}(\xx_i)$ that we wish to reconstruct over a collection of states $\xx_i$, $i=1,\ldots,N$.  The group LASSO convex optimization problem takes the form
\begin{equation}
    \minimize_{\mat{A}_1, \ldots, \mat{A}_M} \sum_{i=1}^N\Big\Vert \vect{g}(\xx_i) - \sum_{j=1}^M\mat{A}_j\vect{m}_j(\xx_i) \Big\Vert_2^2 + \gamma \sum_{j=1}^M \left\Vert \mat{A}_j \right\Vert_F
\end{equation}
and tries to reconstruct the targets as accurately as possible using a linear combination of the measurements subject to a sparsity-promoting penalty.
The strength of the penalty depends on the user-specified parameter $\gamma \geq 0$ and forces the coefficient matrices $\mat{A}_j$ on many of the measurement groups to be identically zero.
Those coefficient matrices with nonzero entries indicate the sensors that should be used to \emph{linearly} reconstruct the target variables with high accuracy.

\subsection{Determinantal ``D''-Optimal Selection}
Suppose the state $\xx$ has a prior probability distribution with covariance $\mat{C}_{\xx}$ and the target variables $\vect{g}(\xx)$ and measurements $\vect{m}_j(\xx)$, $j=1,\ldots,M$ are linear functions of the state
\begin{equation}
    \vect{g}(\xx) = \mat{T}\xx, \qquad \vect{m}_j(\xx) = \mat{M}_j\xx + \vect{n}_j
\end{equation}
where $\vect{n}_j$ is the mean-zero, state independent, noise from the $j$th sensor with covariance $\mat{C}_{\vect{n}_j}$.
Then, if $\mat{M}_{\scrS}$ is a matrix with rows given by $\mat{M}_j$ and $\mat{C}_{\vect{n}_{\scrS}}$ is a block diagonal matrix formed from $\mat{C}_{\vect{n}_j}$, for $j$ in a given set of sensors $\scrS$, then the optimum (least-squares) linear estimate of $\vect{g}(\xx)$ given $\vect{m}_{\scrS}(\xx)$ has error covariance
\begin{equation}
    \mat{C}_{\vect{e}}(\scrS) = \mat{T}\left( \mat{C}_{\xx}^{-1} + \mat{M}_{\scrS}^T\mat{C}_{\vect{n}_{\scrS}}^{-1}\mat{M}_{\scrS} \right)^{-1}\mat{T}^T.
    \label{eqn:OLE_error_covariance}
\end{equation}
If $\xx$ and the noise are independent Gaussian random variables then Eq.~\textbf{\ref{eqn:OLE_error_covariance}} is the covariance of the posterior distribution for $\vect{g}(\xx)$ given $\vect{m}_{\scrS}(\xx)$.
A low-dimensional representation of the state and its covariance are usually found from data via principal component analysis (PCA) \cite{Hotelling1933analysis} or proper orthogonal decomposition (POD) \cite{berkooz1993proper} when an analytical model is not available.

A common technique, referred to as the Bayesian approach in the optimal design of experiments \cite{Pukelsheim2006optimal} is to quantify performance using functions of $\mat{C}_{\vect{e}}(\scrS)$ \cite{Chaloner1995bayesian}.
In particular, Bayesian determinantal or ``D''-optimality entails minimizing $\log{\det{\mat{C}_{\vect{e}}(\scrS)}}$, which, under Gaussian assumptions, is equivalent to minimizing the conditional entropy \cite{Shewry1987maximum}, \cite{Sebastiani2000maximum} or the volumes of confidenece ellipsoids about the maximum a posteriori (MAP) estimate of $\vect{g}(\xx)$ given $\vect{m}_{\scrS}(\xx)$ \cite{Joshi2008sensor}.
This approach is widely used for sensor placement since it readily admits efficient approximations based on convex relaxation \cite{Joshi2008sensor} and greedy algorithms \cite{Shamaiah2010greedy}, \cite{tzoumas2016sensor} with guaranteed performance.
Similar objectives have been used to quantify observability and controllability for sensor and actuator placement in linear dynamical systems \cite{Summers2015submodularity}, \cite{Summers2016submodularity}.

When there is no prior probability distribution for $\xx$ and we want to estimate the full state $\vect{g}(\xx) = \xx$ from measurements corrupted by Gaussian noise, we can construct the maximum likelihood estimate whose error covariance is
\begin{equation}
    \mat{C}_{\vect{e}}(\scrS) = \left( \mat{M}_{\scrS}^T\mat{C}_{\vect{n}_{\scrS}}^{-1}\mat{M}_{\scrS} \right)^{-1}.
\end{equation}
Minimizing the volumes of confidence ellipsoids in this setting as is done in \cite{Joshi2008sensor} is referred to as maximum likelihood ``D''-optimality since it entails maximizing $\log\det{\left(\mat{M}_{\scrS}^T\mat{C}_{\vect{n}_{\scrS}}^{-1}\mat{M}_{\scrS}\right)}$.
In the absence of the regularizing effect the prior distribution has on the estimate, we must have at least as many sensor measurements as state variables in the maximum likelihood setting.

\subsection{Pivoted QR Factorization}
Pivoted matrix factorization techniques, and QR pivoting in particular, have become a popular choice for sensor placement \cite{Manohar2018data}, \cite{Brunton2019data} and feature selection in reduced-order modeling \cite{Chaturantabut2010nonlinear}, \cite{Drmac2016new}, where the method is often referred to as the Discrete Empirical Interpolation Method (DEIM).
This approach dates back to P. Businger and G. H. Golub's seminal work \cite{Businger1965linear}, which introduced Householder-pivoted QR factorization for the purpose of feature selection in least squares fitting problems.
The approach is also closely related to orthogonal matching pursuit \cite{Pati1993orthogonal} and simultaneous orthogonal matching pursuit \cite{Tropp2005simultaneous}, which are widely used sparse approximation algorithms.

In its simplest form, one supposes that the underlying state to be estimated $\vect{g}(\xx) = \xx$ is low dimensional (e.g., using its PCA or POD coordinate representation) and selects the linear measurements from among the rows of a matrix $\mat{M}$ by forming a pivoted QR decomposition of the form
\begin{equation}
    \mat{M}^T\begin{bmatrix}[c|c] \mat{P}_1 & \mat{P}_2\end{bmatrix} = \mat{Q} \begin{bmatrix}[c|c] \mat{R}_1 & \mat{R}_2\end{bmatrix},
\end{equation}
where $\begin{bmatrix}[c|c] \mat{P}_1 & \mat{P}_2\end{bmatrix}$ is a permutation.
The first $K = \dim\xx$ pivot columns forming $\mat{P}_1$ determine a set of sensor measurements $\vect{m}_{\scrS}(\xx) = \mat{M}_{\scrS}\xx = \mat{P}_1^T\mat{M}\xx$ from which $\xx$ can be robustly recovered as
\begin{equation}
    \xx = \left( \mat{P}_1^T\mat{M} \right)^{-1} \vect{m}_{\scrS}(\xx) = \mat{Q}\left(\mat{R}_1^T\right)^{-1}\vect{m}_{\scrS}(\xx).
\end{equation}
This approach is successful because at each step of the QR pivoting process, the measurement that maximizes the corresponding diagonal entry of the upper triangular matrix $\mat{R}_1$ is selected.
The resulting large diagonal entries of $\mat{R}_1$ mean that measurement errors are not strongly amplified by the linear reconstruction map $\mat{Q}\left(\mat{R}_1^T\right)^{-1}$.

\section{Problems with Linear Techniques}
\label{sec:problems_with_linear}

In this section, we illustrate the problems with employing linear state reconstruction and sensor placement techniques for nonlinear systems and data sets by means of an example.
We consider the shock-mixing layer interaction proposed by Yee et al. \cite{Yee1999low}, which has become a canonical problem for studying jet noise production as well as high-order numerical methods.
This problem captures many key elements of shock wave-turbulent boundary layer interactions that, according to S. Priebe and M. P. Mart\'{i}n \cite{PriebeMartin2012low} ``occur in many external and internal compressible flow applications such as transonic  aerofoils, high-speed engine inlets, internal flowpaths of scramjets, over-expanded rocket engine nozzles and deflected control surfaces or any other discontinuities in the surface geometry of high-speed vehicles.''
The resulting pressure and heat transfer fluctuations can be large, so it is important to monitor the state of these flows to ensure the safety of a vehicle.

Our goal will be to choose a small number of sensor locations in this flow at which to measure either the horizontal, $u$, or vertical, $v$, velocity component in order to reconstruct the entire velocity field.
A snapshot of these velocity fields from the fully-developed flow computed using the high-fidelity local WENO-type characteristic filtering method of S.-C. Lo et al. \cite{Lo2010high} is shown in Fig.~\ref{fig:shock_mixing_layer_snapshots}.
While the flow is very nearly periodic, and hence lives near a one-dimensional loop in state space, the complicated physics arising from the interaction of the oblique shock with vortices in the spatially-evolving mixing layer results in data that do not lie near any low-dimensional subspace.
In addition to being high dimensional, this flow exhibits the low-frequency unsteadiness characteristic of shock wave--turbulent boundary layer interactions \cite{PriebeMartin2012low}, \cite{Clemens2014low}, \cite{PriebeTu2016Low} and of spatial mixing layer flows in general \cite{Ho1982subharmonics}.

\begin{figure}
    \centering
    \subfloat[stream-wise $u$ velocity component]{
    \begin{tikzonimage}[trim=26.75 120 13 135, clip=true, width=0.9\textwidth]{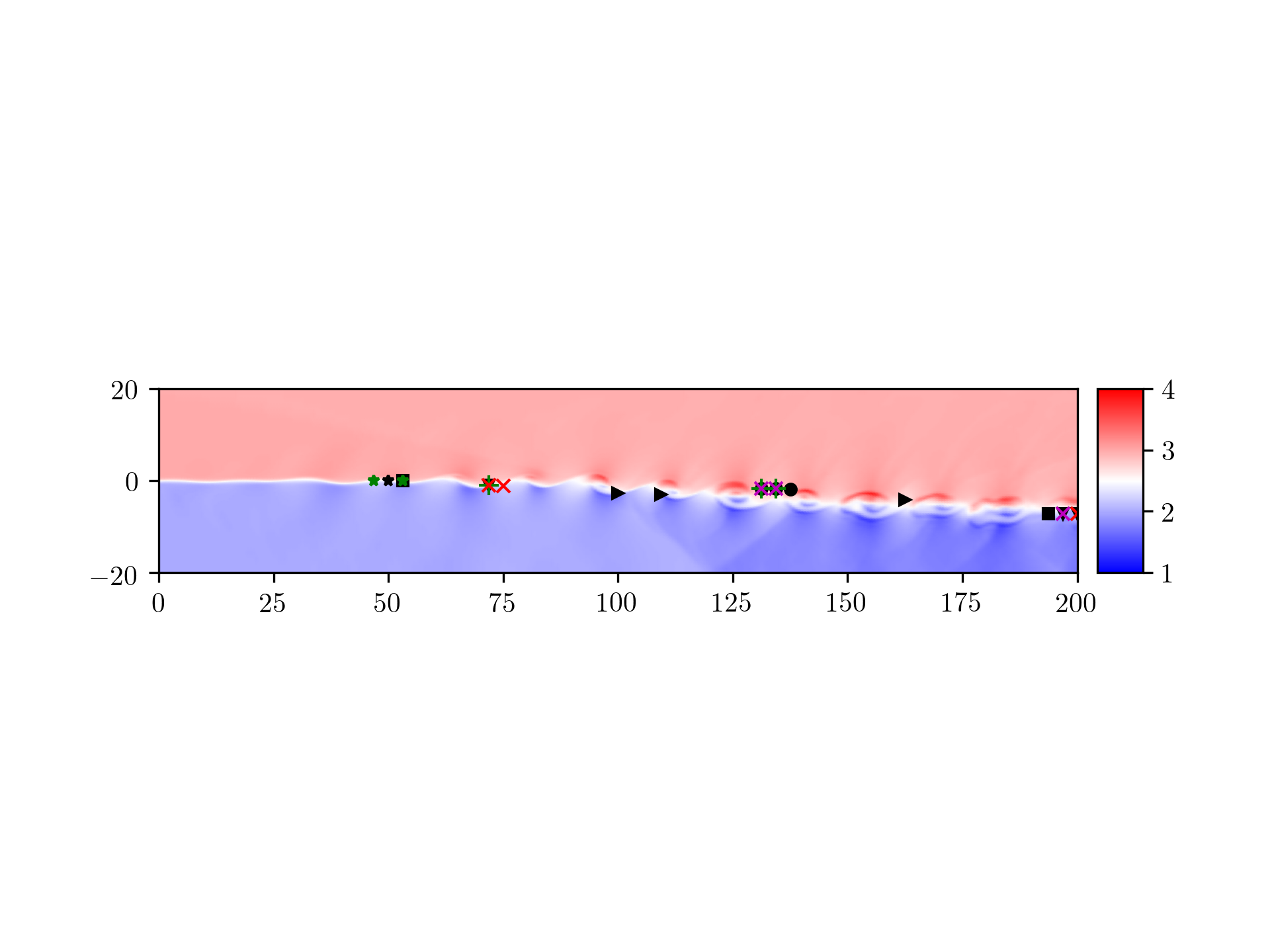}
    \end{tikzonimage}
    \label{subfig:u_snapshot}
    } \\
    \subfloat[transverse $v$ velocity component]{
    \begin{tikzonimage}[trim=30 120 10 135, clip=true, width=0.9\textwidth]{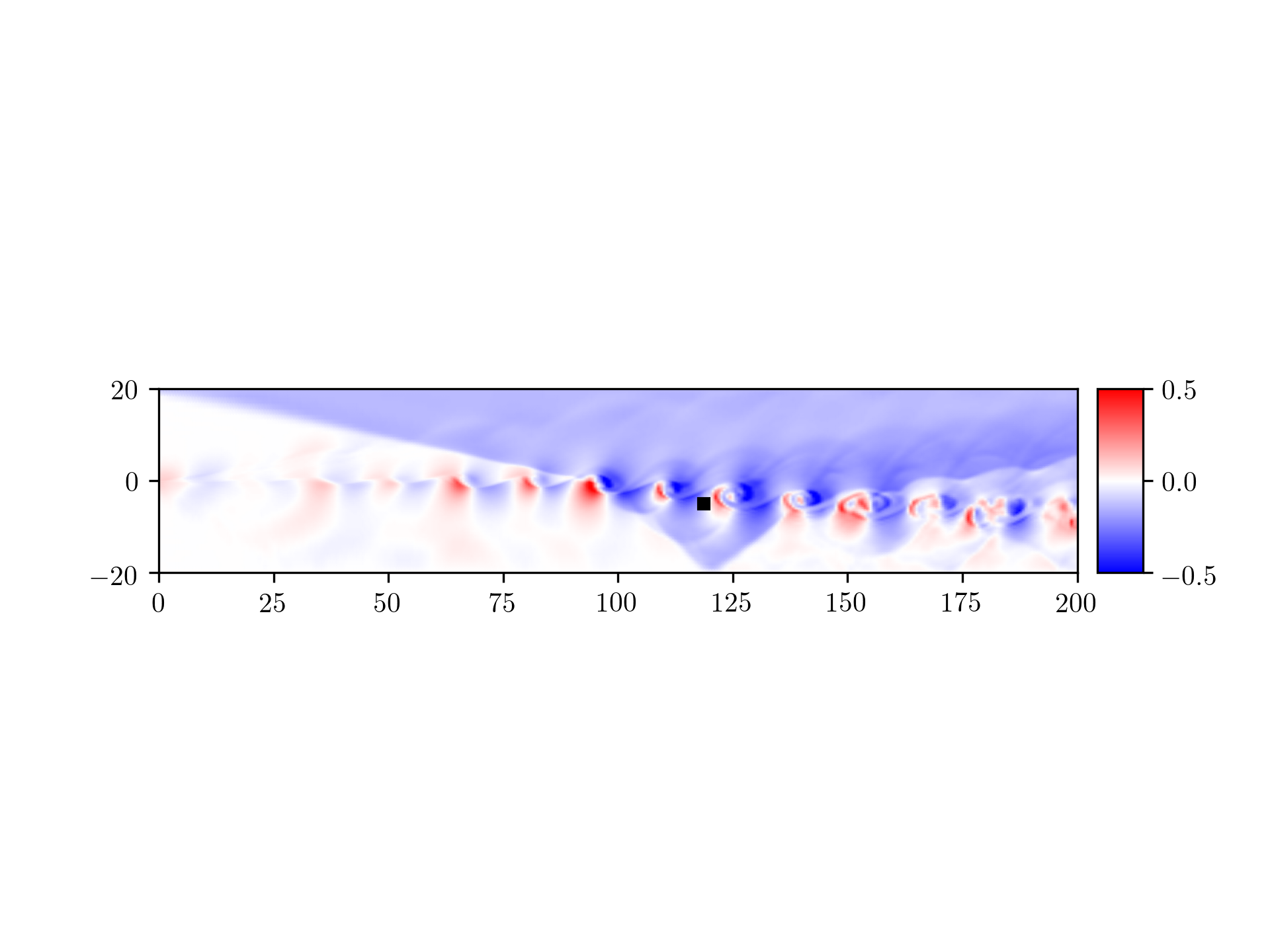}
    \end{tikzonimage}
    \label{subfig:v_snapshot}
    } \\
    \subfloat[available sensor locations]{
    \begin{tikzonimage}[trim=30 120 10 135, clip=true, width=0.9\textwidth]{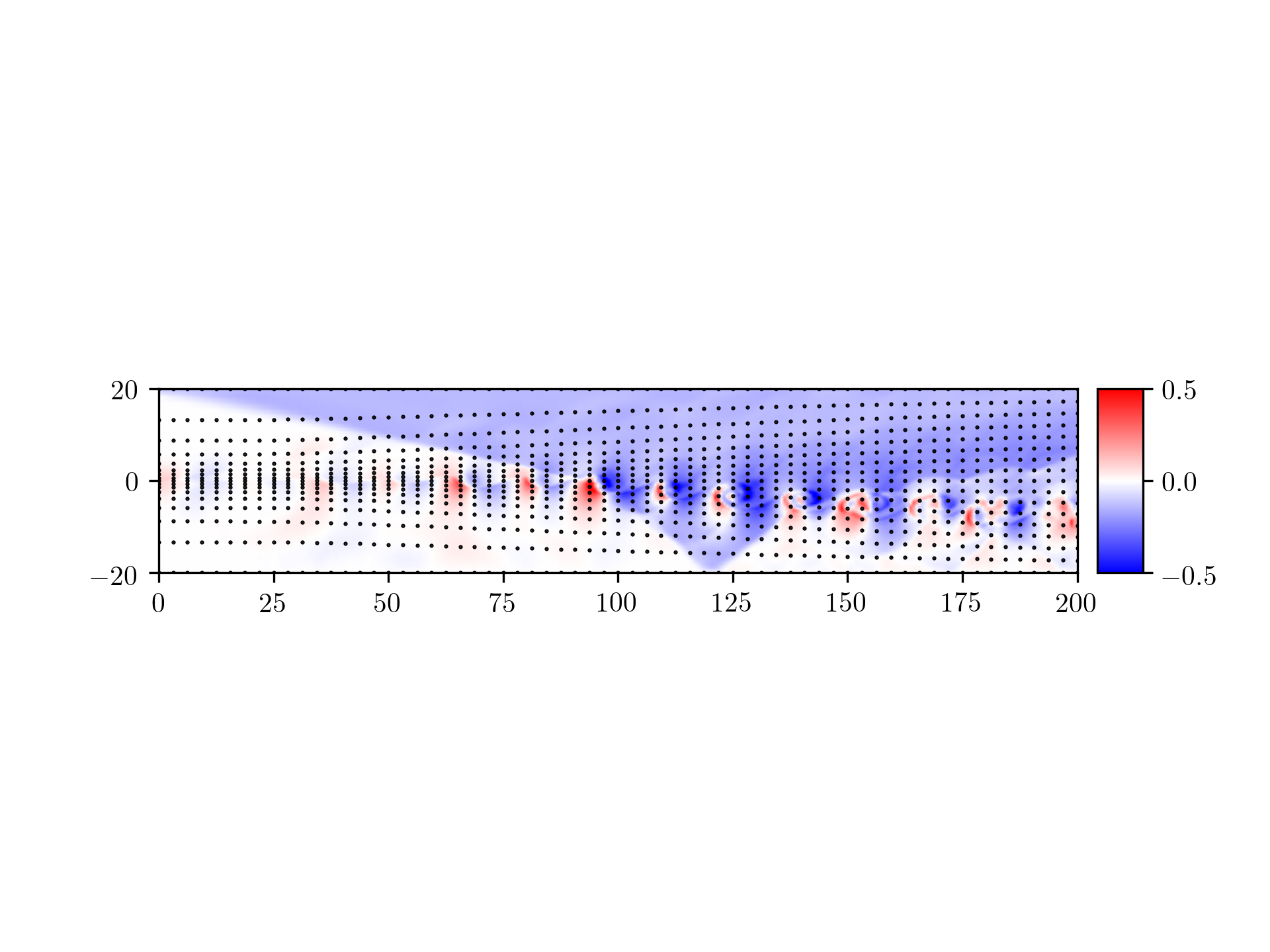}
    \end{tikzonimage}
    \label{subfig:available_sensors}
    }
    \caption{A snapshot of the $u$ and $v$ velocity components in the shock mixing-layer flow is shown in (a) and (b) along with the sensors selected using various methods 
    from among the two components at $1105$ available locations shown in (c). 
    These methods include LASSO with PCA (black o), LASSO with Isomap (red x) greedy Bayes D-optimality (magenta x), convex Bayes D-optimality (black $>$), convex D-optimality for modes $3$ and $4$ (black v), QR pivoting (green +), and secant-based techniques using detectable differences ($\# 1,\# 2$: green star, $\# 3$: black star) and the amplification threshold method (black square).}
    \label{fig:shock_mixing_layer_snapshots}
\end{figure}

\subsection{The Need for Nonlinear Reconstruction}
Linear reconstruction is fundamentally confined to a subspace whose dimension is at most equal to the total number of sensor measurements.
Hence the fraction of the variance that linear reconstruction can capture using $d$ measurements is at most the fraction of the variance along the first $d$ principal components: in particular, the coefficient of determination is bounded by
\begin{equation}
    R^2 \leq \frac{\sigma_1^2 + \cdots + \sigma_d^2}{\sigma_1^2 + \cdots + \sigma_n^2}.
    \label{eqn:linear_reconstruction_bound}
\end{equation}
Examining the fraction of the variance captured by the leading principal subspaces in Figure~\ref{subfig:shock_mixing_variance} leads us to the rather disappointing conclusion that in order to capture $90\%$ of the variance in the shock-mixing layer flow via linear reconstruction, we need at least $11$ independent measurements, and to capture $98\%$ we need at least $33$.
\begin{figure*}
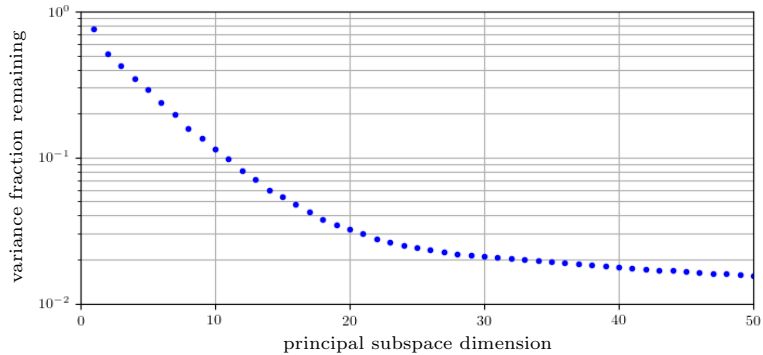
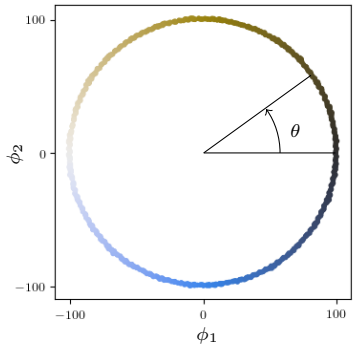
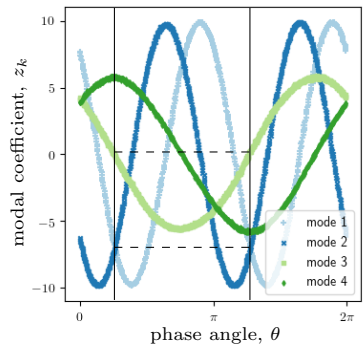

    \centering
    \subfloat[variance orthogonal to principal subspaces]{
    \begin{tikzonimage}[trim=30 5 40 20, clip=true, width=0.8\textwidth]{Figures/ShockMixingLayer/PrincipalVariances.PNG}
    \node[rotate=90] at (0.0,0.53) {\scriptsize variance fraction remaining};
    \node[rotate=0] at (0.525,-0.02) {\scriptsize principal subspace dimension};
    \end{tikzonimage}
    \label{subfig:shock_mixing_variance}
    } \\ 
    \subfloat[Isomap coordinates]{
    \begin{tikzonimage}[trim=0 0 0 0, clip=true, width=0.4\textwidth]{Figures/ShockMixingLayer/isomaps_data.PNG}
    \draw[-] (0.515,0.4925) -- (0.86,0.4925);
    \draw[-] (0.515,0.4925) -- (0.796,0.696);
    \draw[->] (0.715,0.4925) arc[radius=0.2, start angle=0, end angle=35];
    \node[rotate=0] at (0.755,0.55) {\scriptsize $\theta$};
    \node[rotate=0] at (0.525,0.015) {\scriptsize $\phi_1$};
    \node[rotate=90] at (0.02,0.49) {\scriptsize $\phi_2$};
    \end{tikzonimage}
    \label{subfig:shock_mixing_isomaps}
    }
    \subfloat[PCA coefficients]{
    \begin{tikzonimage}[trim=0 0 0 0, clip=true, width=0.4\textwidth]{Figures/ShockMixingLayer/POD_coeffs_vs_phase.PNG}
    \node[rotate=0] at (0.52,0.02) {\scriptsize phase angle, $\theta$};
    \node[rotate=90] at (0.005,0.5) {\scriptsize modal coefficient, $z_k$};
    \draw[-] (0.255,0.112) -- (0.255,0.878);
    \draw[-] (0.61,0.112) -- (0.61,0.878);
    \draw[-,dashed] (0.255,0.5) -- (0.61,0.5);
    \draw[-,dashed] (0.255,0.25) -- (0.61,0.25);
    \end{tikzonimage}
    \label{subfig:shock_mixing_projected}
    }
    \caption{The linear and nonlinear dimension reduction techniques PCA (a.k.a POD) and Isomap are applied to the shock-mixing layer data. (a) shows the remaining fraction of the total variance orthogonal to each leading principal subspace. (b) plots the data in the leading two Isomap embedding coordinates, revealing that it lies very near a loop in state space. (c) shows how the leading principal components (modal coefficients) vary with the phase angle around the loop. The black vertical lines reveal distinct points where the leading three principal components are identical.}
    \label{fig:shock_mixing_POD}
\end{figure*}

The best possible linear reconstruction performance can be arbitrarily poor even though the underlying manifold is low-dimensional.
We illustrate this fact with the following toy model that resembles the phase dependence of principal components in the shock-mixing layer problem shown in Figures~\ref{subfig:shock_mixing_isomaps}~and~\ref{subfig:shock_mixing_projected}.
Let $\theta$ be uniformly distributed over the interval $[0, 2\pi]$ and let the components of the state vector have sinusoidal dependence on the phase given by
\begin{equation}
    x_{2k-1} = \sqrt{2}\cos(k\theta),\ \ x_{2k} = \sqrt{2}\sin(k\theta),\ \ k=1,\ldots, n/2.
    \label{eqn:toy_model_equal}
\end{equation}
Since these components are orthonormal functions of $\theta$ with respect to the uniform probability measure on $[0, 2\pi]$, the state vector has isotropic covariance $\mathbb{E} \xx \xx^T = \mat{I}_n$ and the fraction of the variance captured by the leading $d$ principal components is $d/n$.
As the dimension increases, the highest possible coefficient of determination for linear reconstruction approaches zero since $R^2 \leq d/n \to 0$ as $n\to\infty$.
Meanwhile, it's obvious that the state vector can be perfectly reconstructed as a nonlinear function of $x_1$ and $x_2$ alone.

Indeed, it is possible to reconstruct the entire shock-mixing layer flow-field as a nonlinear function of the velocity measurements at two carefully chosen locations.
In particular, the measurements made at the locations marked by the two green stars in Figure~\ref{fig:shock_mixing_layer_snapshots} are one-to-one with the phase and hence the state of the flow.
This is seen in Figure~\ref{subfig:secant_MSE_data2}, where the phase angle (color) --- hence the full state --- can be determined uniquely from the values of the measurements.
Meanwhile, the best possible linear reconstruction performance using two measurements is $R^2 < 0.5$.

In practice, many nonlinear reconstruction techniques are available including neural networks \cite{Nair2019integrating}, Gaussian process regression \cite{Rasmussen2003gaussian}, and recurrent neural networks for time-delayed measurements \cite{Krishnan2017structured}.
Using Gaussian process regression and the two sensor locations marked by green stars in Figure~\ref{fig:shock_mixing_layer_snapshots}, we obtain near-perfect, robust reconstruction of the leading $100$ principal components.
The resulting reconstruction accuracy for the flow-fields on a held-out set of $250$ snapshots is $R^2=0.986$.

\subsection{The Need for Nonlinear Sensor Placement}
\label{subsec:need_for_nonlinear_placement}
With such poor reconstruction afforded by linear techniques, we cannot expect sensor placement methods based on them to perform any better.
This is not to say that a practitioner won't ever find lucky sensor locations for nonlinear reconstruction by employing a sensor placement technique that maximizes linear reconstruction accuracy.
However, this kind of luck is not guaranteed as illustrated when we apply state of the art linear sensor placement techniques to the shock mixing-layer problem.
Indeed Figures~\ref{subfig:LASSO_data}, \ref{subfig:LASSO_Iso_data}, \ref{subfig:greedy_Bayes_D_opt_data}, \ref{subfig:convex_Bayes_D_opt_data}, and \ref{subfig:pivoted_QR_data} provide visual proof that three sensors chosen using LASSO to reconstruct the leading $100$ principal components, LASSO to reconstruct the leading two Isomap coordinates, the greedy Bayes D-optimality approach, the convex Bayes D-optimality approach, and pivoted QR factorization do not produce measurements that are one-to-one with the state.
Implementation details can be found in Appendix~\ref{app:implementation_details}.
In each case, there are at least two distinct states with different phases on the orbit (color) for which the sensors measure the same values and hence cannot be used to tell them apart.

Even measuring the leading three principal components directly, which are optimal for linear reconstruction, cannot always reveal the state of the shock-mixing layer flow.
The black vertical lines in Figure~\ref{subfig:shock_mixing_projected} indicate the phases of two distinct states for which the leading three principal components agree, yet the fourth differs.
One may wonder whether the fact that the third and fourth principal components are one-to-one with the state can be leveraged for sensor placement.
Even our attempt to place three maximum likelihood D-optimal sensors using the convex optimization approach of \cite{Joshi2008sensor} to reconstruct the third and fourth principal components fails to produce measurement that can recover the phase of the flow as seen in Figure~\ref{subfig:convex_D_opt_34_data}.

On the other hand, it is possible to find two sensor locations whose measurements are one-to-one with the state as shown in Figure~\ref{subfig:secant_MSE_data2}.
The resulting curve near which the data lie has a kink in the lower-right region indicating that while the measurements are one-to-one, the time-derivative of the state cannot be determined at this point.
Capturing time-derivatives is necessary for building reduced-order models, and this can be accomplished using the three sensors marked by black squares in Figure~\ref{fig:shock_mixing_layer_snapshots} and whose measurements are plotted in Figure~\ref{subfig:secant_Lip_data}.
We note, however, that these locations are far apart in space, and so will be more sensitive to perturbations of the shear-layer thickness which affects the horizontal spacing of vortices.

\begin{figure*}
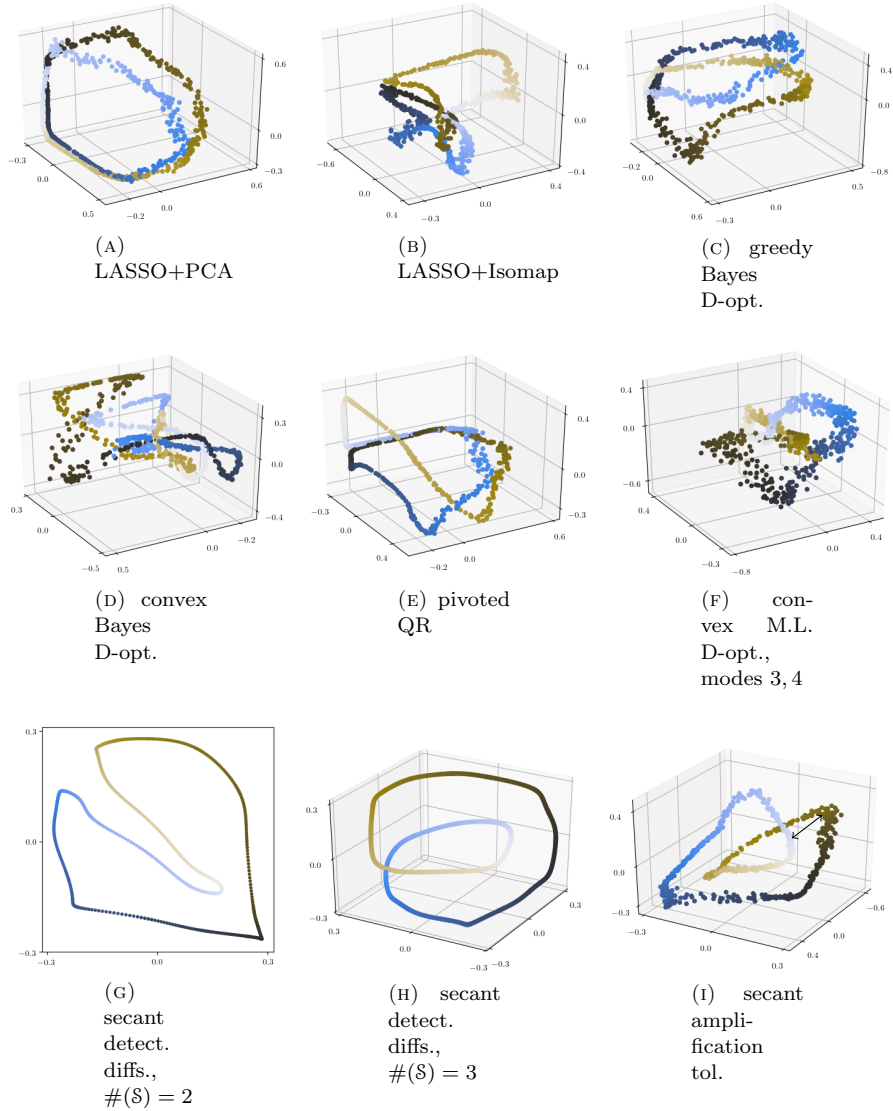

    \centering
    \subfloat[LASSO+PCA]{
    \begin{tikzonimage}[trim=90 30 40 50, clip=true, width=0.3\textwidth]{Figures/ShockMixingLayer/group_LASSO_data.PNG}
    \end{tikzonimage}
    \label{subfig:LASSO_data}
    }
    \subfloat[LASSO+Isomap]{
    \begin{tikzonimage}[trim=90 30 40 50, clip=true, width=0.3\textwidth]{Figures/ShockMixingLayer/group_LASSO_Iso_data.PNG}
    \end{tikzonimage}
    \label{subfig:LASSO_Iso_data}
    }
    \subfloat[greedy Bayes D-opt.]{
    \begin{tikzonimage}[trim=90 30 40 50, clip=true, width=0.3\textwidth]{Figures/ShockMixingLayer/BayesD_greedy_data.PNG}
    \end{tikzonimage}
    \label{subfig:greedy_Bayes_D_opt_data}
    } \\
    \subfloat[convex Bayes D-opt.]{
    \begin{tikzonimage}[trim=90 30 40 50, clip=true, width=0.3\textwidth]{Figures/ShockMixingLayer/BayesD_cvx_data.PNG}
    \end{tikzonimage}
    \label{subfig:convex_Bayes_D_opt_data}
    }
    \subfloat[pivoted QR]{
    \begin{tikzonimage}[trim=90 30 40 50, clip=true, width=0.3\textwidth]{Figures/ShockMixingLayer/pivoted_QR_data.PNG}
    \end{tikzonimage}
    \label{subfig:pivoted_QR_data}
    }
    \subfloat[convex M.L. D-opt., modes $3,4$]{
    \begin{tikzonimage}[trim=70 30 65 50, clip=true, width=0.3\textwidth]{Figures/ShockMixingLayer/D_cvx_modes34_data.PNG}
    \end{tikzonimage}
    \label{subfig:convex_D_opt_34_data}
    } \\
    \subfloat[secant detect. diffs., $\#(\scrS)=2$]{
    \begin{tikzonimage}[trim=70 15 75 35, clip=true, width=0.28\textwidth]{Figures/ShockMixingLayer/MDD_greedy2_data.PNG}
    \end{tikzonimage}
    \label{subfig:secant_MSE_data2}
    }
    \subfloat[secant detect. diffs., $\#(\scrS)=3$]{
    \begin{tikzonimage}[trim=65 30 65 50, clip=true, width=0.3\textwidth]{Figures/ShockMixingLayer/MDD_greedy3_data.PNG}
    \end{tikzonimage}
    \label{subfig:secant_MSE_data3}
    }
    \subfloat[secant amplification tol.]{
    \begin{tikzonimage}[trim=65 30 65 50, clip=true, width=0.3\textwidth]{Figures/ShockMixingLayer/secant_Lip_data.PNG}
        \draw [<->] (0.655, 0.58) -- (0.76, 0.68);
    \end{tikzonimage}
    \label{subfig:secant_Lip_data}
    }
    \caption{these plots show the measurements made by sensors selected using various methods on the shock-mixing layer flow problem. Each dot indicates the values measured by the sensors and its color indicates the phase of the corresponding flowfield. The sensors selected using linear methods shown in (a)-(f) each make identical or nearly identical measurements on distinct flowfields, indicated by overlapping points with different colors. These sensors cannot tell those flowfields apart since the measurements are the same. The sensors selected using secant-based methods shown in (g)-(i) make distinct measurements for distinct states and have no such overlaps.}
    \label{fig:linear_methods_for_shock_mixing}
\end{figure*}

The linear techniques, LASSO, greedy and convex Bayesian D-optimal selection, pivoted QR, and even direct measurement of principal components fail to reveal the minimum number of sensors needed to reconstruct the state because there is important information about the flow contained in less-energetic principal components.
In particular, Figure~\ref{subfig:shock_mixing_projected} shows that the most energetic two principal components oscillate with twice the frequency of the third and fourth most energetic components as one moves around the orbit.
In trying to maximize the variance captured by a linear estimator, the linear sensor placement techniques are doomed to choose sensors whose measurements return to the same values twice in one period as in Figures~\ref{subfig:LASSO_data},~\ref{subfig:greedy_Bayes_D_opt_data},~and~\ref{subfig:pivoted_QR_data}.
In addition, the convex Bayesian D-optimal approach finds sensors that achieve a superior value of the objective $\log\det{\mat{C}_{\vect{e}}(\scrS)}$ than the greedy Bayesian D-optimal approach, yet the resulting measurements in Figure~\ref{subfig:convex_Bayes_D_opt_data} have many more self-intersections than the greedy method in Figure~\ref{subfig:greedy_Bayes_D_opt_data}.

We are forced to conclude that sensor placement based on linear reconstruction is totally unconnected with nonlinear reconstructability when the underlying manifold and principal dimensions do not agree.
This can be seen most clearly from the fact that by simply re-scaling each coordinate in the toy model Eq.~\textbf{\ref{eqn:toy_model_equal}} by positive constants $\alpha_1, \ldots, \alpha_n$, we can trick these techniques into selecting any given collection of coordinates.
Under this scaling, the covariance matrix becomes $\text{diag}(\alpha_1^2, \ldots, \alpha_n^2)$ and if we sort the constants in decreasing order $\alpha_{k_1} \geq \alpha_{k_2} \geq \cdots$ then the variance captured by a linear reconstruction from $d$ measurements cannot exceed
\begin{equation}
    R^2 \leq \frac{\alpha_{k_1}^2 + \cdots + \alpha_{k_d}^2}{\alpha_1^2 + \cdots + \alpha_{n}^2},
\end{equation}
according to the bound in Eq.~\textbf{\ref{eqn:linear_reconstruction_bound}}.
Equality is achieved by the optimal linear estimator based on measured coordinates $x_{k_1}, \ldots, x_{k_d}$.
Meanwhile, the only pair of coordinates needed for nonlinear reconstruction are $x_1$ and $x_2$.

The key point is that sensor placement approaches based on linear reconstruction tend to pick sensor locations that have high variance over other choices that can be more informative.
The linear approach works well when a small number of principal components contain essentially all of the variance or when all higher modal components are very nearly determined by the lower ones.
But as we have shown, linear approaches to sensor placement can fail catastrophically when genuinely informative fluctuations, e.g. sub-harmonics, produce significant variance orthogonal to the leading principal subspace.
In order to reveal minimal sensor locations that can be used for nonlinear reconstruction in such situations, we cannot rely on linear reconstruction performance as an optimization criteria, and an entirely new approach is needed.
In Section~\textbf{\ref{sec:secants}} we discuss an approach that can recover the correct coordinates from which all others can be nonlinearly reconstructed.

\subsection{Selecting Manifold Learning Coordinates}
\label{subsec:manifold_learning}
The examples presented in the previous Section~\ref{subsec:need_for_nonlinear_placement} involved data lying near a one-dimensional underlying manifold.
Essentially the same problems can occur for data lying near higher-dimensional manifolds, and an especially illustrative and practically useful application where this situation is routinely encountered is manifold learning.
In general, manifold learning seeks to find a small collection of nonlinear coordinates that fully describe the structure of a dataset, i.e., that embed it in a lower-dimensional space.
Many techniques including kernel PCA \cite{Scholkopf1998}, Laplacian eigenmaps \cite{Belkin2003}, diffusion maps \cite{Coifman2006}, and Isomap \cite{Tenenbaum2000global} accomplish this via eigen-decomposition of various symmetric matrices 
\begin{equation}
    \mathbf{G} = \boldsymbol{\Phi}\boldsymbol{\Lambda}^2\boldsymbol{\Phi}^T, \qquad 
    \boldsymbol{\Phi} = \begin{bmatrix} \boldsymbol{\phi}_1 & \cdots & \boldsymbol{\phi}_r \end{bmatrix}
\end{equation}
derived from pair-wise similarity among data points.
The $k$th eigen-coordinate of each point in the data set is given by the elements of $\boldsymbol{\phi}_k$, which can be viewed as a discrete approximation of an eigenfunction of some kernel integral operator on the underlying manifold.
These methods suffer from a well-known issue when the dataset has multiple length scales: namely, there may be several redundant harmonically related eigen-coordinates with higher salience (determined by the eigenvalues) before one encounters a new fundamental eigen-coordinate describing a new set of features.
This makes the search for a fundamental set of eigen-coordinates that embed the underlying manifold a potentially large combinatorial search problem.

As a concrete example, consider the Isomap eigen-coordinates shown in Figure~\ref{fig:torus} computed from 2000 points lying on the torus in $\mathbb{R}^3$,
\begin{equation}
    \mathbf{x} = 
    \left(
        \left(5 + \cos{\theta_2}\right)\cos{\theta_1}, \
        \left(5 + \cos{\theta_2}\right)\sin{\theta_1}, \
        \sin{\theta_2}
    \right),
    \label{eqn:torus}
\end{equation}
with $(\theta_1, \theta_2)$ drawn uniformly at random from the square $[0,2\pi]\times[0,2\pi]$.
Toroidal dynamics are known to occur in combustion instabilities where multiple incommensurate frequencies are observed \cite{Dunstan2001fitting}, \cite{Lamraoui2011experimental}, producing data that winds around a torus in high-dimensional state space.
One may want to build simplified reduced-order models of these dynamics by finding a small set of nonlinear coordinates that described the state on the torus using manifold learning.

Considering the torus in Eq.~\textbf{\ref{eqn:torus}}, the underlying kernel integral operators associated with each manifold learning technique mentioned above are equivariant with respect to rotations about $\theta_1$, meaning that among their eigenfunctions are always those of the symmetry's generator, namely $\phi_k(\xx) = e^{i k \theta_1(\xx)}$.
Unsurprisingly, the leading six Isomap eigen-coordinates, ranked by their associated eigenvalues, are all harmonically related modes resembling the real and imaginary parts of $e^{i k \theta_1}$, which provide redundant information about $\theta_1$ and no information about $\theta_2$.
The coordinate $\theta_1$ corresponds to larger spatial variations among points and it is not until we encounter the seventh eigen-coordinate that we learn about the smaller variations associated with $\theta_2$.
A na\"{\i}ve user of Isomap might plot the data in the leading three coordinates and falsely conclude that the data lies on a two-dimensional gasket.
We'd like to provide an efficient method for selecting the fundamental eigen-coordinates $\phi_1$, $\phi_2$ and $\phi_7$, from which all others can be (nonlinearly) reconstructed; yet again, linear methods fundamentally cannot be used to select them.
\begin{figure*}
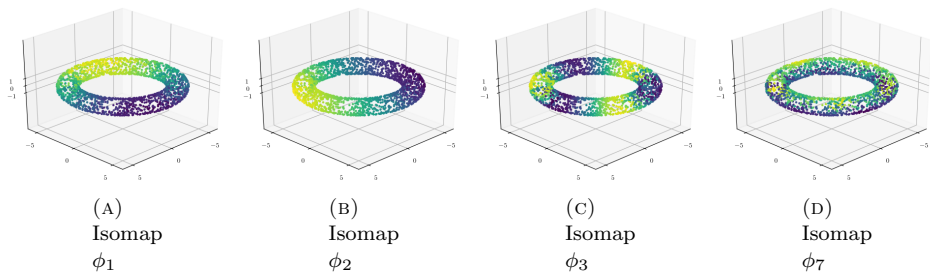

    \centering
    \subfloat[Isomap $\phi_1$]{
    \begin{tikzonimage}[trim=70 25 70 50, clip=true, width=0.23\textwidth]{Figures/torus/isomaps_coordinate_0.PNG}
    \end{tikzonimage}
    \label{subfig:torus_mode1}
    }
    \subfloat[Isomap $\phi_2$]{
    \begin{tikzonimage}[trim=70 25 70 50, clip=true, width=0.23\textwidth]{Figures/torus/isomaps_coordinate_1.PNG}
    \end{tikzonimage}
    \label{subfig:torus_mode2}
    }
    \subfloat[Isomap $\phi_3$]{
    \begin{tikzonimage}[trim=70 25 70 50, clip=true, width=0.23\textwidth]{Figures/torus/isomaps_coordinate_2.PNG}
    \end{tikzonimage}
    \label{subfig:torus_mode3}
    }
    \subfloat[Isomap $\phi_7$]{
    \begin{tikzonimage}[trim=70 25 70 50, clip=true, width=0.23\textwidth]{Figures/torus/isomaps_coordinate_6.PNG}
    \end{tikzonimage}
    \label{subfig:torus_mode7}
    }
    \caption{Isomap coordiantes computed from $2000$ randomly sampled points on the torus defined by Eq.~\textbf{\ref{eqn:torus}}. The leading six coordinates resemble the real and imaginary components of $e^{i k \theta_1}$, $k=1,2,3$, due to the rotational symmetry, providing redundant information about $\theta_1$ and no information about $\theta_2$. The fundamental coordinates $\phi_1$, $\phi_2$, and $\phi_7$ provide an embedding of the data that captures its toroidal structure.}
    \label{fig:torus}
\end{figure*}

Linear methods cannot be used to select manifold learning eigen-coordinates for essentially the same reason why they failed on the toy models in Section~\ref{subsec:need_for_nonlinear_placement}: the coordinates are all mutually orthogonal as functions supported on the data!
In particular, the covariance among the eigen-coordinates over the data is isotropic $\mathbb{E}[\phi_i(\mathbf{x}) \phi_j(\mathbf{x})] = \frac{1}{m}\boldsymbol{\phi}_i^T\boldsymbol{\phi}_j = \frac{1}{m}\delta_{i,j}$ and so all sub-collections of a given size capture the same fraction of the total eigen-coordinate variance.
The methods presented in the following Section~\ref{sec:secants} remedy this issue and are capable of selecting the correct set of fundamental eigen-coordinates on the torus example in Eq.~\textbf{\ref{eqn:torus}}.

\section{Greedy Algorithms using Secants}
\label{sec:secants}
With the failure of techniques based on linear reconstruction to select minimal collections of sensors for nonlinear reconstruction, we propose an alternative approach that relies on a collection of ``secant'' vectors between distinct data points.
In this section, we develop this approach, yielding three related greedy selection techniques with classical theoretical guarantees on their performance.
We also discuss some theoretical results that provide deterministic performance guarantees for the sensors selected by our algorithms on unseen data drawn from an underlying set.

We consider a very general type of sensor placement problem that can be stated as follows.
Let the set $\calX\subset \mathbb{R}^n$ represent the possible states of the system and suppose that we are interested in some relevant information about the state described by a function $\vect{g}:\calX\to\mathbb{R}^q$.
The sensors are also described as functions of the state $\vect{m}_j:\calX\to\mathbb{R}^{d_j}$, $j=1,\ldots,M$ where, with a slight abuse of notation, we will denote the set of all sensors and the set of all sensor indices by $\scrM$ interchangeably.
Our goal is to choose a small subset of sensors $\scrS=\lbrace j_1,\ldots,j_K\rbrace\subseteq\scrM$ so that the relevant information $\vect{g}(\xx)$ about any state $\xx\in\calX$ can be recovered from the combined measurements we have selected
\begin{equation}
    \mS(\xx) = \left( \vect{m}_{j_1}(\xx),\ldots, \vect{m}_{j_K}(\xx) \right)\in\mathbb{R}^{d_{\scrS}},
\end{equation}
where the measurement dimension is $d_{\scrS} = \sum_{j\in\scrS}d_j$.
That is, we want to choose $\scrS$ in such a way that there exists a reconstruction function $\PhiS:\mathbb{R}^{d_{\scrS}}\to\mathbb{R}^q$ so that
\begin{equation}
    \vect{g}(\xx) = \PhiS\left(\mS(\xx)\right)
\end{equation}
for every $\xx\in\calX$.

For such a reconstruction function $\PhiS$ to exist, we must meet the modest condition that any two states $\xx, \xx'\in\calX$ with different target values $\vect{g}(\xx) \neq \vect{g}(\xx')$ produce different measured values $\mS(\xx)\neq \mS(\xx')$.
This is nothing but the vertical line test for $\PhiS$, ensuring that it is a true function that does not take multiple values.
However, this condition may be met for a variety of different choices of measurements $\scrS$ and we shall introduce three different ways to quantify their performance and choose among them.
In these methods, the notion that $\PhiS$ should not be sensitive to perturbations of the measurements is key in quantifying the performance of the sensors.
The techniques we propose each rely on secants, defined below, to measure the sensitivity of $\PhiS$.
\begin{definition}[Secant]
A \emph{secant} is a pair of states $(\vect{x}, \vect{x}')$, where $\vect{x}, \vect{x}'\in\calX$ and $\vect{x}\neq\vect{x}'$.
\end{definition}

By carefully choosing the objective functions $f:2^{\scrM}\to\mathbb{R}$, we can rely on classical results by G. L. Nemhauser and L. A. Wolsey et al. \cite{Nemhauser1978analysis}, \cite{Wolsey1982analysis} to prove that greedy algorithms can be used to place the sensors with near-optimal performance.
In particular, each objective that we propose is normalized so that $f(\emptyset)=0$, monotone non-decreasing so that $\scrS\subseteq\scrS'$ implies $f(\scrS)\leq f(\scrS')$, and has a diminishing returns property called \emph{submodularity}.
\begin{definition}[Submodular Function]
\label{def:submodular_function}
Let $\scrM$ be a finite set and denote the set of all subsets of $\scrM$ by $2^\scrM$.
A real-valued function of the subsets $f:2^\scrM\to\mathbb{R}$ is called ``submodular'' when it has the following diminishing returns property:
for any element $j\in\scrM$ and subsets $\scrS,\scrS'\subseteq\scrM$,
\begin{equation}
\scrS\subseteq\scrS'\subseteq{\scrM}\setminus\lbrace j \rbrace \quad \Rightarrow \quad
f(\scrS\cup\lbrace j \rbrace) - f(\scrS) \geq f(\scrS'\cup\lbrace j \rbrace) - f(\scrS').
\end{equation}
That is, adding any new element $j$ to the smaller set $\scrS$ increases $f$ at least as much as adding the same element to the larger set $\scrS'\supseteq\scrS$.
\end{definition}

Note that in applications we often do not have direct access to the full set $\calX$, which may be continuous. 
Rather, we have a discrete collection of data $\calX_N = \lbrace \xx_1,\ldots, \xx_N \rbrace\subset\calX$, which we assume is large enough to achieve suitable approximations of the underlying set.

\subsection{Maximizing Detectable Differences}
\label{subsec:MSF}
As we have seen in the first half of this paper, a set of sensors can be considered good if nearby measurements come only from states whose target variables are also close together.
Otherwise a small perturbation to the measurements results in a large change in the quantities of interest.
One way to quantify this intuition is to select measurements that minimize the sum of squared differences in the target variables
associated with states whose measurements are closer together than a fixed detection threshold $\gamma > 0$, i.e.,
\begin{equation}
    F_{\gamma}(\scrS) := \sum_{\substack{\xx, \xx' \in\calX_N\ : \\ \Vert \mS(\xx) - \mS(\xx') \Vert_2 < \gamma}} \left\Vert \vect{g}(\vect{x}) - \vect{g}(\vect{x}') \right\Vert_2^2.
    \label{eqn:total_undetectable_differences}
\end{equation}
The length scale $\gamma$ determines how large of a difference between measurements the user deems to be significant enough to tell the two states $\xx$ and $\xx'$ apart.
For instance, $\gamma^2$ might be selected to be proportional to the noise variance using a desired number $\#(\scrS) = K$ of sensors.
Let the sum of squared differences in the target variables along each secant be denoted by
\begin{equation}
    F_{\infty} := \sum_{\xx, \xx' \in\calX_N} \left\Vert \vect{g}(\vect{x}) - \vect{g}(\vect{x}') \right\Vert_2^2.
    \label{eqn:total_fluctuation}
\end{equation}
Then it is clear that minimizing the sum of squared ``undetectable'' differences given by Eq.~\textbf{\ref{eqn:total_undetectable_differences}} is equivalent to maximizing an objective function
\begin{equation}
    \tilde{f}_{\gamma}(\scrS) = F_{\infty} - F_{\gamma}(\scrS)  = \sum_{\xx, \xx' \in\calX_N} \tilde{w}_{\gamma, \xx, \xx'}(\scrS) \Vert \vect{g}(\xx) - \vect{g}(\xx')\Vert_2^2,
    \label{eqn:rigid_MSF_objective}
\end{equation}
where $\tilde{w}_{\gamma, \xx, \xx'}(\scrS)$ is one if $\Vert \mS(\xx) - \mS(\xx') \Vert_2 \geq \gamma$ and is zero otherwise.
This weight function indicates whether our measurements $\mS$ can distinguish the states $\vect{x}$ and $\vect{x}'$ using the detection threshold $\gamma$, and may be written
\begin{equation}
    \tilde{w}_{\gamma, \xx, \xx'}(\scrS) = \bbone\left\lbrace \Vert \mS(\xx) - \mS(\xx') \Vert_2 \geq \gamma \right\rbrace,
\end{equation}
where $\bbone\{ A\} = 1$ if $A$ is true and $0$ if $A$ is false.
Therefore, we can view the objective in Eq.~\textbf{\ref{eqn:rigid_MSF_objective}} as the sum of squared differences that are ``detectable.''

Maximizing the objective in Eq.~\textbf{\ref{eqn:rigid_MSF_objective}} over a fixed number of sensors $\#(\scrS) \leq K$ is a combinatorial optimization problem and to our knowledge does not admit an efficient direct approximation algorithm.
However, if we reformulate the objective using a relaxed weight function
\begin{equation}
    w_{\gamma, \xx, \xx'}(\scrS) = \min\left\lbrace\frac{1}{\gamma^2}\Vert \mS(\xx) - \mS(\xx') \Vert_2^2,\ 1 \right\rbrace,
    \label{eqn:relaxed_weight_function}
\end{equation}
then
\begin{equation}
\boxed{
    f_{\gamma}(\scrS) = \sum_{\xx, \xx' \in\calX_N} w_{\gamma, \xx, \xx'}(\scrS) \left\Vert \vect{g}(\vect{x}) - \vect{g}(\vect{x}') \right\Vert_2^2,
    }
    \label{eqn:relaxed_MSF_objective}
\end{equation}
obtained by replacing $\tilde{w}$ with $w$ in Eq.~\textbf{\ref{eqn:rigid_MSF_objective}}, becomes a normalized, monotone, submodular function on subsets $\scrS\subseteq\scrM$ (Lemma~\ref{lem:submodularity_of_detectable_differences} in the Appendix) and a simple greedy approximation algorithm guarantees near-optimal performance on this problem!
The greedy algorithm produces a sequence of sets $\scrS_1, \scrS_2, \ldots$, by starting with $\scrS_0 = \emptyset$ and adding the sensor $j_k$ to $\scrS_{k-1}$ that maximizes the objective $f_{\gamma}(\scrS_{k-1}\cup\lbrace j\rbrace)$ over all $j\in\scrM\setminus\scrS_{k-1}$.
If $\scrS^*_K$ maximizes $f_{\gamma}(\scrS)$ over all subsets of size $\#(\scrS) = K$ then the classical result of G. L. Nemhauser et al. \cite{Nemhauser1978analysis} states that the objective values attained by the greedily chosen sets satisfy
\begin{equation}
    f_{\gamma}(\scrS_k) \geq \left(1 - e^{-k/K} \right)f_{\gamma}(\scrS^*_K), \qquad k=1,\ldots, \#(\scrM).
    \label{eqn:maximization_guarantee}
\end{equation}

The objective function $f_{\gamma}$ given by Eq.~\textbf{\ref{eqn:relaxed_MSF_objective}} can be viewed as a ``submodular relaxation'' of the original sum of squared differences $\tilde{f}_{\gamma}$ given by Eq.~\textbf{\ref{eqn:rigid_MSF_objective}}.
While $f_{\gamma}(\scrS) \geq \tilde{f}_{\gamma}(\scrS)$ for every $\scrS\subseteq\scrM$, Theorem~\ref{thm:submodular_relaxation}, below, shows that $f_{\gamma}$ also provides a lower bound on $\tilde{f}_{\gamma'}$ at reduced values of the detection threshold $\gamma' < \gamma$.
Hence, maximization of $f_{\gamma}$ is justified as a proxy for maximizing $\tilde{f}_{\gamma'}$.
Moreover, the relaxed objective bounds the total square differences among target variables that are \emph{not detectable} due to corresponding measurement differences smaller than reduced threshold via Eq.~\textbf{\ref{eqn:relaxed_upper_bound_on_error}} of Theorem~\ref{thm:submodular_relaxation}.
\begin{theorem}[Relaxation Bound on Undetectable Differences]
\label{thm:submodular_relaxation}
Consider the rigid and relaxed objectives given by Eq.~\textbf{\ref{eqn:rigid_MSF_objective}} and Eq.~\textbf{\ref{eqn:relaxed_MSF_objective}}.
Then for every $\scrS\subseteq\scrM$ and constant $0 < \alpha < 1$, we have
\begin{equation}
    \tilde{f}_{\alpha\gamma}(\scrS) \geq \frac{1}{1-\alpha^2}\left[ f_{\gamma}(\scrS) - \alpha^2 F_{\infty} \right].
\end{equation}
Furthermore, the total fluctuation between target variables associated with states whose measurements are closer together than the reduced detection threshold $\alpha\gamma$, given by Eq.~\textbf{\ref{eqn:total_undetectable_differences}}, is bounded above by
\begin{equation}
    F_{\alpha\gamma}(\scrS) \leq \frac{1}{1-\alpha^2}\left[ F_{\infty} - f_{\gamma}(\scrS) \right].
    \label{eqn:relaxed_upper_bound_on_error}
\end{equation}
\end{theorem}
\begin{proof}
We observe that
\begin{equation}
    \Vert \mS(\xx) - \mS(\xx') \Vert_2 \geq \alpha\gamma \quad \Leftrightarrow \quad
    w_{\gamma, \xx, \xx'}(\scrS) \geq \alpha^2
\end{equation}
and so we have
\begin{align}
    \tilde{w}_{\alpha\gamma, \xx, \xx'}(\scrS) &= \bbone\left\lbrace \Vert \mS(\xx) - \mS(\xx') \Vert_2 \geq \alpha\gamma \right\rbrace \\
    &= \bbone\left\lbrace w_{\gamma, \xx, \xx'}(\scrS) \geq \alpha^2 \right\rbrace.
\end{align}
Since $0\leq w_{\gamma, \xx, \xx'}(\scrS)\leq 1$, we obtain the following linear lower bound
\begin{equation}
    \tilde{w}_{\alpha\gamma, \xx, \xx'}(\scrS) \geq \frac{1}{1-\alpha^2}\left[ w_{\gamma, \xx, \xx'}(\scrS) - \alpha^2 \right].
\end{equation}
Summing this lower bound over all secants gives
\begin{equation}
    \tilde{f}_{\alpha\gamma}(\scrS) \geq \frac{1}{1-\alpha^2}\left[ f_{\gamma}(\scrS) - \alpha^2 F_{\infty} \right]
\end{equation}
and subtracting each side from $F_{\infty}$ yields the final result.
\end{proof}

When applied to the shock-mixing layer problem with the leading Isomap coordinates taken as the target variables $\vect{g}(\xx) = (\phi_1(\xx), \phi_2(\xx))$, the greedy algorithm maximizing $f_{\gamma}$ first reveals the two sensor locations marked by green stars and then the black star in Figure~\ref{fig:shock_mixing_layer_snapshots} over the range of $0.02 \leq \gamma \leq 0.06$.
These choices produce the measurements shown in Figs.~\ref{subfig:secant_MSE_data2}~and~\ref{subfig:secant_MSE_data3}, which can be used to reveal the exact phase of the system.
Choosing smaller values of $\gamma$ yields different sensors that can also be used to reveal the phase, but with reduced robustness to measurement perturbations.
This method of maximizing detectable differences also reveals the correct $K=3$ fundamental Isomap eigen-coordinates from among the leading $100$ on the torus example in Eq.~\textbf{\ref{eqn:torus}} over a wide range $0.05\leq \gamma\leq 3.0$.
For implementation details, see the Appendix.

\subsection{Minimal Sensing to Meet an Error Tolerance}
\label{subsec:error_tol}
The approach presented above relies on an average and so does not guarantee that the target value $\vect{g}(\xx)$ can be recovered from the selected measurements $\mS(\xx)$ for every $\xx\in\calX$. 
In this section, we modify the technique developed above in order to provide such a guarantee by trying to find the minimum number of sensors so that every pair of states in the sampled set $\calX_N$ with target values separated by at least $\varepsilon$ correspond to measurements separated by at least $\gamma$.
If our sampled points $\calX_N$ come sufficiently close to every point of $\calX$ in the sense of Definition~\ref{def:epsilon_net}, then Proposition~\ref{prop:error_tol_separation}, given below, allows us to draw a similar conclusion about the measurements from all points in the underlying set $\calX$.
\begin{definition}[$\varepsilon_0$-net]
\label{def:epsilon_net}
An $\varepsilon_0$-net of $\calX$ is a finite subset $\calX_N \subset \calX$ satisfying
\begin{equation}
    \forall \xx\in\calX, \quad \exists \xx_i\in\calX_N \quad \mbox{such that}\quad \Vert \xx - \xx_i\Vert_2 < \varepsilon_0.
\end{equation}
We use the subscript $N$ to denote the number of points in $\calX_N$.
\end{definition}
In particular, if $\calX_N$ forms a fine enough $\varepsilon_0$-net of $\calX$, then Proposition~\ref{prop:error_tol_separation} guarantees that small measurement differences never correspond to large target value differences.
\begin{proposition}[Separation Guarantee on Underlying Set]
\label{prop:error_tol_separation}
Let $\calX_N$ be an $\varepsilon_0$-net of $\calX$ (see Definition~\ref{def:epsilon_net})
and let $\scrS$ be a subset of $\scrM$ satisfying
\begin{equation}
    \forall \xx_i,\xx_j\in\calX_N \qquad
    \Vert \vect{g}(\xx_i) - \vect{g}(\xx_j) \Vert_2 \geq \varepsilon \quad \Rightarrow \quad \Vert \mS(\xx_i) - \mS(\xx_j) \Vert_2 \geq \gamma.
\end{equation}
If $\mS$ and $\vect{g}$ are Lipschitz functions with Lipschitz constants $\Vert \mS\Vert_{\text{lip}}$ and $\Vert \vect{g}\Vert_{\text{lip}}$ respectively, then
\begin{multline}
    \forall \xx,\xx'\in\calX \qquad
    \Vert \vect{g}(\xx) - \vect{g}(\xx') \Vert_2 \geq \varepsilon + 2\varepsilon_0 \Vert \vect{g}\Vert_{\text{lip}} \\
    \quad \Rightarrow \quad 
    \Vert \mS(\xx) - \mS(\xx') \Vert_2 > \gamma - 2\varepsilon_0 \Vert \mS \Vert_{\text{lip}}.
\end{multline}
\end{proposition}
\begin{proof}
The proof follows immediately from successive applications of the triangle inequality and so we relegate it to Appendix~\ref{app:proofs}
\end{proof}
Consequently, the approach described in this section allows one to reconstruct $\vect{g}(\xx)$ from a perturbed measurement $\mS(\xx) + \vect{n}$ by taking the value $\vect{g}(\xx')$ from its nearest neighbor $\mS(\xx')$ with $\xx'\in\calX$ and achieve small error $\Vert \vect{g}(\xx) - \vect{g}(\xx')\Vert_2$ as long as the perturbation $\Vert \vect{n}\Vert_2$ is below a threshold.

Supposing that the desired 
separation
can be obtained using all of the sensors, i.e., $\scrS=\scrM$, then we can take the sum in the objective $f_{\gamma}$ given by Eq.~\textbf{\ref{eqn:relaxed_MSF_objective}} only over those pairs $\xx,\xx'\in\calX_N$ with targets separated by at least $\Vert \vect{g}(\xx) - \vect{g}(\xx')\Vert_2 \geq \varepsilon$, i.e.,
\begin{equation}
    \boxed{
    f_{\gamma,\varepsilon}(\scrS) = \sum_{\substack{\xx, \xx' \in\calX_N\ : \\ \Vert \vect{g}(\xx) - \vect{g}(\xx')\Vert_2 \geq \varepsilon}} w_{\gamma, \xx, \xx'}(\scrS) \Vert \vect{g}(\xx) - \vect{g}(\xx')\Vert_2^2,
    }
    \label{eqn:max_sep_objective}
\end{equation}
and state the problem formally as
\begin{equation}
    \minimize_{\scrS\subseteq\scrM} \#(\scrS) \quad \mbox{subject to} \quad f_{\gamma,\varepsilon}(\scrS) = f_{\gamma,\varepsilon}(\scrM).
    \label{eqn:set_cover}
\end{equation}
We observe that if all points $\xx,\xx'\in\calX_N$ with $\Vert \vect{g}(\xx) - \vect{g}(\xx')\Vert_2 \geq \varepsilon$ can be separated by at least $\gamma$ using $\scrS = \scrM$ then $w_{\gamma,\xx,\xx'}(\scrM) = 1$ for each term in Eq.~\textbf{\ref{eqn:max_sep_objective}}.
On the other hand if there is such a pair $\xx,\xx'$ with $\Vert \mS(\xx) - \mS(\xx') \Vert_2 < \gamma$ then that term has $w_{\gamma,\xx,\xx'}(\scrS) < 1$ and $f_{\gamma,\varepsilon}(\scrS) < f_{\gamma,\varepsilon}(\scrM)$ as a consequence.

One can show, by using the same argument as in Lemma~\ref{lem:submodularity_of_detectable_differences} of the Appendix, that the objective Eq.~\textbf{\ref{eqn:max_sep_objective}} is submodular in addition to being normalized and monotone non-decreasing.
It follows that Eq.~\textbf{\ref{eqn:set_cover}} is a classical submodular set cover problem for which a greedy algorithm maximizing $f_{\gamma,\varepsilon}$ and stopping when $f_{\gamma,\varepsilon}(\scrS_K) = f_{\gamma,\varepsilon}(\scrM)$ will always find, up to a logarithmic factor, the minimum possible number of sensors \cite{Wolsey1982analysis}.
In particular, suppose that $\scrS^*$ is a subset of minimum size with $f_{\gamma,\varepsilon}(\scrS^*) = f(\scrM)$ and that the greedy algorithm chooses a sequence of subsets $\scrS_1, \ldots, \scrS_K$ with $f_{\gamma,\varepsilon}(\scrS_K) = f_{\gamma,\varepsilon}(\scrM)$.
If we define the ``increment condition number'' to be the ratio of the largest and smallest increments in the objective during greedy optimization
\begin{equation}
    \kappa = \frac{f_{\gamma,\varepsilon}(\scrS_1)}{f_{\gamma,\varepsilon}(\scrS_K) - f_{\gamma,\varepsilon}(\scrS_{K-1})},
    \label{eqn:conditon_number}
\end{equation}
then the classical result of L. A. Wolsey \cite{Wolsey1982analysis} proves that the greedily chosen set is no larger than
\begin{equation}
    \#(\scrS_K) \leq (1 + \ln{\kappa})\#(\scrS^*).
    \label{eqn:set_cover_guarantee}
\end{equation}

\subsection{Minimal Sensing to Meet an Amplification Tolerance}
\label{subsec:amplification_tol}
The approaches discussed above are capable of choosing measurements that separate states with distant target values by at least a fixed distance~$\gamma$.
However, the nearby measurements separated by less than $\gamma$ may not adequately capture the local behavior of the target variables as illustrated by the kink in the measurements made by these sensors in the shock-mixing layer flow shown in Figure~\ref{subfig:secant_MSE_data2}.
This means that while the state can be reconstructed from the measurements, its time derivative cannot.
This would be a major problem if we wish to build a reduced-order model of this system based only on the fluid velocities measured at the chosen points.
In addition, we may want the separation between the measurements to grow with the corresponding separation in target values, rather than potentially saturating at the $\gamma$ threshold.

Attempting to select sensors $\scrS$ whose measurements capture both the local and global structure of the target variables leads us to consider disturbance amplification as a performance metric.
In this section, we try to find the minimum number of sensors so that the Lipschitz constant of the reconstruction function does not exceed a user-specified threshold~$L$.  In practice, we do not have access to the true Lipschitz constant, so instead we bound a proxy defined below:
\begin{equation}
    \Vert \PhiS\Vert_{\text{lip}} 
    \approx \Vert \PhiS \Vert_{\calX_N, \text{lip}} 
    = \max_{\xx,\xx'\in\calX_N} \frac{\Vert \vect{g}(\xx) - \vect{g}(\xx')\Vert_2}{\Vert \mS(\xx) - \mS(\xx') \Vert_2} \leq L.
    \label{eqn:Lipschitz_condition}
\end{equation}
Proposition~\ref{prop:amplification_tol_separation}, below, shows that it suffices to enforce this condition over an $\varepsilon_0$-net, $\calX_N$, of $\calX$ (see Definition~\ref{def:epsilon_net}) in order to bound the amplification over all of $\calX$ up to a slight relaxation for measurement differences on the same scale $\varepsilon_0$ as the sampling.
\begin{proposition}[Amplification Guarantee on Underlying Set]
\label{prop:amplification_tol_separation}
Let $\calX_N$ be an $\varepsilon_0$-net of $\calX$ and let $\scrS$ be a subset of $\scrM$ satisfying
\begin{equation}
    \xx_i,\xx_j\in\calX_N \qquad
    \Vert \vect{g}(\xx_i) - \vect{g}(\xx_j) \Vert_2 \leq L \Vert \mS(\xx_i) - \mS(\xx_j) \Vert_2.
\end{equation}
If $\mS$ and $\vect{g}$ are Lipschitz functions, with Lipschitz constants $\Vert \mS\Vert_{\text{lip}}$ and $\Vert \vect{g}\Vert_{\text{lip}}$ respectively, then
\begin{multline}
    \forall \xx,\xx'\in\calX \qquad
    \Vert \vect{g}(\xx) - \vect{g}(\xx') \Vert_2 <
    L \Vert \mS(\xx) + \mS(\xx') \Vert_2 \\
    + 2\left( \Vert \vect{g}\Vert_{\text{lip}} + L \Vert \mS \Vert_{\text{lip}} \right)\varepsilon_0.
\end{multline}
\end{proposition}
\begin{proof}
The proof is a direct application of the triangle inequality and so it is relegated Appendix~\ref{app:proofs}.
\end{proof}

If the Lipschitz condition in Eq.~\textbf{\ref{eqn:Lipschitz_condition}} over $\calX_N$ can be met using all of the sensors $\scrS=\scrM$ then the problem we hope to solve can be stated formally as in Eq.~\textbf{\ref{eqn:set_cover}}, where the condition Eq.~\textbf{\ref{eqn:Lipschitz_condition}} is imposed using a different normalized, monotone, submodular function
\begin{equation}
        \boxed{
        f_{L}(\scrS) = \sum_{\substack{\xx, \xx'\in\calX_N\\ \vect{g}(\xx) \neq \vect{g}(\xx')}}
        \min\left\lbrace \frac{\Vert \mS(\xx) - \mS(\xx') \Vert_2^2 }
        {\Vert \vect{g}(\xx) - \vect{g}(\xx') \Vert_2^2},\ 
        \frac{1}{L^2} \right\rbrace.
        }
        \label{eqn:Lipschitz_constraint_function}
\end{equation}
See Lemma~\ref{lem:submodularity_of_Lipschitz_objective} in the Appendix for proof of these properties.

We observe that if there is any secant $(\xx,\xx')\in\calX_N\times\calX_N$ for which Eq.~\textbf{\ref{eqn:Lipschitz_condition}} is not satisfied for a given $\scrS\subset\scrM$, then the corresponding term of Eq.~\textbf{\ref{eqn:Lipschitz_constraint_function}} is less than $1/L^2$ and $f_{L}(\scrS)<f_{L}(\scrM)$.
Otherwise, each term of Eq.~\textbf{\ref{eqn:Lipschitz_constraint_function}} is $1/L^2$ and we have $f_{L}(\scrS)=f_{L}(\scrM)$.
Again, the classical result in~\cite{Wolsey1982analysis} shows that a greedy approximation algorithm maximizing Eq.~\textbf{\ref{eqn:Lipschitz_constraint_function}} and stopping when $f_{L}(\scrS_K) = f_{L}(\scrM)$ finds the minimum possible number of sensors up to a logarithmic factor so that the Lipschitz condition Eq.~\textbf{\ref{eqn:Lipschitz_condition}} is satisfied.
In particular, the same guarantee stated in Eq.~\textbf{\ref{eqn:set_cover_guarantee}} holds for the Lipschitz objective too.

In some applications, we may instead want to find the measurements that minimize the reconstruction Lipschitz constant $\Vert \PhiS \Vert_{\calX_N, \text{lip}}$ using a fixed sensor budget $\#(\scrS) \leq C$.
By running the greedy algorithm repeatedly using different thresholds $L$ it is possible to obtain upper and sometimes lower bounds on this budget-constrained minimum Lipschitz constant $L^*$.
This idea is closely related to the approach of \cite{Krause2008robust}.
If the greedy algorithm using Lipschitz constant $L$ chooses sensors $\scrS$ that meet the budget $\#(\scrS) \leq C$ then $L$ is obviously an upper bound on $L^*$.
In practice, we can use a bisection search over $L$ to find nearly the smallest $L$ to any given tolerance for which $\#(\scrS) \leq C$.
To get the lower bound, the greedy algorithm is run with a small enough $L$ so that the bound on the minimum possible cost from Eq.~\textbf{\ref{eqn:set_cover_guarantee}} exceeds the budget
\begin{equation}
    C < \#(\scrS)/(1 + \ln \kappa).
\end{equation}
If this is the case, there is no collection of measurements with amplification at most $L$ that meets the cost constraint.
Thus, such an $L$ is a lower bound on the minimum possible amplification using measurement budget $C$.
Again, bisection search can be used to find nearly the largest $L$ so that $C < \#(\scrS)/(1 + \ln \kappa)$.

With the leading Isomap coordinates taken as the target variables $\vect{g}(\xx) = (\phi_1(\xx), \phi_2(\xx))$, a bisection search over $L$ identifies the three sensor locations marked by black squares in Figure~\ref{fig:shock_mixing_layer_snapshots} on the shock-mixing layer problem and the correct fundamental Isomap eigenfunctions $\phi_1,\phi_2,\phi_7$ on the torus example in Eq.~\textbf{\ref{eqn:torus}}.
The measurements made by these sensors on the shock-mixing layer problem are shown in Figure~\ref{subfig:secant_Lip_data} and indicate, by the lack of self-intersections, that they can be used to recover the phase.

The minimum number of sensors selected by the greedy algorithm that allow one to reconstruct both the relevant information $\vect{g}(\xx)$ and its time derivative is usually persistent over a wide range of Lipschitz constants with fewer sensors not being chosen until $L$ is made extremely large.
In the shock-mixing layer problem, three sensors that successfully reveal the underlying phase are found for values of $L$ ranging from $1868$ to $47624$, above which only two sensors that cannot reveal the underlying phase are selected.
The fact that a smaller set of inadequate sensors are selected for extremely large~$L$ reflects our use of a discrete approximation~$\calX_N$ of the continuous set~$\calX$.
Measurements from $\calX_N$ will almost never truly overlap to give $\Vert \PhiS \Vert_{\text{lip}} = \infty$ as they would for measurements from~$\calX$.

We also find that with $L=129$, the minimum possible number of sensors exceeds $\#(\scrS_K)/(1+\ln{\kappa}) = 3.18 > 3$ on the shock-mixing layer problem.
Therefore, the minimum possible reconstruction Lipschitz constant using three sensors that one might find by an exhaustive search over the $\binom{2210}{3} \approx 1.8\times 10^9$ possible combinations must be greater than $129$.
For implementation details, see the Appendix.

\section{Computational Considerations and Down-Sampling}
\label{sec:computational_considerations}
So far, the three secant-based methods we presented involve objectives that sum over $\mathcal{O}(N^2)$ pairs of points from the sampled set $\calX_N$.
In this section, we discuss how this large collection of secants can be sub-sampled to produce high-probability performance guarantees using a number of secants that scales more favorably with the size of the data set.
By sub-sampling we do pay a price in the sense that some ``bad'' secants may escape our sampling scheme and so we cannot draw the same conclusions about every point in the underlying set as we did in Propositions~\ref{prop:error_tol_separation}~and~\ref{prop:amplification_tol_separation} for the sensors chosen using the methods in Sections~\ref{subsec:error_tol}~and~\ref{subsec:amplification_tol}.
Instead, we can bound the size of the set of these ``bad'' secants with high probability by using a sampled collection of secants that scales linearly with $N$.
In the case of the total detectable difference-based objective discussed in Section~\ref{subsec:MSF}, we can prove high-probability bounds for the sum of squared undetectable differences in the target variables using a constant number of secants that doesn't depend on $N$ at all.

Before getting started with our discussion of down-sampling, let us first mention that the calculation of each of the objectives formulated in Section~\ref{sec:secants} is easily parallelizable, whether or not they are down-sampled.
Even though the computation of each objective function given by Eq.~\textbf{\ref{eqn:relaxed_MSF_objective}},~\textbf{\ref{eqn:max_sep_objective}},~or~\textbf{\ref{eqn:Lipschitz_constraint_function}} requires $\mathcal{O}(N^2)$ operations, the terms being summed can be distributed among many processors without the need for any communication except at the end when each processor reports the sum over the secants allocated to it.
Furthermore, because each secant-based objective we consider in this paper is submodular, it is not actually necessary to evaluate the objectives over all of the remaining sensors during each step of the greedy algorithm.
By employing the ``accelerated greedy'' algorithm of M. Minoux \cite{Minoux1978accelerated}, the same set of sensors can be found using a minimal number of evaluations of the objective.
We provide a summary of the accelerated greedy algorithm in Section~\ref{subapp:AG_algorithm} of the Appendix.

The computational cost of evaluating the objectives in  Sections~\ref{subsec:error_tol}~and~\ref{subsec:amplification_tol} during each step of the greedy algorithm may also be reduced by
exploiting the fact that each term in the sum is truncated once the measurements achieve a certain level of separation.
This means that only the nearest neighbors within a known distance of each $\mS(\xx)$, $\xx\in\calX_N$ need to be computed and rest of the terms all achieve the threshold and need not be computed explicitly.
To compute the sum efficiently, fixed-radius near neighbors algorithms \cite{Bentley1975survey}, \cite{Bentley1977complexity} could be employed.

\subsection{Maximizing Detectable Differences}
\label{subsec:sampling_for_detectable_differences}
The main results of this section are Theorems~\ref{thm:sampled_mean_fluctuation_near_optimality}~and~\ref{thm:mean_square_undetectable_difference}, which show that with high probability we can obtain guaranteed performance in terms of mean undetectable differences by sampling a constant number of secants (i.e., independent of $N$) selected at random.
In particular, Theorem~\ref{thm:sampled_mean_fluctuation_near_optimality} bounds the worst-case performance of the greedy algorithm with high probability using the sampled objective. 
Theorem~\ref{thm:mean_square_undetectable_difference}, on the other hand, shows that if one only considers randomly sampled secants with target variables separated by at least $\varepsilon$ (see Section~\ref{subsec:error_tol}), then the mean square undetectable difference between target values is less than $2\varepsilon^2$ with high probability.

While the original mean square fluctuation objective in Eq.~\textbf{\ref{eqn:relaxed_MSF_objective}} was formulated over the discrete set $\calX_N$, we can actually prove more versatile approximation results about an objective defined as an average over the entire, possibly continuous, set $\calX$ with respect to a probability measure $\mu$.
In particular, we assume the target variables $\vect{g}$ and measurements $\vect{m}_j$, $j\in\scrM$ are measurable functions on $\calX$ and consider an average detectable difference objective
\begin{equation}
    f_{\gamma}(\scrS) = \int_{\calX\times\calX} w_{\gamma,\xx,\xx'}(\scrS) \left\Vert \vect{g}(\xx) - \vect{g}(\xx')\right\Vert_2^2 \ d\mu(\xx)d\mu(\xx')
    \label{eqn:relaxed_MSF_integral_objective}
\end{equation}
with $w_{\gamma,\xx,\xx'}(\scrS)$ defined by Eq.~\textbf{\ref{eqn:relaxed_weight_function}}.
We also denote the average fluctuations between target variables associated with states whose measurements are closer together than the detection threshold $\gamma$ by
\begin{equation}
    F_{\gamma}(\scrS) :=
    \int_{\substack{(\vect{x},\vect{x}') \in \calX\times\calX\ :\\
    \Vert \mS(\vect{x}) - \mS(\vect{x}')\Vert_2 < \gamma }} \left\Vert \vect{g}(\xx) - \vect{g}(\xx')\right\Vert_2^2\ d\mu(\xx)d\mu(\xx')
    \label{eqn:continuous_mean_undetectable_differences}
\end{equation}
and the total fluctuation among target variables by
\begin{equation}
    F_{\infty} :=
    \int_{\calX\times\calX} \left\Vert \vect{g}(\xx) - \vect{g}(\xx')\right\Vert_2^2\ d\mu(\xx)d\mu(\xx').
    \label{eqn:total_fluctuation_continuous}
\end{equation}
Note that the original objective formulated in Section~\ref{subsec:MSF} as well as Eq.~\textbf{\ref{eqn:total_undetectable_differences}} are special cases of Eq.~\textbf{\ref{eqn:relaxed_MSF_integral_objective}} and Eq.~\textbf{\ref{eqn:continuous_mean_undetectable_differences}}, up to an irrelevant constant factor, when $\mu=\frac{1}{N}\sum_{\xx\in\calX_N}\delta_{\xx}$ and $\delta_{\xx}(A) = \mathbbm{1}\lbrace \xx\in A \rbrace$ is the Dirac measure on Borel sets $A\subseteq\calX$.
By Lemma~\ref{lem:submodularity_of_detectable_differences}, Eq.~\textbf{\ref{eqn:relaxed_MSF_integral_objective}} is submodular in addition to being normalized and monotone non-decreasing.
Furthermore, by an identical argument to Theorem~\ref{thm:submodular_relaxation}, we know that the mean square fluctuation between target variables associated with states whose measurements are closer together than a reduced detection thereshold $\alpha\gamma$ with $0<\alpha<1$ is bounded above by
\begin{equation}
    F_{\alpha\gamma}(\scrS)
    \leq \frac{1}{1-\alpha^2}\left[ F_{\infty} - f_{\gamma}(\scrS) \right].
    \label{eqn:undetectable_differences_bound}
\end{equation}

We begin with Lemma~\ref{lem:uniform_bound_on_sampled_mean_fluction_objective}, which shows that by sampling a large enough collection of points $\xx_1,\xx'_1, \ldots,\xx_m,\xx'_m\in\calX$ independently according to $\mu$, the objective $f_{\gamma}$ can be uniformly approximated by a sample-based average
\begin{equation}
    f_{\gamma,m}(\scrS) = \frac{1}{m}\sum_{i=1}^m w_{\gamma,\xx_i,\xx_i'}(\scrS) \left\Vert \vect{g}(\xx_i) - \vect{g}(\xx_i') \right\Vert_2^2
    \label{eqn:relaxed_MSF_sampled_objective}
\end{equation}
over all $\scrS\subseteq\scrM$ of size $\#(\scrS)\leq L$ with high probability over the sample points.
Most importantly, the number of sample points needed for this approximation guarantee is independent of the distribution $\mu$.
Consequently if we have access to $N$ points making up $\calX_N$ that have been sampled independently according to $\mu$, we need only keep the first $2m$ of them to accurately approximate the objective.
The number $m$ of such sub-sampled points depends only on the quality of the probabilistic guarantee and not on the size of the data set $N$.

\begin{lemma}[Accuracy of the Down-Sampled Objective]
    \label{lem:uniform_bound_on_sampled_mean_fluction_objective}
    Consider the objectives $f_{\gamma}$ and $f_{\gamma,m}$ defined according to Eq.~\textbf{\ref{eqn:relaxed_MSF_integral_objective}} and Eq.~\textbf{\ref{eqn:relaxed_MSF_sampled_objective}}.
    Assume that the target function is bounded over $\calX$ so that 
    \begin{equation}
        D = \diam \vect{g}(\calX) = \sup_{\xx,\xx'\in\calX}\Vert \vect{g}(\xx) - \vect{g}(\xx')\Vert_2 < \infty.
    \end{equation}
    and that $\xx_1,\xx'_1, \ldots,\xx_m,\xx'_m\in\calX$ are sampled independently according to a probability measure $\mu$ on $\calX$.
    If the number of sampled pairs is at least
    \begin{equation}
        m \geq \frac{D^4}{2\varepsilon^2}\left[ L\ln{\#(\scrM)} - \ln{\left( (L-1)! \right)} - \ln{\left(\frac{p}{2}\right)} \right],
        \label{eqn:num_secants_for_bound_on_sampled_mean_fluction}
    \end{equation}
    then $\vert f_{\gamma,m}(\scrS) - f_{\gamma}(\scrS) \vert < \varepsilon$ for every $\scrS\subseteq\scrM$ of size $\#(\scrS)\leq L$ with probability at least $1-p$.
\end{lemma}
\begin{proof}
    For simplicity, we will drop $\gamma$ from the subscripts on our objectives since $\gamma$ remains fixed throughout the proof.
    Let us begin by fixing a set $\scrS\subseteq\scrM$ of size $\#(\scrS) \leq L$ and denoting $M=\#(\scrM)$ for short.
    Under the assumption that the points $\xx_i,\xx'_i$ are sampled independently and identically under $\mu$, the random variables
    \begin{equation}
        Z_i(\scrS) = w_{\xx_i,\xx_i'}(\scrS) \Vert \vect{g}(\xx_i) - \vect{g}(\xx_i') \Vert_2^2, \quad i=1,\ldots,m,
    \end{equation}
    are independent and bounded by $0 \leq Z_i(\scrS) \leq D^2$.
    The value of the optimization objective is the expectation $f(\scrS) = \mathbb{E}[Z_i(\scrS)]$ and the value of our sub-sampled objective is the empirical average
    \begin{equation}
        f_m(\scrS) = \frac{1}{m}\sum_{i=1}^m Z_i(\scrS).
    \end{equation}
    Hoeffding's inequality allows us to bound the probability that $f_m(\scrS)$ differs from $f(\scrS)$ by more than $\varepsilon$ according to
    \begin{equation}
        \mathbb{P}\left\lbrace \left\vert f_m(\scrS) - f(\scrS) \right\vert \geq \varepsilon \right\rbrace \leq 2 \exp\left( -\frac{2 m \varepsilon^2}{D^4} \right).
    \end{equation}
    We want the objective to be accurately approximated with tolerance $\varepsilon$ uniformly over all collections of sensors of size $\#(\scrS) \leq L$. 
    We unfix $\scrS$ by taking the union bound
    \begin{equation}
        \mathbb{P}\bigcup_{\substack{\scrS\subseteq\scrM: \\ \#(\scrS) \leq L}}\left\lbrace \left\vert f_m(\scrS) - f(\scrS) \right\vert \geq \varepsilon \right\rbrace \leq  \sum_{\substack{\scrS\subseteq\scrM: \\ \#(\scrS) \leq L}} 2 \exp\left( -\frac{2 m \varepsilon^2}{D^4} \right).
    \end{equation}
    The combinatorial inequality
    \begin{equation}
        \#\left(\lbrace \scrS\subseteq\scrM \ :\ \#(\scrS) \leq L \rbrace\right) = \sum_{k=1}^L \binom{M}{k}
        \leq \sum_{k=1}^L \frac{M^k}{k !}
        \leq L \frac{M^L}{L !}
        = \frac{M^L}{(L-1)!}
        \label{eqn:comb_ineq}
    \end{equation}
    yields the bound
    \begin{equation}
        \mathbb{P}\bigcup_{\substack{\scrS\subseteq\scrM: \\ \#(\scrS) \leq L}}\left\lbrace \left\vert f_m(\scrS) - f(\scrS) \right\vert \geq \varepsilon \right\rbrace \leq  2\exp\left( L\ln{M} - \ln{\left( (L-1)! \right)} - \frac{2 m \varepsilon^2}{D^4} \right) \leq p
    \end{equation}
    when the number of sampled pairs $\xx_i,\xx_i'$ satisfies Eq.~\textbf{\ref{eqn:num_secants_for_bound_on_sampled_mean_fluction}}.
\end{proof}

The uniform accuracy of the sampled objective $f_{\gamma,m}$ over the feasible subsets $\scrS$ in our optimization problem
\begin{equation}
    \maximize_{\substack{\scrS\subseteq\scrM \ :\ \#(\scrS)\leq K}} f_{\gamma}(\scrS)
\end{equation}
established in Lemma~\ref{lem:uniform_bound_on_sampled_mean_fluction_objective} leads to performance guarantees for the greedy approximation algorithm when the sampled objective $f_{\gamma.m}$ is used in place of $f_{\gamma}$.
In particular, Theorem~\ref{thm:sampled_mean_fluctuation_near_optimality} shows that the greedy algorithm can be applied to the sampled objective Eq.~\textbf{\ref{eqn:relaxed_MSF_sampled_objective}} and still achieve near-optimal performance with respect to the original objective Eq.~\textbf{\ref{eqn:relaxed_MSF_integral_objective}} on the underlying set $\calX$ with high probability.
This sampling-based approach therefore completely eliminates the $\mathcal{O}(N^2)$ dependence of the computational complexity involved in evaluating the objective at a penalty on the worst case performance that can be made arbitrarily small by sampling more points.

\begin{theorem}[Greedy Performance using Sampled Objective]
\label{thm:sampled_mean_fluctuation_near_optimality}
Assume the same hypotheses as Lemma~\ref{lem:uniform_bound_on_sampled_mean_fluction_objective} and let $\scrS^*$ denote an optimal solution of
\begin{equation}
    \maximize_{\substack{\scrS\subseteq\scrM \ :\ \#(\scrS)\leq K}} f_{\gamma}(\scrS),
\end{equation}
with $f_{\gamma}$ given by Eq.~\textbf{\ref{eqn:relaxed_MSF_integral_objective}} and $K \leq L$.
If $\scrS_1, \ldots, \scrS_L$ are the sequence of subsets selected by the greedy algorithm using the sampled objective $f_{\gamma,m}$ given by Eq.~\textbf{\ref{eqn:relaxed_MSF_sampled_objective}}, then
\begin{equation}
    f_{\gamma}(\scrS_k) \geq \left(1 - e^{-k/K} \right)f_{\gamma}(\scrS^*) - \left(2 - e^{-k/K} \right)\varepsilon, \qquad k=1,\ldots, L,
\end{equation}
with probability at least $1-p$ over the sample points.
\end{theorem}
\begin{proof}
For simplicity, we will drop $\gamma$ from the subscripts on our objectives since $\gamma$ remains fixed throughout the proof.
Let $\scrS_m^*$ denote the optimal solution of
\begin{equation}
    \maximize_{\substack{\scrS\subseteq\scrM \ :\ \#(\scrS)\leq K}} f_m(\scrS),
\end{equation}
 using the sampled objective and assume that $\vert f(\scrS) - f_m(\scrS) \vert < \varepsilon$ for every subset $\scrS$ of $\scrM$ with $\#(\scrS)\leq L$.
 According to Lemma~\ref{lem:uniform_bound_on_sampled_mean_fluction_objective}, this happens with probability at least $1-p$ over the sample points.
 Using this uniform approximation and the guarantee on the performance of the greedy algorithm for $f_m$, we have
 \begin{equation}
     f(\scrS_k) \geq f_m(\scrS_k) - \varepsilon
     \geq \left(1 - e^{-k/K} \right)f_m(\scrS_m^*) - \varepsilon.
 \end{equation}
 Since $\scrS_m^*$ is the optimal solution using the sampled objective, we must have $f_m(\scrS_m^*) \geq f_m(\scrS^*)$.
 Using this fact and the uniform approximation gives
  \begin{align}
     f(\scrS_k) &\geq \left(1 - e^{-k/K} \right)f_m(\scrS^*) - \varepsilon \\
     &\geq \left(1 - e^{-k/K} \right)\left(f(\scrS^*) - \varepsilon\right) - \varepsilon.
 \end{align}
 Combining the terms on $\varepsilon$ completes the proof.
\end{proof}

\begin{remark}
While Theorem~\ref{thm:sampled_mean_fluctuation_near_optimality} tells us that down-sampling has a small effect on the worst-case performance of the greedy algorithm, unfortunately, we cannot say much beyond that.
It may be the case that the greedy solution using the sampled objective $f_{\gamma,m}$ produces a very different value of $f_{\gamma}$ than the greedy solution using $f_{\gamma}$ directly, even though these functions are both submodular and differ by no more than an arbitrarily small $\varepsilon>0$.
Consider the following example in Table~\ref{tab:a_sensitive_submodular_function} where we have two submodular objectives, $f$ and $\tilde{f}$, that differ by no more than $\varepsilon\ll 1$, yet the greedy algorithm applied to $f$ and $\tilde{f}$ yield results that differ by $\mathcal{O}(1)$.
\begin{table}[h]
\centering
\begin{tabular}{c|c|c}
    $\scrS$ & $f(\scrS)$ & $\tilde{f}(\scrS)$ \\
    \hline
    $\emptyset$ & $0$ & $0$ \\
    $\lbrace a \rbrace$ & $2+\varepsilon$ & $2$ \\
    $\lbrace b \rbrace$ & $2$ & $2+\varepsilon$ \\
    $\lbrace c \rbrace$ & $1$ & $1$ \\
    $\lbrace a,b \rbrace$ & $2+\varepsilon$ & $2+2\varepsilon$ \\
    $\lbrace a,c \rbrace$ & $3+\varepsilon$ & $3$ \\
    $\lbrace b,c \rbrace$ & $2$ & $2+\varepsilon$ \\
    $\lbrace a,b,c \rbrace$ & $3$ & $3$ \\
\end{tabular}
\caption{Two submodular functions are given that differ by no more than $\varepsilon \ll 1$, yet produce very different greedy solutions and objective values.}
\label{tab:a_sensitive_submodular_function}
\end{table}
One can easily verify that both functions in Table~\ref{tab:a_sensitive_submodular_function} are normalized, monotone, and submodular.
When selecting subsets of size $2$, the greedy algorithm for $f$ picks $\emptyset \to \lbrace a \rbrace \to \lbrace a,c \rbrace$ and the greedy algorithm for for $\tilde{f}$ picks $\emptyset \to \lbrace b \rbrace \to \lbrace a,b \rbrace$.
The values of $f$ on the chosen sets, $f(\lbrace a,c \rbrace) = 3+\varepsilon$ and $f(\lbrace a,b \rbrace)=2+2\varepsilon$, differ by $1-\varepsilon \gg \varepsilon$, and similarly for $\tilde{f}(\lbrace a,c \rbrace) = 3$ and $\tilde{f}(\lbrace a,b \rbrace)=2+\varepsilon$, which also differ by $1-\varepsilon \gg \varepsilon$.
Thus the performance of the greedy algorithm can be sensitive to small perturbations of the objective even though the lower bound on performance is not sensitive.
\end{remark}

It turns out that by solving the error tolerance problem in Section~\ref{subsec:error_tol} greedily using a down-sampled objective, we can provide high probability bounds directly on the mean square undetectable differences in Eq.~\textbf{\ref{eqn:continuous_mean_undetectable_differences}}.
We will use the down-sampled objective
\begin{equation}
    f_{\gamma,\varepsilon,m}(\scrS) = \frac{1}{m}\sum_{\substack{i\in\lbrace 1,\ldots, m\rbrace : \\ \Vert \vect{g}(\xx_i) - \vect{g}(\xx'_i)\Vert_2 \geq \varepsilon}} w_{\gamma, \xx_i, \xx'_i}(\scrS) \Vert \vect{g}(\xx_i) - \vect{g}(\xx'_i)\Vert_2^2,
\end{equation}
with the relaxed weight function in Eq.\textbf{\ref{eqn:relaxed_weight_function}} in a greedy approximation algorithm for the submodular set-cover problem
\begin{equation}
    \minimize_{\scrS\subseteq\scrM} \#(\scrS) \quad \mbox{subject to} \quad f_{\gamma,\varepsilon,m}(\scrS) = f_{\gamma,\varepsilon,m}(\scrM).
    \label{eqn:sampled_set_cover}
\end{equation}
Using the resulting greedy solution $\mathscr{S}_K$ that satisfies $f_{\gamma,\varepsilon,m}(\scrS_K) = f_{\gamma,\varepsilon,m}(\scrM) = \tilde{f}_{\gamma,\varepsilon,m}(\scrM)$, Theorem~\ref{thm:mean_square_undetectable_difference} provides a high-probability bound on the mean square undetectable difference in the target variables, Eq.~\textbf{\ref{eqn:continuous_mean_undetectable_differences}},
over the entire set $\calX\times\calX$ rather than merely $\calX_N\times\calX_N$.

\begin{theorem}[Sample Separation Bound on Undetectable Differences]
    \label{thm:mean_square_undetectable_difference}
    Consider the functions $f_{\gamma,\varepsilon,m}$ and $F_{\gamma}$ defined by Eqs.~\textbf{\ref{eqn:sampled_set_cover}}~and~\textbf{\ref{eqn:continuous_mean_undetectable_differences}} and assume that the condition $\Vert \vect{m}_{\scrM}(\xx)-\vect{m}_{\scrM}(\xx')\Vert_2 \geq \gamma$ holds for $\mu$-almost every $\xx,\xx'\in\calX$ such that $\Vert \vect{g}(\xx) - \vect{g}(\xx') \Vert_2 \geq \varepsilon$.
    Suppose that the target function is bounded over $\calX$ so that 
    \begin{equation}
        D = \diam \vect{g}(\calX) = \sup_{\xx,\xx'\in\calX}\Vert \vect{g}(\xx) - \vect{g}(\xx')\Vert_2 < \infty.
    \end{equation}
    and that $\xx_1,\xx'_1, \ldots,\xx_m,\xx'_m\in\calX$ are sampled independently according to the probability measure $\mu$ on $\calX$.
    If the number of sampled pairs is at least
    \begin{equation}
        m \geq \frac{D^4}{2\varepsilon^4}\left( \#(\scrM) \ln 2 - \ln{p} \right),
    \end{equation}
    and the greedy approximation of Eq.~\textbf{\ref{eqn:sampled_set_cover}} produces a set $\scrS_K$,
    then
    \begin{equation}
        F_{\gamma}(\scrS_K) < 2\varepsilon^2
    \end{equation}
    with probability at least $1-p$.
\end{theorem}
\begin{proof}
For simplicity, we will drop $\gamma,\varepsilon$ from the subscripts on our objectives since $\gamma$ and $\varepsilon$ remain fixed throughout the proof.
Let
\begin{equation}
\mathcal{D} = \left\lbrace (\xx, \xx')\in\calX\times\calX \ :\ \Vert\vect{g}(\xx) - \vect{g}(\xx') \Vert_2\geq\varepsilon \right\rbrace
\end{equation}
and
\begin{multline}
    \tilde{f}(\scrS) = \mathbb{E}\left[ \tilde{f}_m(\mathscr{S}) \right] \\
    = \int_{\calX\times\calX} \chi_{\mathcal{D}}(\xx,\xx') \tilde{w}_{\gamma,\xx,\xx'}(\scrS) \Vert\vect{g}(\xx) - \vect{g}(\xx') \Vert_2^2 \ d\mu(\xx)d\mu(\xx'),
\end{multline}
where $\chi_{\mathcal{D}}$ is the characteristic function of the set $\mathcal{D}$.
From our assumption that $\Vert \vect{m}_{\scrM}(\xx)-\vect{m}_{\scrM}(\xx)\Vert_2 \geq \gamma$ for $\mu$-almost every $\xx,\xx'\in\calX$ with $\Vert \vect{g}(\xx) - \vect{g}(\xx') \Vert_2 \geq \varepsilon$, it follows
\begin{equation}
    \tilde{f}(\scrM) = \int_{\calX\times\calX} \chi_{\mathcal{D}}(\xx,\xx') \Vert\vect{g}(\xx) - \vect{g}(\xx') \Vert_2^2 \ d\mu(\xx)d\mu(\xx').
\end{equation}
Expanding our definition of $F_{\gamma}$ in Eq.\textbf{\ref{eqn:continuous_mean_undetectable_differences}}, we find
\begin{multline}
    F_{\gamma}(\scrS) = \tilde{f}(\scrM) - \tilde{f}(\scrS) \\
    +\int_{\calX\times\calX} \chi_{\mathcal{D}^c}(\xx,\xx') \left[1-\tilde{w}_{\gamma,\xx,\xx'}(\scrS)\right] \Vert\vect{g}(\xx) - \vect{g}(\xx') \Vert_2^2 \ d\mu(\xx)d\mu(\xx')
\end{multline}
and therefore
\begin{equation}
    F_{\gamma}(\scrS) \leq \tilde{f}(\scrM) - \tilde{f}(\scrS) + \varepsilon^2.
    \label{thmeqn:E_bnd}
\end{equation}
We shall now use a similar Hoeffding and union bound argument as in Thm.~\ref{lem:uniform_bound_on_sampled_mean_fluction_objective} to relate $\tilde{f}(\scrM) - \tilde{f}(\scrS)$ to $\tilde{f}_m(\scrM) - \tilde{f}_m(\scrS)$ uniformly over every subset $\scrS\subseteq\scrM$.
Fixing such $\scrS \subset \scrM$, the one-sided Hoeffding inequality tells us that
\begin{equation}
    \mathbb{P}\left\lbrace \left[\tilde{f}(\scrM) - \tilde{f}(\scrS) \right] - \left[\tilde{f}_m(\scrM) - \tilde{f}_m(\scrS) \right] \geq \varepsilon^2 \right\rbrace \leq \exp{\left(- \frac{2 m \varepsilon^4}{D^4} \right)}.
\end{equation}
Unfixing $\scrS$ using the union bound tells us that 
\begin{equation}
    \tilde{f}(\scrM) - \tilde{f}(\scrS) < \tilde{f}_m(\scrM) - \tilde{f}_m(\scrS) + \varepsilon^2
\end{equation}
uniformly over all $\scrS \subset \scrM$ with probability at least $1-p$.
Since the greedy algorithm terminates when $\tilde{f}_m(\scrS_K) = \tilde{f}_m(\scrM)$, it follows by substitution into Eq.~\textbf{\ref{thmeqn:E_bnd}} that
\begin{equation}
    F_{\gamma}(\scrS) < 2\varepsilon^2
\end{equation}
with probability at least $1-p$ over the sample points.
\end{proof}

\subsection{Minimal Sensing to Meet Separation or Amplification Tolerances}
If we want to draw stronger conclusions about the underlying set $\calX$ than are captured by the mean square (un)detectable differences, then we must increase the number of sample points.
The following
Theorems~\ref{thm:minimal_sensing_sampled_error_tolerance}~and~\ref{thm:minimal_sensing_sampled_amplification_tolerance} show that similar conclusions about the separation of points as in Propositions~\ref{prop:error_tol_separation}~and~\ref{prop:amplification_tol_separation} can be achieved over large subsets of $\calX$ with high probability by considering secants between a randomly chosen set of ``base points'' and an the full data set.
More precisely, we will consider secants between an $\varepsilon_0$-net $\calX_N$ of $\calX$ and a collection of base point $\mathcal{B}_m\subset\calX$ with size $m$ independent of $N$.
This leads to linear $\mathcal{O}(N)$ scaling of the cost to evaluate the down-sampled versions of the objectives given by Eqs.~\textbf{\ref{eqn:max_sep_objective}}~and~\textbf{\ref{eqn:Lipschitz_constraint_function}} in Sections~\ref{subsec:error_tol}~and~\ref{subsec:amplification_tol} to achieve these relaxed guarantees.

The strong guarantee of Proposition~\ref{prop:error_tol_separation} requires that we use an objective like Eq.~\textbf{\ref{eqn:max_sep_objective}} in the submodular set-cover problem Eq.~\textbf{\ref{eqn:set_cover}} where the sum in Eq.~\textbf{\ref{eqn:max_sep_objective}} is taken over $\calX_N\times\calX_N$ and $\calX_N$ is an $\varepsilon_0$-net of the underlying set~$\calX$.
The problem is that the $\varepsilon_0$-net $\calX_N$ may be quite large and the number of operations needed to evaluate the sum in the objective scales with the square of the size of $\calX_N$.
Here we will prove that a similar guarantee as in Proposition~\ref{prop:error_tol_separation} holds with high probability over a large subset of $\calX$ when the sum in Eq.~\textbf{\ref{eqn:max_sep_objective}} is taken over secants between a randomly chosen collection of base points $\mathcal{B}_m = \lbrace \vect{b}_1,\ldots,\vect{b}_m\rbrace \subseteq\calX$ and the $\varepsilon_0$-net $\calX_N$.
Most importantly, the number of base points depends on the quality of the guarantee and not on size of the $\varepsilon_0$-net, so that the computational cost can be reduced to linear dependence on the size of $\calX_N$.

Specifically, in place of Eq.~\textbf{\ref{eqn:max_sep_objective}}, we can consider the sampled objective
\begin{equation}
    f_{\gamma,\varepsilon,m}(\scrS) = \frac{1}{m N}\sum_{\substack{1\leq i \leq m, \ 1\leq j\leq N :\\
    \Vert \vect{g}(\vect{b}_i) - \vect{g}(\xx_j)\Vert_2 \geq \varepsilon}}
    w_{\gamma,\vect{b}_i,\xx_j}(\scrS) \Vert \vect{g}(\vect{b}_i) - \vect{g}(\xx_j)\Vert_2^2
    \label{eqn:sampled_net_error_tolerance_objective}
\end{equation}
with $w_{\gamma,\vect{b}_i,\xx_j}(\scrS)$ defined by Eq.~\textbf{\ref{eqn:relaxed_weight_function}} in the optimization problem Eq.~\textbf{\ref{eqn:set_cover}}.
The greedy approximation algorithm produces a set of sensors $\scrS_K$ such that
\begin{equation}
    \Vert \vect{g}(\vect{b}_i) - \vect{g}(\xx_j) \Vert_2 \geq \varepsilon
    \quad \Rightarrow \quad
    \Vert \vect{m}_{\scrS_K}(\vect{b}_i) - \vect{m}_{\scrS_K}(\xx_j) \Vert_2 \geq \gamma
    \label{eqn:sampled_net_separation_condition}
\end{equation}
for every $\vect{b}_i\in\mathcal{B}_m$ and $\xx_j\in\calX_N$.
Theorem~\ref{thm:minimal_sensing_sampled_error_tolerance} guarantees that with high probability, only a small subset of points in $\calX$ have target values that cannot be distinguished from the rest by measurements separated by a relaxed detection threshold.
This size of this ``bad set'' is determined by its $\mu$-measure, which can be made arbitrarily small with high probability by taking more sample base points $m$.

\begin{theorem}[Sampled Separation Guarantee]
\label{thm:minimal_sensing_sampled_error_tolerance}
Let $\calX_N$ be an $\varepsilon_0$-net of $\calX$ and let the base points $\mathcal{B}_m$ be sampled independently according to a probability measure $\mu$ on $\calX$ with 
\begin{equation}
    m \geq \frac{1}{2\delta^2}\left( \#(\scrM) \ln{2} - \ln{p} \right),
\end{equation}
where $p,\delta \in (0, 1)$.
Consider the objective $f_{\gamma,\varepsilon,m}$ given by Eq.~\textbf{\ref{eqn:sampled_net_error_tolerance_objective}} for a certain choice of $\gamma > 0$ and $\varepsilon > 0$ for which every $\vect{b}_i\in\mathcal{B}_m$ and $\xx_j\in\calX_N$ satisfies
\begin{equation}
    \Vert \vect{g}(\vect{b}_i) - \vect{g}(\xx_j) \Vert_2 \geq \varepsilon
    \quad \Rightarrow \quad
    \Vert \vect{m}_{\scrM}(\vect{b}_i) - \vect{m}_{\scrM}(\xx_j) \Vert_2 \geq \gamma.
\end{equation}
Suppose also that $\vect{g}$ and the measurement functions $\vect{m}_k$, $k\in\scrM$ are all Lipschitz over $\calX$.
If $f_{\gamma,\varepsilon,m}(\scrS) = f_{\gamma,\varepsilon,m}(\scrM)$, then the $\mu$ measure of points $\xx\in\calX$ such that
\begin{multline}
    \Vert \vect{g}(\xx) - \vect{g}(\xx')\Vert_2 \geq \varepsilon + \varepsilon_0\Vert\vect{g}\Vert_{\text{lip}} \\
    \Rightarrow \quad
    \Vert \mS(\xx) - \mS(\xx')\Vert_2 > \gamma - \varepsilon_0\Vert \mS\Vert_{\text{lip}}
\end{multline}
for every $\xx'\in\calX$ is at least $1-\delta$ with probability at least $1-p$.
\end{theorem}
\begin{proof}
For simplicity, we will drop $\gamma,\varepsilon$ from the subscript on our objective since $\gamma$ and $\varepsilon$ remain fixed throughout the proof.
Let us begin by fixing a set $\scrS\subseteq\scrM$ and define the random variables
\begin{equation}
    Z_{\scrS}(\vect{b}_i) = \max_{\xx\in\calX_N} \bbone \big\lbrace \Vert \mS(\vect{b}_i) - \mS(\xx)\Vert_2 < \gamma \quad \mbox{and} \quad
    \Vert \vect{g}(\vect{b}_i) - \vect{g}(\xx) \Vert_2 \geq \varepsilon \big\rbrace.
\end{equation}
If $Z_{\scrS}(\vect{b}_i) = 0$ then every $\xx\in\calX_N$ with $\Vert \vect{g}(\vect{b}_i) - \vect{g}(\xx) \Vert_2 \geq \varepsilon$ also satisfies $\Vert \mS(\vect{b}_i) - \mS(\xx)\Vert_2 \geq \gamma$, otherwise $Z_{\scrS}(\vect{b}_i) = 1$.
We observe that $Z_{\scrS}(\vect{b}_i)$, $i=1,\ldots,m$ are independent, identically distributed Bernoulli random variables whose expectation
\begin{multline}
    \mathbb{E}\left[ Z_{\scrS}(\vect{b}_i) \right] = \mu\big(\big\lbrace \xx\in\calX\ :\ \exists\xx'\in\calX_N \quad \mbox{s.t.} \\
    \Vert \mS(\xx) - \mS(\xx')\Vert_2 < \gamma, \quad \Vert \vect{g}(\xx) - \vect{g}(\xx') \Vert_2 \geq \varepsilon \big\rbrace\big)
\end{multline}
is the $\mu$-measure of points in $\calX$ for which target values differing by at least $\varepsilon$ with points of $\calX_N$ are separated by measurements differing by less than $\gamma$.
Suppose that for a fixed $\xx\in\calX$ we have
\begin{equation}
    \Vert \vect{g}(\xx) - \vect{g}(\xx_j) \Vert_2 \geq \varepsilon 
    \quad \Rightarrow \quad
    \Vert \mS(\xx) - \mS(\xx_j)\Vert_2 \geq \gamma
    \label{eqn:thm_sep_assumption}
\end{equation}
for every $\xx_j\in\calX_N$.
For any $\xx'\in\calX$, there is an $\xx_j\in\calX_N$ with $\Vert \xx' - \xx_j\Vert_2 < \varepsilon_0$ and so we have
\begin{equation}
\begin{split}
    \varepsilon + \varepsilon_0 \Vert \vect{g}\Vert_{\text{lip}} &\leq \Vert \vect{g}(\xx) - \vect{g}(\xx') \Vert_2 \\
    &\leq \Vert \vect{g}(\xx) - \vect{g}(\xx_j) \Vert_2 + \Vert \vect{g}(\xx_j) - \vect{g}(\xx') \Vert_2 \\
    &< \Vert\vect{g}(\xx) - \vect{g}(\xx_j) \Vert_2 + \varepsilon_0 \Vert \vect{g}\Vert_{\text{lip}}.
\end{split}
\end{equation}
Hence, $\varepsilon \leq \Vert \vect{g}(\xx) - \vect{g}(\xx_j) \Vert_2$, which implies that $\gamma \leq \Vert \mS(\xx) - \mS(\xx_j) \Vert_2$ by assumption.
From this we obtain
\begin{equation}
\begin{split}
    \gamma &\leq \Vert \mS(\xx) - \mS(\xx') \Vert_2 + \Vert \mS(\xx') - \mS(\xx_j) \Vert_2 \\
    &< \Vert \mS(\xx) - \mS(\xx') \Vert_2 + \varepsilon_0 \Vert \mS\Vert_{\text{lip}}.
\end{split}
\end{equation}
Therefore, for such an $\xx\in\calX$ we have
\begin{multline}
    \forall \xx'\in\calX \qquad \Vert \vect{g}(\xx) - \vect{g}(\xx') \Vert_2 \geq \varepsilon + \varepsilon_0 \Vert \vect{g}\Vert_{\text{lip}} \\
    \Rightarrow \quad
    \Vert \mS(\xx) - \mS(\xx') \Vert_2 > \gamma - \varepsilon_0 \Vert \mS\Vert_{\text{lip}}.
\end{multline}
It follows that $\mathbb{E}\left[ Z_{\scrS}(\vect{b}_i) \right]$ is an upper bound on the $\mu$-measure of points in $\calX$ for which there is another point in $\calX$ with a close measurement and distant target value, that is
\begin{multline}
    \mathbb{E}\left[ Z_{\scrS}(\vect{b}_i) \right] \geq \mu\big(\big\lbrace \xx\in\calX\ :\ \exists\xx'\in\calX\quad \mbox{s.t.} \\
    \Vert \mS(\xx) - \mS(\xx')\Vert_2 \leq \gamma - \varepsilon_0 \Vert \mS\Vert_{\text{lip}}, \\
    \Vert \vect{g}(\xx) - \vect{g}(\xx') \Vert_2 \geq \varepsilon + \varepsilon_0 \Vert \vect{g}\Vert_{\text{lip}} \big\rbrace\big).
    \label{eqn:thm_measure_of_points_with_large_error}
\end{multline}

By assumption, we have a set $\scrS \subset \scrM$ so that $Z_{\scrS}(\vect{b}_i) = 0$ for each $i=1,\ldots,m$.
And so it remains to bound the difference between the empirical and true expectation of $Z_{\scrS}(\vect{b}_i)$ uniformly over every subset $\scrS \subset \scrM$.
For fixed $\scrS$, the one-sided Hoeffding inequality gives
\begin{equation}
    \mathbb{P}\Big\lbrace \frac{1}{m}\sum_{i=1}^m\left( \mathbb{E}[Z_{\scrS}(\vect{b}_i)] - Z_{\scrS}(\vect{b}_i) \right) \geq \delta \Big\rbrace \leq e^{-2 m \delta^2}.
\end{equation}
Unfixing $\scrS$ via the union bound over all $\scrS \subset \scrM$ and applying our assumption about the number of base points $m$ yields
\begin{multline}
    \mathbb{P}\bigcup_{\scrS\subseteq\scrM}\Big\lbrace \frac{1}{m}\sum_{i=1}^m\left( \mathbb{E}[Z_{\scrS}(\vect{b}_i)] - Z_{\scrS}(\vect{b}_i) \right) \geq \delta \Big\rbrace
    \leq \exp{\left[ \#(\scrM) \ln{2} - 2 m \delta^2\right]} \\
    \leq p.
\end{multline}
Since our assumed choice of $\scrS$ has $f_m(\scrS) = f_m(\scrM)$ it follows that all $Z_{\scrS}(\vect{b}_i) = 0$, $i=1,\ldots,m$, hence we have
\begin{equation}
    \mathbb{E}[Z_{\scrS}(\vect{b}_i)] < \delta
\end{equation}
with probability at least $1-p$.
Combining this with Eq.~\textbf{\ref{eqn:thm_measure_of_points_with_large_error}} completes the proof.
\end{proof}

It is also possible to use a down-sampled objective to greedily choose sensors that satisfy a similarly relaxed version of the amplification guarantee given by
Proposition~\ref{prop:amplification_tol_separation} with high probability over a large subset of $\calX$.
In order to do this, we take the sum in Eq.~\textbf{\ref{eqn:Lipschitz_constraint_function}} over secants between a randomly chosen collection of base points $\mathcal{B}_m = \lbrace \vect{b}_1,\ldots,\vect{b}_m\rbrace \subseteq\calX$ and the $\varepsilon_0$-net $\calX_N$.
Again, the number of base points depends on the quality of the guarantee and not on size of the $\varepsilon_0$-net, so that the computational cost can be reduced to linear dependence on the size of $\calX_N$.

Specifically, in place of Eq.~\textbf{\ref{eqn:Lipschitz_constraint_function}}, we consider
\begin{equation}
    f_{L,m}(\scrS) = \sum_{\substack{1\leq i \leq m,\ 1\leq j\leq N, \\ \vect{g}(\vect{b}_i) \neq \vect{g}(\xx_j)}}
    \min\left\lbrace \frac{\Vert \mS(\vect{b}_i) - \mS(\xx_j) \Vert_2^2 }
    {\Vert \vect{g}(\vect{b}_i) - \vect{g}(\xx_j) \Vert_2^2},\ 
    \frac{1}{L^2} \right\rbrace.
    \label{eqn:sampled_net_Lipschitz_objective}
\end{equation}
In Theorem~\ref{thm:minimal_sensing_sampled_amplification_tolerance} we show that when a sufficiently small set of sensors $\scrS$ is found, e.g., using the greedy algorithm with the sampled objective $f_{L,m}$, that satisfies the amplification tolerance over $\mathcal{B}_m\times\calX_N$, we can conclude that that a slightly relaxed amplification bound holds with high probability over a large subset of $\calX$.
In particular, the subset of ``bad points'' in $\xx\in\calX$ for which there is another point $\xx'\in\calX$ with a different target value, but not a sufficiently different measured value, has small $\mu$-measure with high probability.

\begin{theorem}[Sampled Amplification Guarantee]
\label{thm:minimal_sensing_sampled_amplification_tolerance}
Let $\calX_N$ be an $\varepsilon_0$-net of $\calX$ and let the base points $\mathcal{B}_m$ be sampled independently according to a probability measure $\mu$ on $\calX$ with
\begin{equation}
    m \geq \frac{1}{2\delta^2}\left( \#(\scrM)\ln{2} - \ln{p} \right).
\end{equation}
Consider the objective $f_m$ given by Eq.~\textbf{\ref{eqn:sampled_net_Lipschitz_objective}} for a certain choice of $L>0$ for which
\begin{equation}
    \Vert \vect{g}(\vect{b}_i) - \vect{g}(\xx_j) \Vert_2 \leq L \Vert \vect{m}_{\scrM}(\vect{b}_i) - \vect{m}_{\scrM}(\xx_j) \Vert_2
\end{equation}
is achieved for all $\vect{b}_i\in\mathcal{B}_m$, $\xx_j\in\calX_N$.
Suppose also that $\vect{g}$ and the measurement functions $\vect{m}_k$, $k\in\scrM$ are all Lipschitz functions over $\calX$.
If $f_{L,m}(\scrS) = f_{L,m}(\scrM)$, then the $\mu$-measure of points $\xx\in\calX$ such that
\begin{equation}
    \Vert \vect{g}(\xx) - \vect{g}(\xx') \Vert_2 <
    L \Vert \mS(\xx) + \mS(\xx') \Vert_2
    + \left( \Vert \vect{g}\Vert_{\text{lip}} + L \Vert \mS \Vert_{\text{lip}} \right)\varepsilon_0
\end{equation}
for every $\xx'\in\calX$ is at least $1-\delta$ with probability at least $1-p$.
\end{theorem}
\begin{proof}
The proof is analogous to Theorem~\ref{thm:minimal_sensing_sampled_error_tolerance} and so we relegate it to Appendix~\ref{app:proofs}.
\end{proof}

\section{Working with Noisy Data}
\label{sec:noisy_data}
So far, we have considered maximizing different measures of robust reconstructability given a collection of noiseless data.
That is, the resulting sensors are selected in order to be noise robust, but we have assumed that the measurements $\vect{m}_j(\vect{x}_i)$, $j\in\scrM$ and target variables $\vect{g}(\vect{x}_i)$ used during the sensor selection process are noiseless over the sampled states $\vect{x}_i\in\calX_N$.
In many applications, however, our data may contain noisy measurements, target variables, or both.
In this section, we study the effect of noisy data on the performance of our proposed secant-based greedy algorithms.
By ``noise'' we mean specifically that we are given a collection of available measurements $\left\lbrace \vect{\tilde{m}}_{i,\scrM} = \vect{m}_{\scrM}(\vect{x}_i) + \vect{u}_{i,\scrM} \right\rbrace_{i=1}^N$ that are corrupted by unknown noise $\vect{u}_{i,\scrM}$ together with the corresponding target values $\left\lbrace \vect{\tilde{g}}_i = \vect{g}(\vect{x}_i) + \vect{v}_i \right\rbrace_{i=1}^N$ that are also corrupted by unknown noise $\vect{v}_i$.
That is, we do not have access to the measurement functions $\vect{m}_{\scrM}$ or the target function $\vect{g}$ and must rely solely on noisy data generated by them.

First, we mention that the minimal sensing method to meet an error tolerance discussed in Section~\ref{subsec:error_tol} is robust to bounded noise in the measurements and target variables.
In particular, since the selected sensors $\scrS$ using the approach described in Section~\ref{subsec:error_tol} automatically satisfy Eq.~\textbf{\ref{eqn:noisy_separation_condition}}, Proposition~\ref{prop:noisy_separation_guarantee}, below, shows that the true measurements coming from states with sufficiently distant true target values must also be separated by the measurements.
\begin{proposition}[Noisy Separation Guarantee]
\label{prop:noisy_separation_guarantee}
Let $\calX_N$ be an $\varepsilon_0$-net of $\calX$ (see Definition~\ref{def:epsilon_net})
and let $\vect{v}_i\in \mathbb{R}^{\dim\vect{g}}$, $\vect{u}_{i,\scrS}\in\mathbb{R}^{d_{\scrS}}$, $i=1,\ldots,N$ be bounded vectors with
\begin{equation}
    \forall i=1,\ldots,N \qquad \left\Vert \vect{u}_{i,\scrS}\right\Vert_2 \leq \delta_{u}, \qquad \left\Vert \vect{v}_i\right\Vert_2 \leq \delta_{v}.
\end{equation}
Suppose that there exists $\epsilon > 0$ and $\gamma > 0$ such that
\begin{multline}
    \forall \xx_i,\xx_j\in\calX_N \qquad
    \Vert \left(\vect{g}(\xx_i) + \vect{v}_i\right) - \left(\vect{g}(\xx_j) + \vect{v}_j\right) \Vert_2 \geq \varepsilon \\
    \Rightarrow \quad \Vert \left(\mS(\xx_i) + \vect{u}_{i,\scrS}\right) - \left(\mS(\xx_j) + \vect{u}_{j,\scrS}\right) \Vert_2 \geq \gamma.
    \label{eqn:noisy_separation_condition}
\end{multline}
If $\mS$ and $\vect{g}$ are Lipschitz functions with Lipschitz constants $\Vert \mS\Vert_{\text{lip}}$ and $\Vert \vect{g}\Vert_{\text{lip}}$ respectively, then
\begin{multline}
    \forall \xx,\xx'\in\calX \qquad
    \Vert \vect{g}(\xx) - \vect{g}(\xx') \Vert_2 \geq \varepsilon + 2\delta_{v} + 2\varepsilon_0 \Vert \vect{g}\Vert_{\text{lip}} \\
    \quad \Rightarrow \quad 
    \Vert \mS(\xx) - \mS(\xx') \Vert_2 > \gamma - 2\delta_{u} - 2\varepsilon_0 \Vert \mS \Vert_{\text{lip}}.
\end{multline}
\end{proposition}
\begin{proof}
The proof is analogous to Proposition~\ref{prop:error_tol_separation} and has been relegated to Appendix~\ref{app:proofs}.
\end{proof}
As a consequence of Proposition~\ref{prop:noisy_separation_guarantee}, the reconstruction error for the desired quantities using these sensors can still be bounded if the thresholds $\varepsilon$ and $\gamma$ exceed twice the noise level of the target variable and measurements respectively (with a little extra padding based on the sampling fineness).

On the other hand, the minimal sensing method to meet an amplification tolerance discussed in Section~\ref{subsec:amplification_tol} is very sensitive to noisy data.
This is because measurement noise can bring two nearby measurements $\mS(\vect{x})$ and $\mS(\vect{x}')$ arbitrarily close together while the corresponding target variables $\vect{g}(\vect{x})$ and $\vect{g}(\vect{x}')$ remain separated.
Such terms can result in arbitrarily large data-driven estimates of the reconstruction Lipschitz constant.
Consequently it may not be possible to find a small set of sensors $\scrS$ such that
\begin{equation}
    \max_{1\leq i<j\leq N} \frac{\left\Vert \vect{\tilde{g}}_i - \vect{\tilde{g}}_j \right\Vert_2}{\left\Vert \vect{\tilde{m}}_{i,\scrS} - \vect{\tilde{m}}_{j,\scrS} \right\Vert_2} \leq L
\label{eqn:noisy_Lip_condition}
\end{equation}
for acceptable values of $L$.

One way to deal with this problem is to smooth out the target variables.
For instance, given the available noisy measurement and target pairs $\left\lbrace\left( \vect{\tilde{m}}_{i,\scrM},\ \vect{\tilde{g}}_i \right) \right\rbrace_{i=1}^N$, one can find an approximation of the reconstruction function $\boldsymbol{\Phi}_{\scrM}$ via regression.
Using the predicted target variables
\begin{equation}
    \vect{\hat{g}}_i := \boldsymbol{\Phi}_{\scrM}\left(\vect{\tilde{m}}_{i,\scrM}\right)
\end{equation}
in place of the noisy data $\vect{\tilde{g}}_i$ fixes the problem of infinite Lipschitz constants.
This is because the amplification-based approach using these data seeks to find the minimal set of sensors $\scrS$ such that
\begin{equation}
    \max_{1\leq i<j\leq N} \frac{\left\Vert \vect{\hat{g}}_i - \vect{\hat{g}}_j \right\Vert_2}{\left\Vert \vect{\tilde{m}}_{i,\scrS} - \vect{\tilde{m}}_{j,\scrS} \right\Vert_2} \leq L
\label{eqn:smoothed_Lip_condition}
\end{equation}
rather than satisfying Eq.~\textbf{\ref{eqn:noisy_Lip_condition}}.

We use a similar type of smoothing approach for the shock-mixing layer problem by choosing the leading two Isomap coordinates $\vect{g}(\vect{x}) = \left( \phi_1(\vect{x}),\ \phi_2(\vect{x}) \right)$ rather than simply taking $\vect{g}(\vect{x}) = \vect{x}$.
This is because the full state $\vect{x}$ contains some small noise, meaning that it does not lie exactly on the one-dimensional loop in state space, but rather on a very thin manifold with full dimensionality.
If we were to use the Lipschitz-based approach to reconstruct $\vect{x}$ directly, we would need enough sensors to reconstruct this noise.
By seeking to reconstruct the leading Isomap coordinates instead, we have regularized our selection algorithm to choose only those sensors that are needed to reconstruct the dominant periodic behavior.

Reconstructing smoothed target variables turns out to be a robust method for sensor placement, as we show by introducing increasing levels of noise in the shock-mixing layer problem.
We added independent Gaussian noise with standard deviations $\sigma_{\text{noise}} = 0.01$, $0.02$, $0.03$, $0.04$, and $0.05$ to each velocity component at every location on the computational grid, yielding noisy snapshots like the one shown in Figure~\ref{fig:sml_flowfields_noise}.
This reflects the typical situation when the underlying data given to us are noisy.
At each noise level we selected three sensors using the detectable difference-based method of Section~\ref{subsec:MSF} as well as the amplification tolerance-based method of Section~\ref{subsec:amplification_tol}, with a bisection search over the threshold Lipschitz constant~$L$, to reconstruct the leading two Isomap coordinates of the noisy data.
Despite the noise, the leading two Isomap coordinates continued to accurately capture the dominant periodic behavior of the underlying system, making them good reconstruction target variables.
The thresholds for the detectable difference method were fixed at $\gamma = 0.04$ except in the $\sigma_{\text{noise}}=0.02$ case, where better performance was achieved using $\gamma = 0.02$.
\begin{figure*}
    \centering
    \subfloat{
    \begin{tikzonimage}[trim=26.75 120 13 135, clip=true, width=0.45\textwidth]{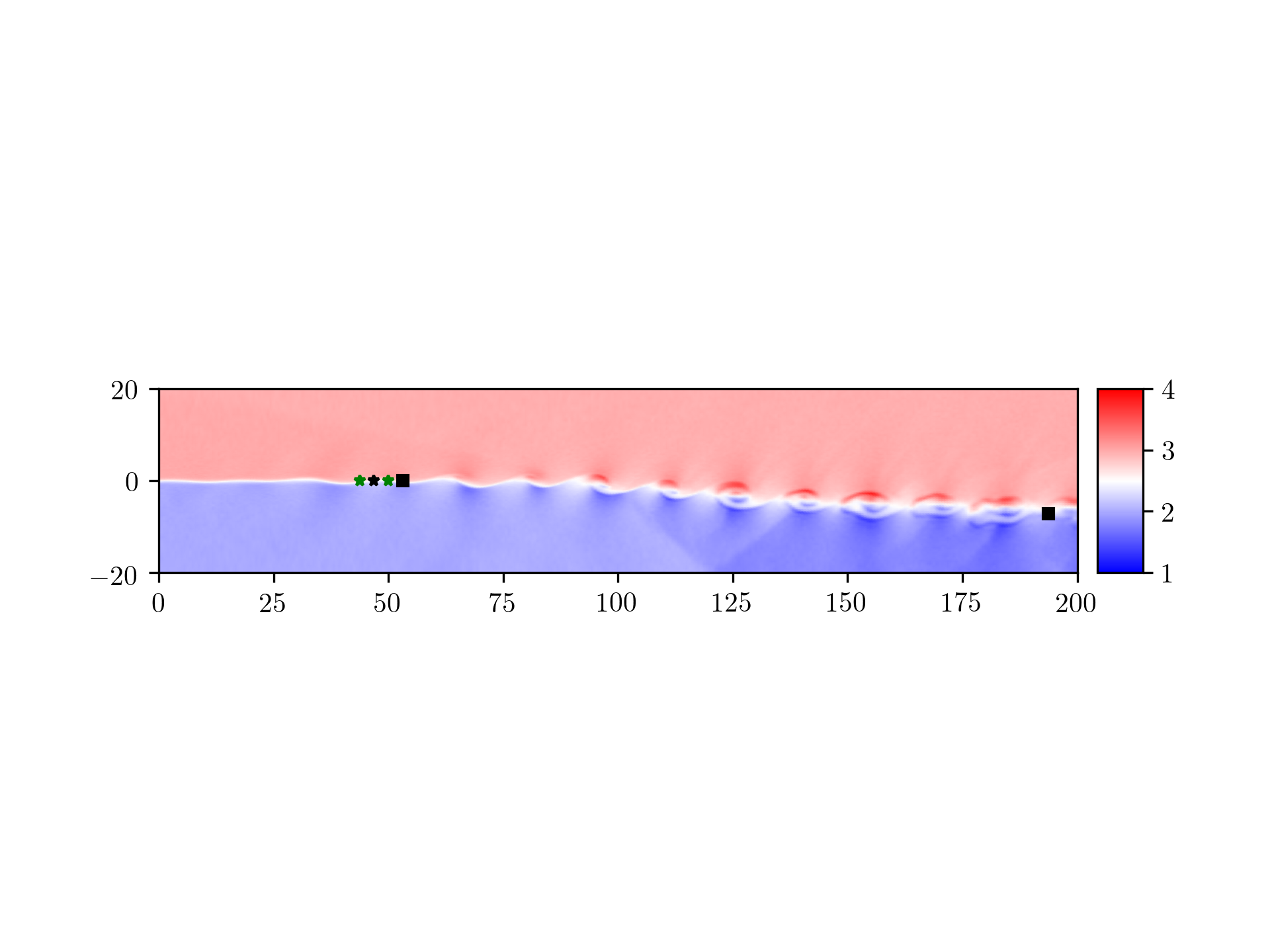}
    \end{tikzonimage}
    \label{subfig:u_snapshot_noise_0p01}
    }
    \subfloat{
    \begin{tikzonimage}[trim=30 120 10 135, clip=true, width=0.45\textwidth]{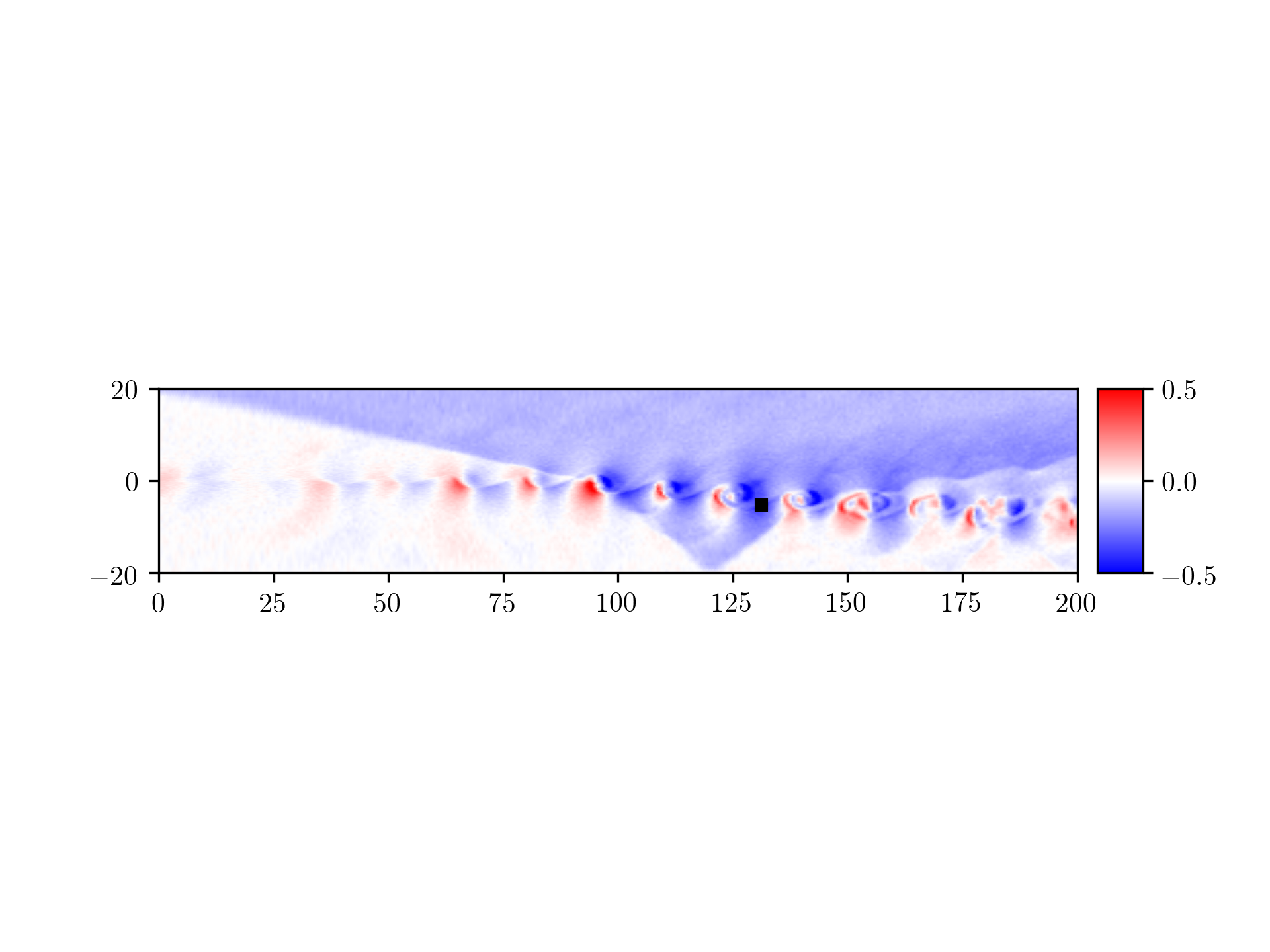}
    \end{tikzonimage}
    \label{subfig:v_snapshot_0p01}
    } \\
    \subfloat{
    \begin{tikzonimage}[trim=26.75 120 13 135, clip=true, width=0.45\textwidth]{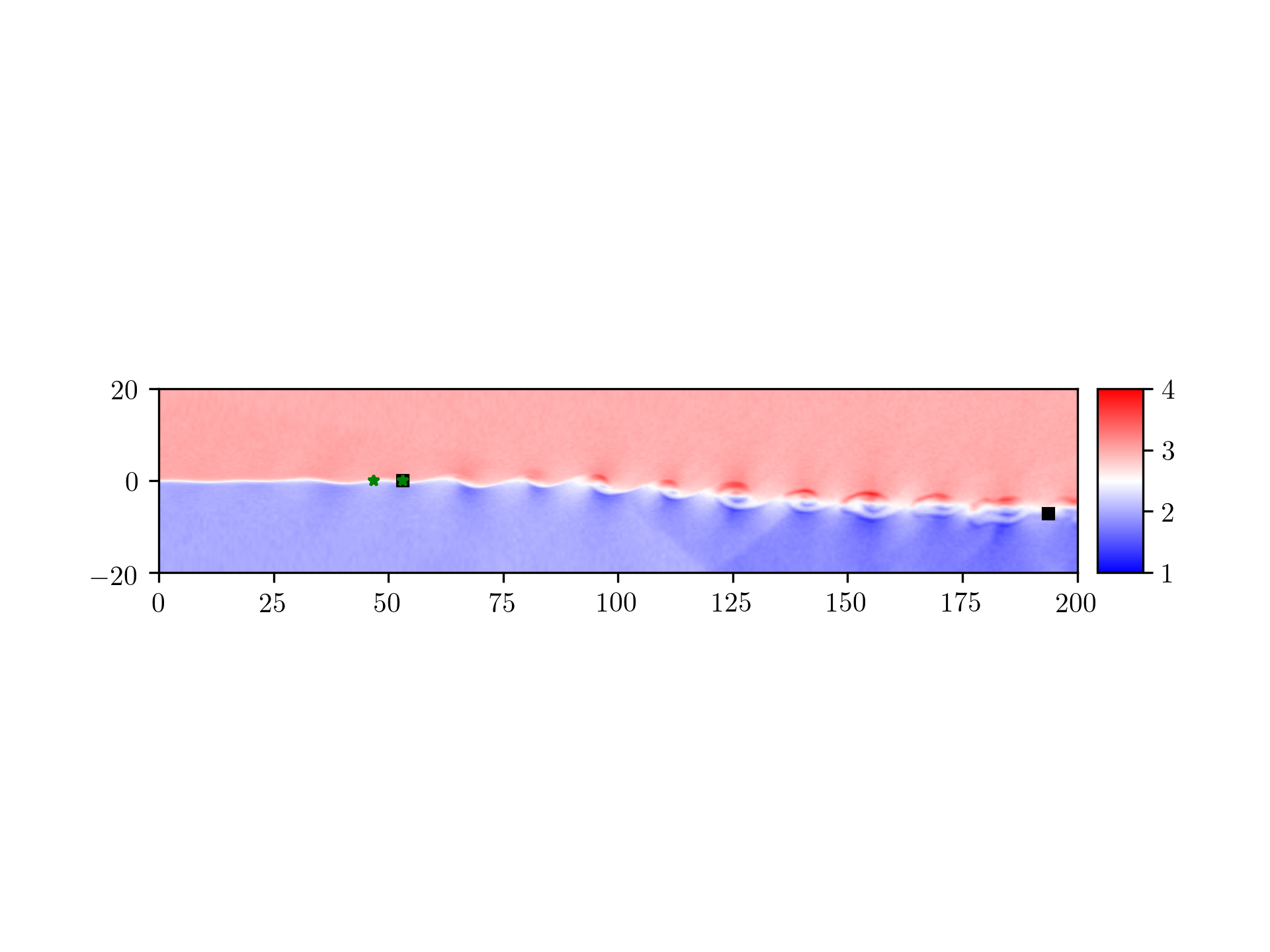}
    \end{tikzonimage}
    \label{subfig:u_snapshot_noise_0p02}
    }
    \subfloat{
    \begin{tikzonimage}[trim=30 120 10 135, clip=true, width=0.45\textwidth]{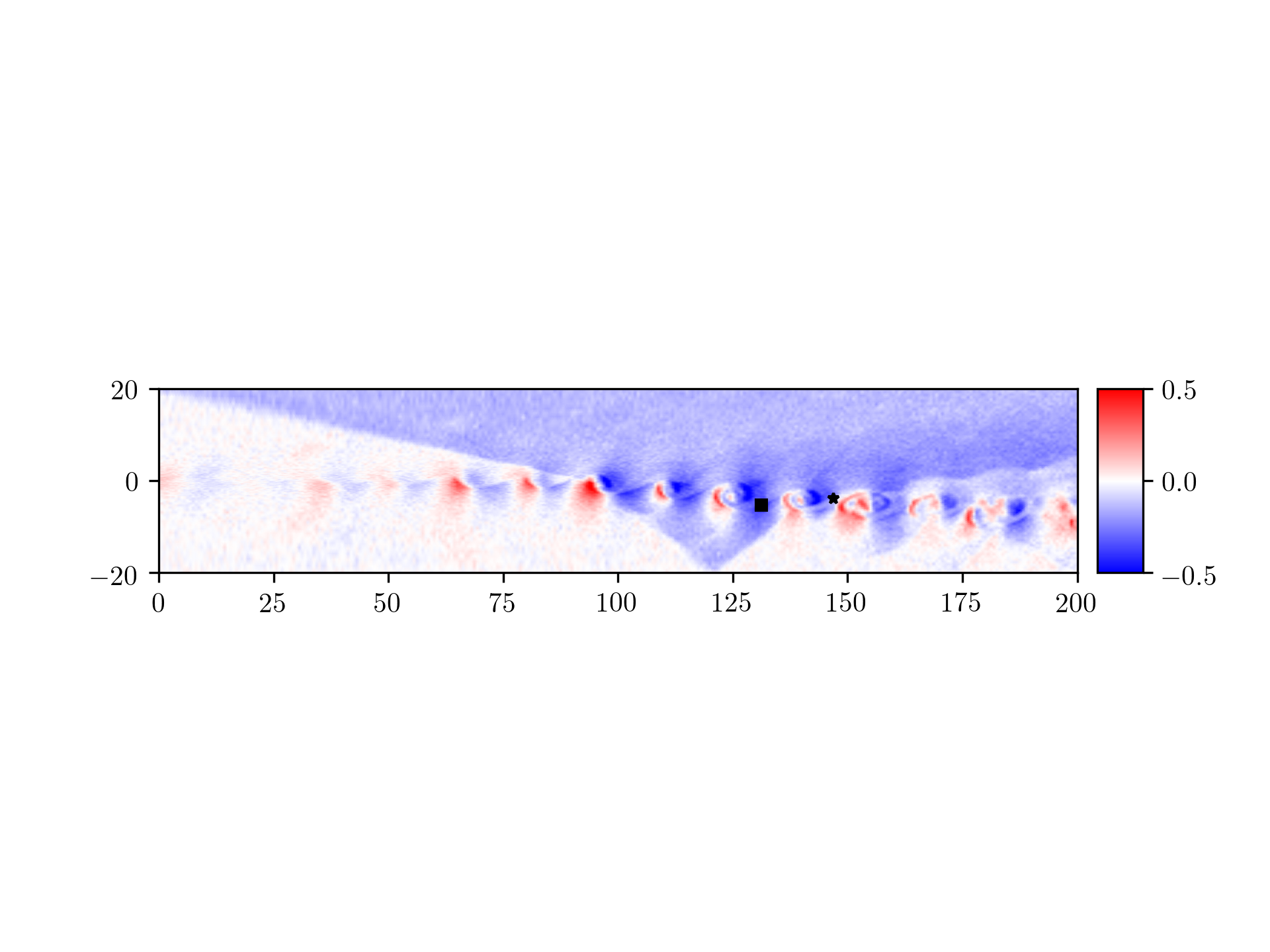}
    \end{tikzonimage}
    \label{subfig:v_snapshot_0p02}
    } \\
    \subfloat{
    \begin{tikzonimage}[trim=26.75 120 13 135, clip=true, width=0.45\textwidth]{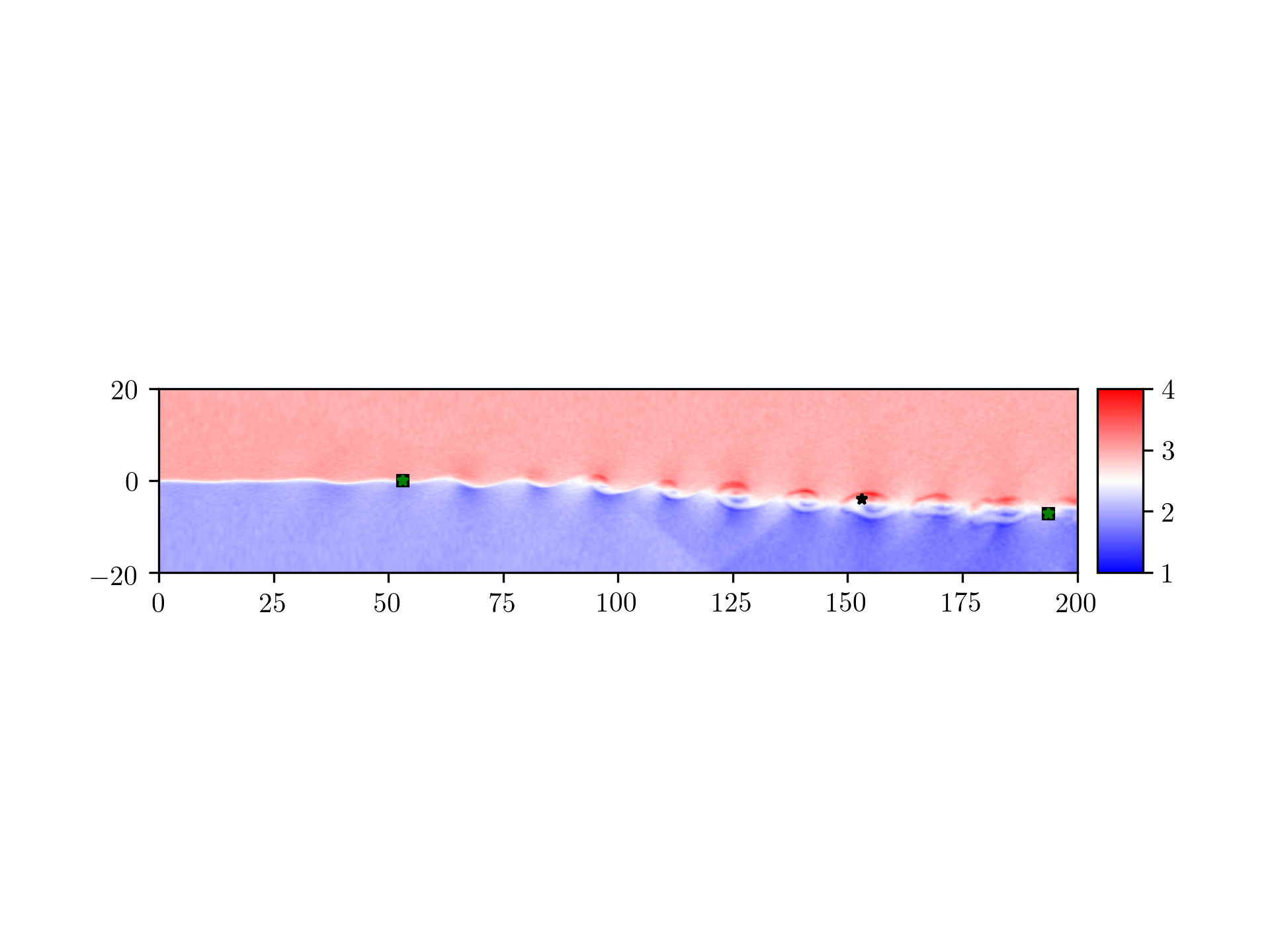}
    \end{tikzonimage}
    \label{subfig:u_snapshot_noise_0p03}
    }
    \subfloat{
    \begin{tikzonimage}[trim=30 120 10 135, clip=true, width=0.45\textwidth]{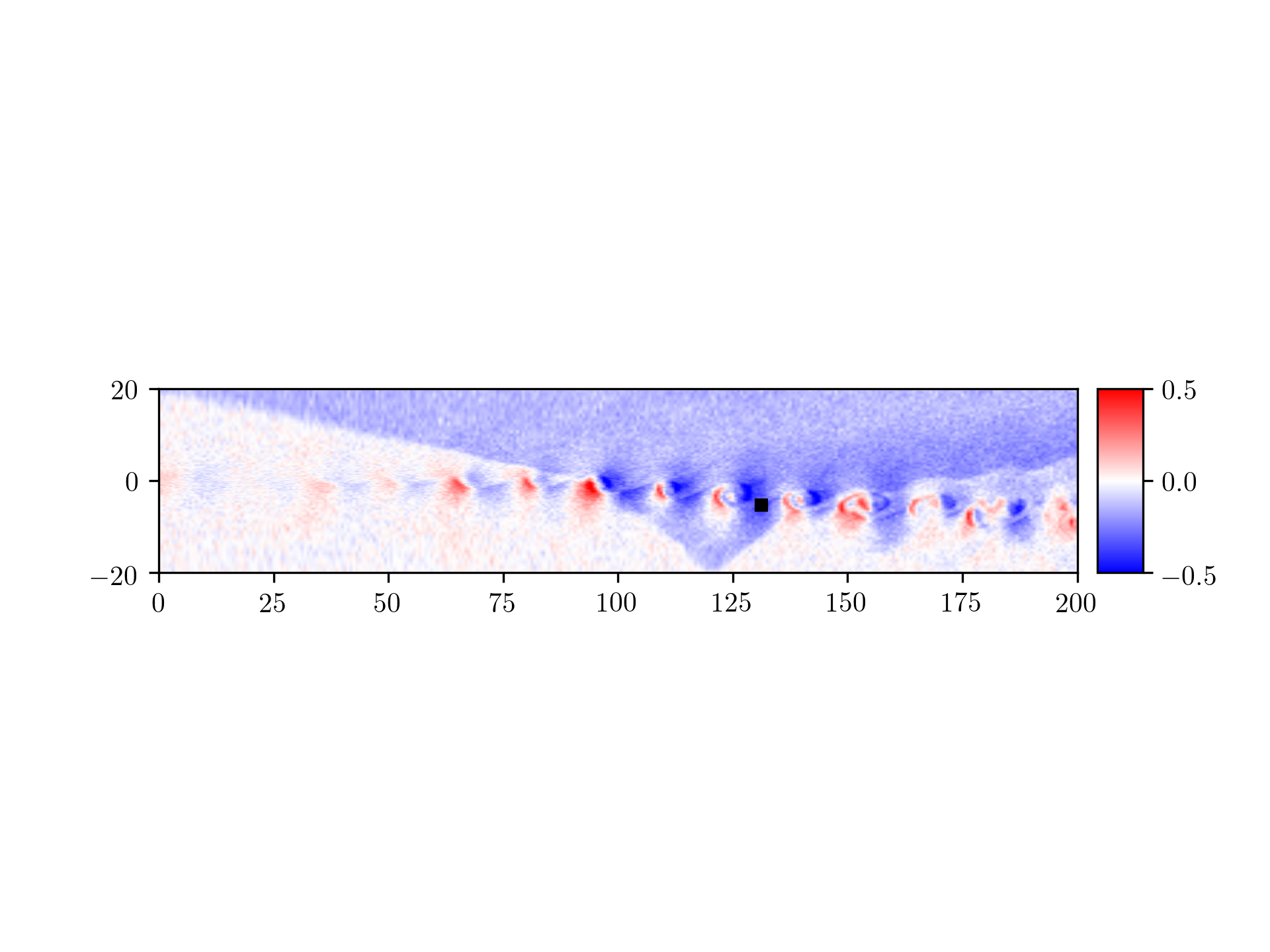}
    \end{tikzonimage}
    \label{subfig:v_snapshot_0p03}
    } \\
    \subfloat{
    \begin{tikzonimage}[trim=26.75 120 13 135, clip=true, width=0.45\textwidth]{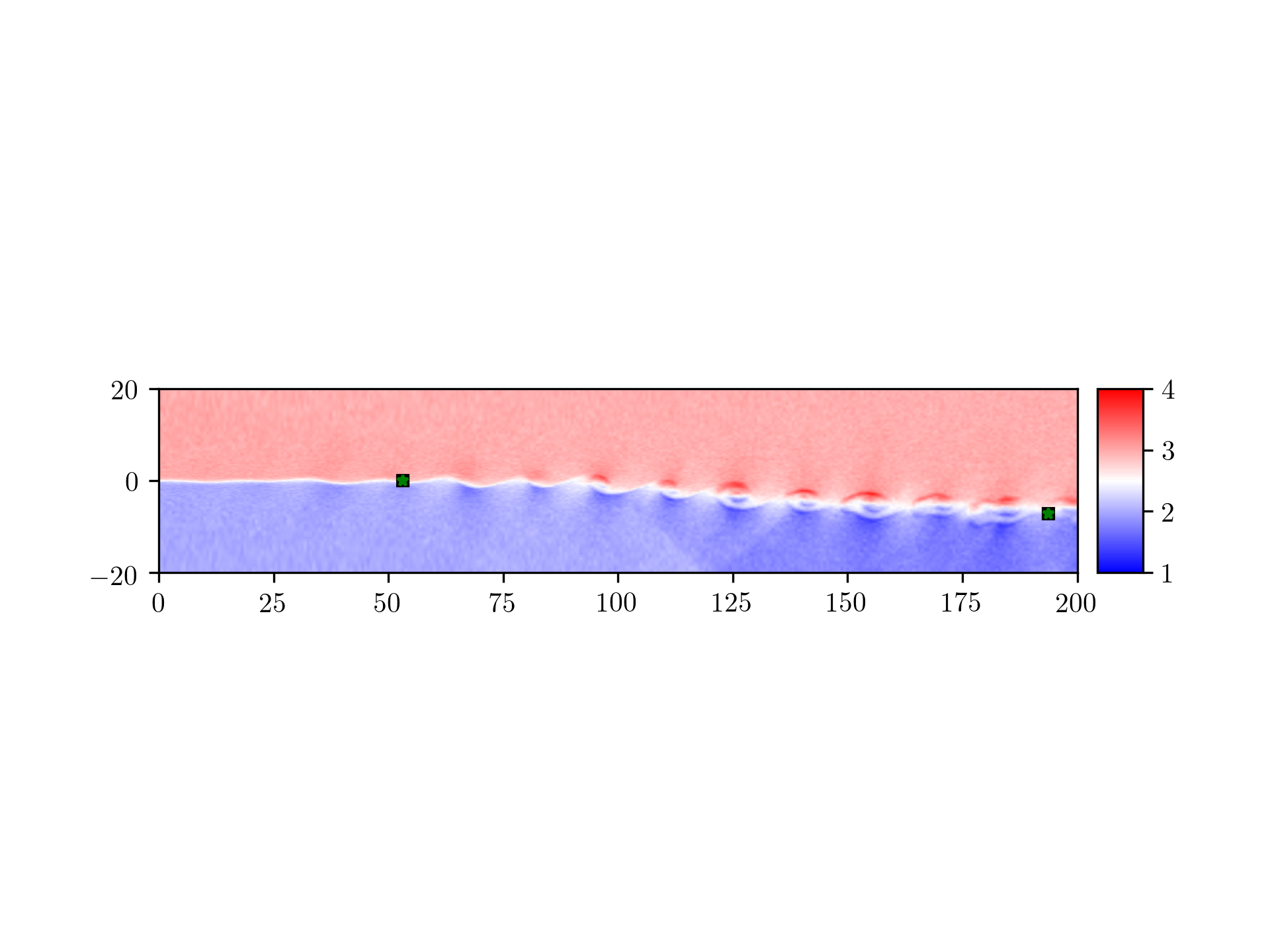}
    \end{tikzonimage}
    \label{subfig:u_snapshot_noise_0p04}
    }
    \subfloat{
    \begin{tikzonimage}[trim=30 120 10 135, clip=true, width=0.45\textwidth]{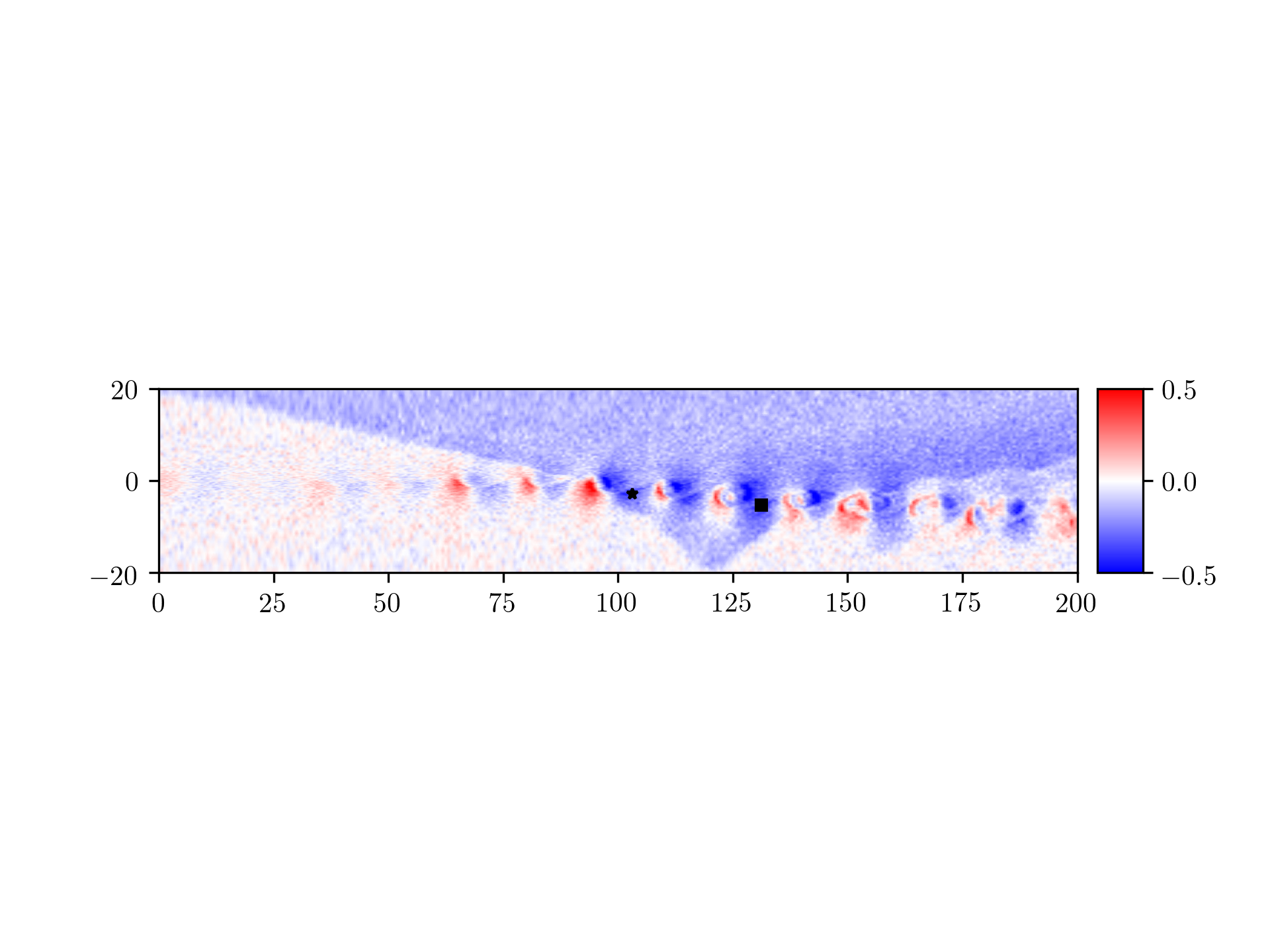}
    \end{tikzonimage}
    \label{subfig:v_snapshot_0p04}
    } \\
    \subfloat{
    \begin{tikzonimage}[trim=26.75 120 13 135, clip=true, width=0.45\textwidth]{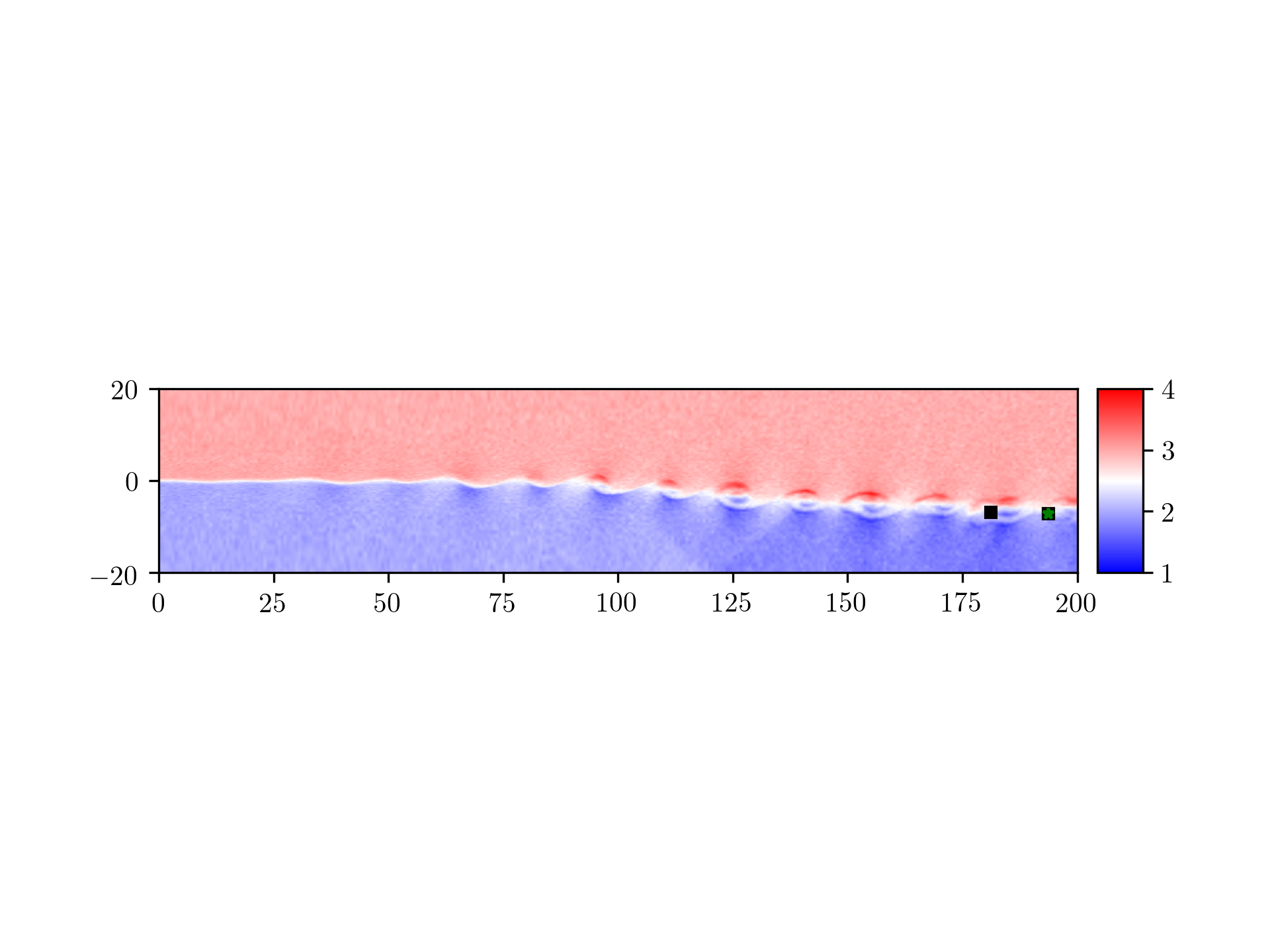}
    \end{tikzonimage}
    \label{subfig:u_snapshot_noise_0p05}
    }
    \subfloat{
    \begin{tikzonimage}[trim=30 120 10 135, clip=true, width=0.45\textwidth]{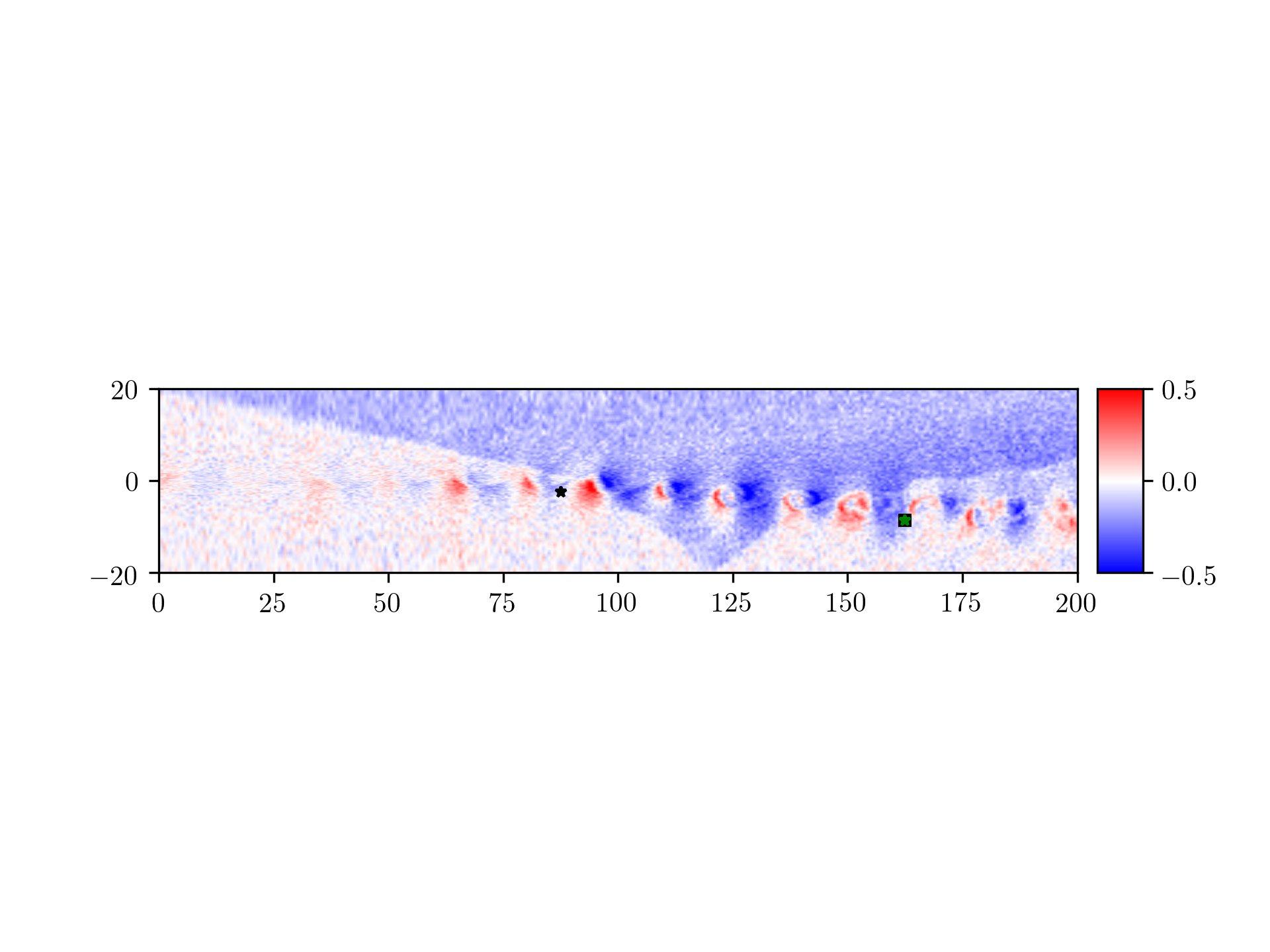}
    \end{tikzonimage}
    \label{subfig:v_snapshot_0p05}
    }
    \caption{We show the stream-wise (first column) and transverse (second column) components of velocity for a single snapshot of the shock-mixing layer flow with increasing levels of noise added in each successive row. Independent Gaussian noise with standard deviations $\sigma_{\text{noise}}=0.01$, $0.02$, $0.03$, $0.04$, and $0.05$ are added to each velocity component at each location on the computational grid.
    The first two sensors chosen by detectable difference method of Section~\ref{subsec:MSF} are indicated by green stars and the third is indicated by a black star. The three sensors selected using the amplification tolerance method of Section~\ref{subsec:amplification_tol} with bisection search over $L$ are indicated by black squares.}
    \label{fig:sml_flowfields_noise}
\end{figure*}

We found that the amplification tolerance-based method identified the same sensors across each of the first four noise levels $\sigma_{\text{noise}} = 0.01$, $0.02$, $0.03$, and $0.04$.
While these sensor locations differed slightly from the ones selected without noise (shown in Figure~\ref{fig:shock_mixing_layer_snapshots}), they too were capable of robustly recovering the underlying phase of the system as illustrated by their corresponding measurements in the third column of Figure~\ref{fig:sml_measurements_noise}.
At the largest noise level $\sigma_{\text{noise}}=0.5$, the sensors selected using this method changed, but were still capable of revealing the phase as shown in the bottom right plot of Figure~\ref{fig:sml_measurements_noise}.
The detectable difference-based method selected the same three sensors as in the zero noise case when $\sigma_{\text{noise}} = 0.01$ with the first two remaining the same up to $\sigma_{\text{noise}}=0.02$.
At these noise levels the first two sensors are sufficient to reveal the underlying phase of the system as shown in the first two plots in the first column of Figure~\ref{fig:sml_measurements_noise}.
Beyond this level of noise, the first two sensors were no longer able to reveal the phase as illustrated by the self-intersections in the last three plots in the first column of Figure~\ref{fig:sml_measurements_noise}.
While it is admittedly difficult to see from the last three plots in the middle column of Figure~\ref{fig:sml_measurements_noise}, the third sensor eliminated these self-intersections by raising one of the two intersecting branches and allowing the phase to be determined.
\begin{figure*}
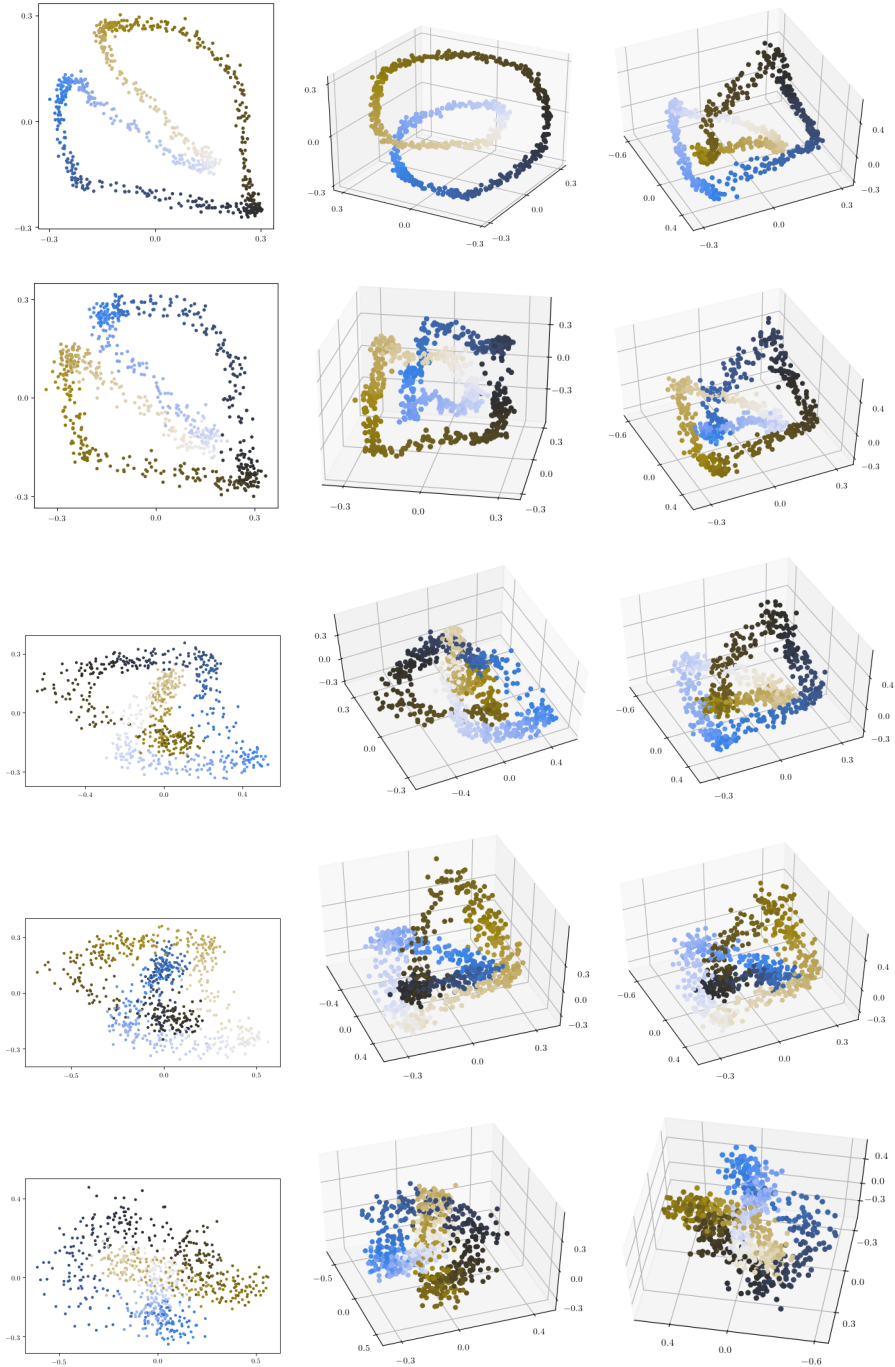

    \centering
    \subfloat{
    \begin{tikzonimage}[trim=70 15 75 35, clip=true, width=0.28\textwidth]{Figures/ShockMixingLayer/noise_0p01/MDD_greedy2_data.PNG}
    \end{tikzonimage}
    \label{subfig:secant_MSE_data2_0p01}
    }
    \subfloat{
    \begin{tikzonimage}[trim=65 30 65 50, clip=true, width=0.3\textwidth]{Figures/ShockMixingLayer/noise_0p01/MDD_greedy3_data.PNG}
    \end{tikzonimage}
    \label{subfig:secant_MSE_data3_0p01}
    }
    \subfloat{
    \begin{tikzonimage}[trim=90 30 50 40, clip=true, width=0.3\textwidth]{Figures/ShockMixingLayer/noise_0p01/secant_Lip_data.PNG}
    \end{tikzonimage}
    \label{subfig:secant_Lip_data_0p01}
    } \\
    \subfloat{
    \begin{tikzonimage}[trim=70 15 75 35, clip=true, width=0.28\textwidth]{Figures/ShockMixingLayer/noise_0p02/MDD_greedy2_data.PNG}
    \end{tikzonimage}
    \label{subfig:secant_MSE_data2_0p02}
    }
    \subfloat{
    \begin{tikzonimage}[trim=100 35 70 55, clip=true, width=0.3\textwidth]{Figures/ShockMixingLayer/noise_0p02/MDD_greedy3_data.PNG}
    \end{tikzonimage}
    \label{subfig:secant_MSE_data3_0p02}
    }
    \subfloat{
    \begin{tikzonimage}[trim=90 30 50 40, clip=true, width=0.3\textwidth]{Figures/ShockMixingLayer/noise_0p02/secant_Lip_data.PNG}
    \end{tikzonimage}
    \label{subfig:secant_Lip_data_0p02}
    } \\
    \subfloat{
    \begin{tikzonimage}[trim=30 45 30 55, clip=true, width=0.3\textwidth]{Figures/ShockMixingLayer/noise_0p03/MDD_greedy2_data.PNG}
    \end{tikzonimage}
    \label{subfig:secant_MSE_data2_0p03}
    }
    \subfloat{
    \begin{tikzonimage}[trim=75 40 70 40, clip=true, width=0.3\textwidth]{Figures/ShockMixingLayer/noise_0p03/MDD_greedy3_data.PNG}
    \end{tikzonimage}
    \label{subfig:secant_MSE_data3_0p03}
    }
    \subfloat{
    \begin{tikzonimage}[trim=95 30 50 45, clip=true, width=0.3\textwidth]{Figures/ShockMixingLayer/noise_0p03/secant_Lip_data.PNG}
    \end{tikzonimage}
    \label{subfig:secant_Lip_data_0p03}
    } \\
    \subfloat{
    \begin{tikzonimage}[trim=30 45 30 55, clip=true, width=0.3\textwidth]{Figures/ShockMixingLayer/noise_0p04/MDD_greedy2_data.PNG}
    \end{tikzonimage}
    \label{subfig:secant_MSE_data2_0p04}
    }
    \subfloat{
    \begin{tikzonimage}[trim=100 30 55 45, clip=true, width=0.3\textwidth]{Figures/ShockMixingLayer/noise_0p04/MDD_greedy3_data.PNG}
    \end{tikzonimage}
    \label{subfig:secant_MSE_data3_0p04}
    }
    \subfloat{
    \begin{tikzonimage}[trim=95 30 45 45, clip=true, width=0.3\textwidth]{Figures/ShockMixingLayer/noise_0p04/secant_Lip_data.PNG}
    \end{tikzonimage}
    \label{subfig:secant_Lip_data_0p04}
    } \\
    \subfloat{
    \begin{tikzonimage}[trim=30 30 30 40, clip=true, width=0.3\textwidth]{Figures/ShockMixingLayer/noise_0p05/MDD_greedy2_data.PNG}
    \end{tikzonimage}
    \label{subfig:secant_MSE_data2_0p05}
    }
    \subfloat{
    \begin{tikzonimage}[trim=100 30 50 45, clip=true, width=0.3\textwidth]{Figures/ShockMixingLayer/noise_0p05/MDD_greedy3_data.PNG}
    \end{tikzonimage}
    \label{subfig:secant_MSE_data3_0p05}
    }
    \subfloat{
    \begin{tikzonimage}[trim=105 40 70 50, clip=true, width=0.3\textwidth]{Figures/ShockMixingLayer/noise_0p05/secant_Lip_data.PNG}
    \end{tikzonimage}
    \label{subfig:secant_Lip_data_0p05}
    }
    \caption{These plots show the measurements made by sensors selected using the detectable difference method of Section~\ref{subsec:MSF} with two (first column) and three (second column) sensors along with the amplification tolerance method of Section~\ref{subsec:amplification_tol} with three sensors (third column) on the shock-mixing layer flow problem with various levels of added noise. 
    Each row shows the result of adding independent Gaussian noise with standard deviations $\sigma_{\text{noise}}=0.01$, $0.02$, $0.03$, $0.04$, and $0.05$ to each velocity component at each location on the computational grid.}
    \label{fig:sml_measurements_noise}
\end{figure*}

\section{Conclusion}
In this paper we have identified a common type of nonlinear structure that causes techniques for sensor placement relying on linear reconstruction accuracy as an optimization criterion to consistently fail to identify minimal sets of sensors.
Specifically, these techniques break down and lead to costly over-sensing when the data is intrinsically low dimensional, but is curved in such a way that energetic components are functions of less energetic ones, but not vice versa.
This problem occurs commonly in fluid flows, period-doubling bifurcations in ecology and cardiology, as well as in spectral methods for manifold learning.
We demonstrated that a representative collection of linear techniques fail to identify sensors from which the state of a shock-mixing layer flow can be reconstructed, and we provide a simple example that illustrates that the performance of the linear techniques can be arbitrarily bad.
In addition, we demonstrated that it is impossible to use linear feature selection methods to choose fundamental nonlinear eigen-coordinates in manifold learning problems.

To remedy these issues, we proposed a new approach for sensor placement that relies on the information contained in secant vectors between data points to quantify nonlinear reconstructability of desired quantities from measurements.
The resulting secant-based optimization problems turn out to have useful diminishing returns properties that enable efficient greedy approximation algorithms to achieve guaranteed high levels of performance.
We also describe how down-sampling can be used to improve the computational scaling of these algorithms while still providing guarantees regarding the reconstructability of states in the underlying set from which the available data is sampled.
Finally, these methods prove to be capable of selecting minimal collections of sensors in the shock-mixing layer problem as well as selecting the minimal set of fundamental manifold learning coordinates on a torus --- both of which are problems where the linear techniques fail.

\begin{acknowledgements}
The authors would like to thank Gregory Blaisdell,  Shih-Chieh Lo, Tasos Lyrintzis, and Kurt Aikens for providing the code used to compute the shock-mixing layer interaction. 
We also want to thank Alberto Padovan and Anastasia Bizyaeva for providing key references that motivate our main example, provide connections with period doubling, and reveal how linear methods can fail to find adequate sensor and actuator locations in real-world problems.
\end{acknowledgements}


%% file: secant_appendices.tex
\section{Implementation Details}
\label{app:implementation_details}

\subsection{Principal Component Analysis (PCA) and Isomap}
In this paper, we used principal component analysis (PCA) \cite{Hotelling1933analysis} in order to find a modal basis for pivoted QR factorization and to identify a low-dimensional representation of the state and its covariance for determinantal D-optimal selection techniques on the shock-mixing layer flow.
In order to perform PCA, one needs an appropriate inner product on the space in which the data lives.
In the case of the shock-mixing layer problem, we use the energy-based inner product for compressible flows developed in \cite{Rowley2004model} together with trapezoidal quadrature weights to approximate the integrals of the spatial fields over a stretched computational grid.
In this problem, the data consists of vectors $\vect{z}$ whose elements are the streamwise velocity $u$, transverse velocity $v$, and the local speed of sound $a$ over a $321\times 81$ computational grid.
The inner product between two snapshots $\vect{z}$ and $\vect{z}'$ is defined by
\begin{multline}
    \left\langle \vect{z},\ \vect{z}' \right\rangle = \vect{z}^T\mat{W}\vect{z}' = \sum_{i=1}^{321}\sum_{j=1}^{81}w_{i,j}\left( u_{i,j}^2 + v_{i,j}^2 + a_{i,j}^2 \right) \\
    \approx \int_{\Omega}\left[u(\xi_1,\xi_2)^2 + v(\xi_1,\xi_2)^2 + a(\xi_1,\xi_2)^2 \right] d\xi_1 d\xi_2,
\end{multline}
where the weights $\lbrace w_{i,j} \rbrace$ are selected to perform trapezoidal quadrature.
Principal component analysis is performed by computing an economy-sized singular value decomposition of the mean-subtracted data matrix
\begin{equation}
    \tilde{\mat{U}}\boldsymbol{\Sigma}\mat{V}^T = \mat{W}^{1/2} \begin{bmatrix} (\vect{z}_1-\bar{\vect{z}}) & \cdots & (\vect{z}_N-\bar{\vect{z}}) \end{bmatrix}, \quad \bar{\vect{z}} = \frac{1}{N}\sum_{i=1}^N \vect{z}_i
\end{equation}
and forming the matrix of principal vectors $\mat{U} = \mat{W}^{-1/2}\tilde{\mat{U}}$.
These vectors, making up the columns of $\mat{U}$, are orthonormal with respect to the $\mat{W}$-weighted inner product.
If we represent the states in this basis so that $\vect{z}_i = \bar{\vect{z}} + \mat{U}\vect{x}_i$ then $\vect{x}$ has empirical covariance $\mat{C}_{\vect{x}} = \frac{1}{N}\boldsymbol{\Sigma}^2$.

The same weighted inner product was used to compute the distances between each data point $\vect{z}_i$ and its $10$ nearest neighbors in order to compute the leading $50$ Isomap coordinates using scikit learn's implementation found at \url{https://scikit-learn.org/stable/modules/generated/sklearn.manifold.Isomap.html}.

\subsection{(Group) LASSO}

We use the Python implementation of group LASSO~\cite{Yuan2006model} by Yngve Mardal Moe at the University of Oslo that can be found at \url{https://group-LASSO.readthedocs.io/en/latest/index.html}.
We select among $2210$ sensor measurements of $u$ and $v$ velocity components over a grid of $1105$ spatial locations taken directly from the shock-mixing layer snapshot data.
We tried two different kinds of target variables to be reconstructed via group LASSO.
For the method we call ``LASSO+PCA'', the target variables were the data's leading $100$ principal components which capture over $99\%$ of the data's variance.
For the method we call ``LASSO+Isomap'', the target variables were the leading two Isomap coordinates $\vect{g}(\xx) = \left( \phi_1(\xx), \phi_2(\xx) \right)$, which reveal the phase angle $\theta$.
The sparsity-promoting regularization parameter was found using a bisection search in each case and was the smallest value, to within a tolerance of $10^{-5}$, for which group LASSO selected $3$ sensors.

\subsection{Bayesian D-Optimal Selection}
\label{sec:bayesian-d-optimal}
We use two different approaches for Bayesian D-optimal sensor placement: the greedy technique of \cite{Shamaiah2010greedy} and the convex relaxation approach by \cite{Joshi2008sensor}.
In the greedy approach, we leverage the submodularity of the objective in the case when $\mat{T}=\mat{I}$ in order to use the accelerated greedy algorithm of M.~Minoux \cite{Minoux1978accelerated}.
For the convex approach, we wrote a direct Python translation of a MATLAB code
written by S.~Joshi and S.~Boyd that implements a Newton method with line
search, and may be found at \url{https://web.stanford.edu/~boyd/papers/matlab/sensor_selection/}.
We use the gradient and Hessian matrices for the Bayesian D-optimal objective from their paper \cite{Joshi2008sensor}.

In both the greedy and convex approach for the shock-mixing layer problem, we take the state to be its representation using $100$ principal components with covariance given by $\mat{C}_{\vect{x}} = \frac{1}{N}\boldsymbol{\Sigma}^2$ as computed by PCA.
These principal components were also used as the relevant information to be reconstructed, i.e., $\mat{T} = \mat{I}$.
The sensor noise was assumed to be isotropic with covariance $\mat{C}_{\vect{n}_{\scrS}} = \sigma^2 \mat{I}_{d_{\scrS}}$ with $\sigma = 0.02$.
We tried many other values of $\sigma$, yielding different sensor locations, none of which could be used for nonlinear reconstruction.
The ones we show at $\sigma = 0.02$ are representative.

\subsection{Maximum Likelihood D-Optimal Selection}
We used the maximum likelihood D-optimal selection technique based on convex relaxation found in \cite{Joshi2008sensor} in order to choose sensors to try to reconstruct only the $3$rd and $4$th principal components of the shock-mixing layer snapshots.
That is, if $\mat{U} = \begin{bmatrix} \vect{u}_1 & \vect{u}_2 & \cdots \end{bmatrix}$ is the matrix of principal components, we model the state as a linear combination of $\vect{u}_3$ and $\vect{u}_4$ together with isotropic Gaussian noise.
We try to find the sensors so that the correct coefficients on $\vect{u}_3$ and $\vect{u}_4$ can be recovered with high confidence from the measurements.
The rationale for doing so is the fact that these two components are sufficient to nonlinearly reconstruct the state of the system if they can be measured. 
As in Section~\ref{sec:bayesian-d-optimal} above, we use a direct Python translation of a MATLAB code written by S.~Joshi and
S.~Boyd, which may be found at \url{https://web.stanford.edu/~boyd/papers/matlab/sensor_selection/}.

\subsection{Pivoted QR Factorization}
For the pivoted QR factorization method~\cite{Drmac2016new,Businger1965linear} applied to the shock-mixing layer flow, we represent the state approximately as a linear combination of the leading three principal components.
Scipy's implementation of pivoted QR factorization found at \url{https://docs.scipy.org/doc/scipy/reference/generated/scipy.linalg.qr.html} was used to select among the $2210$ allowable sensors those that allow robust reconstruction of these first three principal components.
We also tried representing the state using more principal components and taking the first three sensor locations chosen via pivoted QR factorization.
As with the case when only three principal components are used, these sensors do not enable nonlinear reconstruction of the state.

\subsection{Secant-Based Detectable Differences}
The secant-based detectable difference method was implemented using the accelerated greedy algorithm of M. Minoux \cite{Minoux1978accelerated} to optimize the objective computed over all secants between points in the training data set consisting of $N=750$ snapshots of the shock-mixing layer velocity field.
We select among the $2210$ sensor measurements of $u$ and $v$ velocity components on a grid of $1105$ spatial locations taken directly from the shock-mixing layer snapshot data.
The target variables were chosen to be the leading two Isomap coordinates $\vect{g}(\xx) = \left( \phi_1(\xx), \phi_2(\xx) \right)$, which reveal the phase angle $\theta$.
The greedy algorithm first reveals the two sensor locations marked by green stars and then the black star in Figure~\ref{fig:shock_mixing_layer_snapshots} over the range of $0.02 \leq \gamma \leq 0.06$, which can be used to reveal the exact phase of the system.
Choosing smaller values of $\gamma$ produce different sensors that can also be used to reveal the phase, but with reduced robustness to measurement perturbations.
Gaussian process regression \cite{Rasmussen2003gaussian} was used to reconstruct the leading $100$ principal components of the flowfields from the sensor measurements.
We used scikit learn's implementation which can be found at \url{https://scikit-learn.org/stable/modules/generated/sklearn.gaussian_process.GaussianProcessRegressor.html} together with a Mat\'{e}rn and white noise kernel whose parameters were optimized during the fit.

For the torus example, the relevant information we wish to reconstruct are the leading $100$ Isomap eigen-coordinates $\vect{g}(\xx) = \left( \phi_1(\xx), \ldots, \phi_{100}(\xx) \right)$ computed from $2000$ points sampled from the torus according to Eq.~\textbf{\ref{eqn:torus}}.
The objective function was evaluated using secants between $\#(\mathcal{B}) = 100$ randomly sampled base points and the original set of $N=2000$ points.
The correct three coordinates $\phi_1, \phi_2, \phi_7$ are selected from among the first $100$ consistently across a wide range of measurement separation values $0.05\leq \gamma \leq 3.0$.
We note that these values vary slightly with the selected base points and these particular values hold only for one instance.

\subsection{Secant-Based Amplification Tolerance}
Like the secant-based detectable difference method described above, the secant-based amplification tolerance method was implemented using the same data, secant vectors, and target variables with the accelerated greedy algorithm.
A bisection search was used to find the smallest Lipschitz constant $L=1868$ to within a tolerance of $1$ for which the algorithm selects three sensors on the shock-mixing layer flow.
Three (different) sensors that correctly reveal the state of the flow are selected by this algorithm over a range $1868\leq L \leq 47624$, above which only two sensors that cannot reveal the state are selected.
We also find that with $L=129$, the minimum possible number of sensors exceeds $\#(\scrS_K)/(1+\ln{\kappa}) = 3.18 > 3$.
Therefore, the minimum possible reconstruction Lipschitz constant using three sensors that one might find by an exhaustive combinatorial search must be greater than $129$.
We admit that this is likely a rather pessimistic bound, but we cannot check it as there are $\binom{2210}{3} \approx 1.8\times 10^9$ possible choices for three sensors in this problem.

When applied to select from among the leading $100$ Isomap eigen-coordinates on the torus example with the same setup as the secant-based detectable differences method, the amplification tolerance method selects the appropriate collection $\phi_1, \phi_2, \phi_7$ over the range $7.1\leq L \leq 25$.
We note that these value vary slightly with the selected base points and these particular values hold only for one instance.

\section{Submodularity of Objectives}

We will need the definition of a modular function given below.
\begin{definition}[Modular Function]
\label{def:modular_function}
Denote the set of all subsets of $\scrM$ by $2^\scrM$.
A real-valued function of the subsets $f:2^\scrM\to\mathbb{R}$ is called ``modular'' when it can be written as a sum
\begin{equation}
    f(\scrS) = \sum_{j\in\scrS} a_j
\end{equation}
of constants $a_j$, $j\in\scrM$.
\end{definition}
The key ingredient needed to prove submodularity for the objectives described in Section~\ref{sec:secants} is the following lemma.
\begin{lemma}[Concave Composed with Modular is Submodular]
\label{lem:concave_composed_with_modular_is_submodular}
Let $h:\mathbb{R}\to\mathbb{R}$ be a concave function and let $a:2^\scrM\to\mathbb{R}$ defined by
\begin{equation}
    a(\scrS) = \sum_{j\in\scrS} a_j
\end{equation}
be a modular function (Def.~\ref{def:modular_function}) of subsets $\scrS\subseteq\scrM$ with $a_j \geq 0$ for all $j\in\scrM$.
Then the function $f:2^\scrM\to\mathbb{R}$ defined by
\begin{equation}
    f(\scrS) = h(a(\scrS))
\end{equation}
is submodular.
\end{lemma}
\begin{proof}
Suppose that $\scrS\subseteq\scrS'\subseteq{\scrM}\setminus\lbrace j \rbrace$.
By concavity of $h$ we have
\begin{equation}
    h_{\alpha} = h((1-\alpha) a(\scrS) + \alpha (a(\scrS') + a_j)) \geq (1-\alpha) h_{0} + \alpha h_{1}
    \label{eqn:concavity_of_h}
\end{equation}
for every $\alpha \in [0,1]$, where we note that $h_0 = f(\scrS)$ and $h_1 = f(\scrS'\cup\lbrace j\rbrace)$.

Since $\lbrace a_{l} \rbrace$ are non-negative we have $a(\scrS) \leq a(\scrS) + a_j \leq a(\scrS') + a_j$ and $a(\scrS) \leq a(\scrS') \leq a(\scrS') + a_j$.
We can therefore find
\begin{equation}
    \alpha_1 = \frac{a_j}{a(\scrS') + a_j - a(\scrS)}, \quad \alpha_2 = \frac{a(\scrS') - a(\scrS)}{a(\scrS') + a_j - a(\scrS)}
\end{equation}
so that $h_{\alpha_1} = f(\scrS\cup\lbrace j\rbrace)$ and $h_{\alpha_2} = f(\scrS')$.
Note that $\alpha_1 + \alpha_2 = 1$.

We now use Eq.~\textbf{\ref{eqn:concavity_of_h}} at $\alpha_1$ and $\alpha_2$ to bound the increments of $f$:
\begin{equation}
    f(\scrS\cup\lbrace j\rbrace) - f(\scrS) = h_{\alpha_1} - h_0 \geq \alpha_1 (h_1 - h_0),
    \label{eqn:increment_bnd_1}
\end{equation}
\begin{equation}
    f(\scrS'\cup\lbrace j\rbrace) - f(\scrS') = h_1 - h_{\alpha_2} \leq (1 - \alpha_2) (h_1 - h_0)
    \label{eqn:increment_bnd_2}
\end{equation}
Combining the bounds Eq.~\textbf{\ref{eqn:increment_bnd_1}} and Eq.~\textbf{\ref{eqn:increment_bnd_2}} on the increments using $1 - \alpha_2 = \alpha_1$ we conclude that $f$ is submodular
\begin{equation}
    f(\scrS\cup\lbrace j\rbrace) - f(\scrS) \geq f(\scrS'\cup\lbrace j\rbrace) - f(\scrS').
  \end{equation}
\end{proof}
Using Lemma~\ref{lem:concave_composed_with_modular_is_submodular} it suffices to observe that each of the objectives described in Section~\ref{sec:secants} can be written as the composition of a concave function and a modular function.
We carry this out below in addition to proving normalization and monotonicity for these objectives.
\begin{lemma}[Detectable Difference Objective is Submodular]
    \label{lem:submodularity_of_detectable_differences}
    Suppose that the target variables $\vect{g}$ and measurements $\vect{m}_j$, $j\in \scrM$ are measurable functions.
    If $\mu$ and $\nu$ are measures on $\calX$, then the function defined by
    \begin{equation}
        f(\scrS) = \int_{\substack{(\xx, \xx')\in\calX\times\calX : \\ \Vert \vect{g}(\xx) - \vect{g}(\xx')\Vert_2 \geq \varepsilon}} w_{\gamma,\xx,\xx'}(\scrS) \Vert \vect{g}(\xx) - \vect{g}(\xx')\Vert_2^2 \ d\mu(\xx) \nu(d\xx'),
    \end{equation}
    for any $\varepsilon\geq 0$ with 
    \begin{equation}
        w_{\gamma, \xx, \xx'}(\scrS) = \min\left\lbrace\frac{1}{\gamma^2}\Vert \mS(\xx) - \mS(\xx') \Vert_2^2,\ 1 \right\rbrace,
    \end{equation}
    is normalized so that $f(\emptyset) = 0$, monotone non-decreasing so that $\scrS\subseteq\scrS'\ \Rightarrow\ f(\scrS) \leq f(\scrS')$, and submodular (Def.~\ref{def:submodular_function}).
\end{lemma}
\begin{proof}
Normalization is obvious.
It suffices to prove that the function $w_{\xx, \xx'}(\scrS)$ is monotone and submodular for any fixed $\xx,\xx'\in\calX$.
For if we suppose that
\begin{multline}
    \scrS\subseteq\scrS'\subseteq{\scrM}\setminus\lbrace j \rbrace \quad \Rightarrow \quad
    w_{\gamma,\xx,\xx'}(\scrS\cup\lbrace j\rbrace) - w_{\gamma,\xx,\xx'}(\scrS) \\
    \geq w_{\gamma,\xx,\xx'}(\scrS'\cup\lbrace j\rbrace) - w_{\gamma,\xx,\xx'}(\scrS'),
\end{multline}
then multiplying both sides of the inequality by $\Vert \vect{g}(\xx) - \vect{g}(\xx')\Vert_2^2$ and integrating proves that $f$ is submodular.
The same argument also proves monotonicity.

Let $\xx,\xx'\in\calX$ be fixed.
The squared separation between the measurements is given by a modular (Def.~\ref{def:modular_function}) sum
\begin{equation}
    \scrS \ \mapsto\ \Vert \mS(\xx) - \mS(\xx')\Vert_2^2 = \sum_{j\in\scrS} \Vert \vect{m}_j(\xx) - \vect{m}_j(\xx')\Vert_2^2
    \label{eqn:lem_modularity_of_squared_norm}
\end{equation}
of non-negative constants $\Vert \vect{m}_j(\xx) - \vect{m}_j(\xx')\Vert_2^2$ over each $j\in\scrS$.
Since $x\mapsto \min\lbrace x/\gamma^2,\ 1 \rbrace$ is a non-decreasing function, it follows that $\scrS\subseteq\scrS'\ \Rightarrow\ w_{\xx, \xx'}(\scrS) \leq w_{\xx, \xx'}(\scrS')$, proving monotonicity.

Submodularity of $w_{\xx, \xx'}(\scrS)$ follows from Lemma~\ref{lem:concave_composed_with_modular_is_submodular} since $w_{\xx, \xx'}(\scrS)$ is the composition of a concave function $x\mapsto \min\lbrace x/\gamma^2,\ 1 \rbrace$ with the modular function in Eq.~\textbf{\ref{eqn:lem_modularity_of_squared_norm}}. \end{proof}
\begin{lemma}[Lipschitz Objective is Submodular]
    \label{lem:submodularity_of_Lipschitz_objective}
    Suppose that the target variables $\vect{g}$ and measurements $\vect{m}_j$, $j\in \scrM$ are measurable functions.
    If $\mu$ and $\nu$ are measures on $\calX$, then the function defined by
    \begin{equation}
        f(\scrS) = \int_{\substack{(\xx, \xx')\in\calX\times\calX : \\ \vect{g}(\xx) \neq \vect{g}(\xx')}} g_{\xx,\xx'}(\scrS) \ d\mu(\xx)\nu(d\xx'),
    \end{equation}
    with
    \begin{equation}
        g_{\xx,\xx'}(\scrS) = \min\left\lbrace \frac{\Vert \mS(\xx) - \mS(\xx') \Vert_2^2 }
        {\Vert \vect{g}(\xx) - \vect{g}(\xx') \Vert_2^2},\ 
        \frac{1}{L^2} \right\rbrace,
    \end{equation}
    is normalized so that $f(\emptyset) = 0$, monotone non-decreasing so that $\scrS\subseteq\scrS'\ \Rightarrow\ f(\scrS) \leq f(\scrS')$, and submodular (Def.~\ref{def:submodular_function}).
\end{lemma}
\begin{proof}
Normalization is obvious.
It suffices to prove that the function $g_{\xx, \xx'}(\scrS)$ is monotone and submodular for any fixed $\xx,\xx'\in\calX$.
For if we suppose that
\begin{equation}
    \scrS\subseteq\scrS'\subseteq{\scrM}\setminus\lbrace j \rbrace \quad \Rightarrow \quad
    g_{\xx,\xx'}(\scrS\cup\lbrace j\rbrace) - g_{\xx,\xx'}(\scrS)
    \geq g_{\xx,\xx'}(\scrS'\cup\lbrace j\rbrace) - g_{\xx,\xx'}(\scrS'),
\end{equation}
then integrating both sides of the inequality proves that $f$ is submodular.
The same argument also proves monotonicity.

Let $\xx,\xx'\in\calX$ be fixed.
The squared separation between the measurements is given by a modular (Def.~\ref{def:modular_function}) sum
\begin{equation}
    \scrS \ \mapsto\ \Vert \mS(\xx) - \mS(\xx')\Vert_2^2 = \sum_{j\in\scrS} \Vert \vect{m}_j(\xx) - \vect{m}_j(\xx')\Vert_2^2
    \label{eqn:lem_modularity_of_squared_norm_2}
\end{equation}
of non-negative constants $\Vert \vect{m}_j(\xx) - \vect{m}_j(\xx')\Vert_2^2$ over each $j\in\scrS$.
Since 
\begin{equation}
    x\ \mapsto\ \min\left\lbrace \frac{x}{\Vert \vect{g}(\xx) - \vect{g}(\xx')\Vert_2^2},\ \frac{1}{L^2} \right\rbrace
    \label{eqn:lem_Lipschitz_tmp_fcn}
\end{equation}
is a non-decreasing function, it follows that $\scrS\subseteq\scrS'\ \Rightarrow\ g_{\xx, \xx'}(\scrS) \leq g_{\xx, \xx'}(\scrS')$, proving monotonicity.

Submodularity of $g_{\xx, \xx'}(\scrS)$ follows from Lemma~\ref{lem:concave_composed_with_modular_is_submodular} since $g_{\xx, \xx'}(\scrS)$ is the composition of the concave function in Eq.~\textbf{\ref{eqn:lem_Lipschitz_tmp_fcn}} with the modular function in Eq.~\textbf{\ref{eqn:lem_modularity_of_squared_norm_2}}.
\end{proof}

\section{Proofs}
\label{app:proofs}

\begin{proof}[Proposition~\ref{prop:error_tol_separation}: Separation Guarantee on Underlying Set]
The result follows immediately from the triangle inequality.
Let $\xx,\xx'\in\calX$ and $\xx_i,\xx_j\in\calX_N$ so that $\Vert \xx - \xx_i\Vert_2 < \varepsilon_0$ and $\Vert \xx' - \xx_j\Vert_2 < \varepsilon_0$.
Then $\varepsilon + 2\varepsilon_0 \Vert \vect{g}\Vert_{\text{lip}} \leq \Vert \vect{g}(\xx) - \vect{g}(\xx') \Vert_2$ implies that
\begin{equation}
\begin{split}
    \varepsilon + 2\varepsilon_0 \Vert \vect{g}\Vert_{\text{lip}} &\leq \Vert \vect{g}(\xx) - \vect{g}(\xx') \Vert_2 \\
    &\leq \Vert \vect{g}(\xx) - \vect{g}(\xx_i) \Vert_2 + \Vert \vect{g}(\xx') - \vect{g}(\xx_j) \Vert_2 + \Vert \vect{g}(\xx_i) - \vect{g}(\xx_j) \Vert_2 \\
    &< \Vert \vect{g}(\xx_i) - \vect{g}(\xx_j) \Vert_2 + 2\varepsilon_0 \Vert \vect{g}\Vert_{\text{lip}},
\end{split}
\end{equation}
hence, $\Vert \vect{g}(\xx_i) - \vect{g}(\xx_j) \Vert_2 \geq \varepsilon$.
By assumption, this implies that $\Vert \mS(\xx_i) - \mS(\xx_j) \Vert_2 \geq \gamma$ and
\begin{equation}
\begin{split}
    \gamma &\leq \Vert \mS(\xx_i) - \mS(\xx_j) \Vert_2 \\
    &\leq \Vert \mS(\xx_i) - \mS(\xx) \Vert_2 + \Vert \mS(\xx') - \mS(\xx_j) \Vert_2 + \Vert \mS(\xx) - \mS(\xx') \Vert_2 \\
    &< 2 \varepsilon_0 \Vert \mS\Vert_{\text{lip}} + \Vert \mS(\xx) - \mS(\xx') \Vert_2,
\end{split}
\end{equation}
hence, $\Vert \mS(\xx) - \mS(\xx') \Vert_2 > \gamma - 2 \varepsilon_0 \Vert \mS\Vert_{\text{lip}}$ as claimed.
\end{proof}

\begin{proof}[Proposition~\ref{prop:amplification_tol_separation}: Amplification Guarantee on Underlying Set]
The result follows immediately from the triangle inequality.
Let $\xx,\xx'\in\calX$ and $\xx_i,\xx_j\in\calX_N$ so that $\Vert \xx - \xx_i\Vert_2 < \varepsilon_0$ and $\Vert \xx' - \xx_j\Vert_2 < \varepsilon_0$, then
\begin{equation}
\begin{split}
    \Vert \vect{g}(\xx) - \vect{g}(\xx') \Vert_2 &\leq \Vert \vect{g}(\xx) - \vect{g}(\xx_i) \Vert_2 + \Vert \vect{g}(\xx') - \vect{g}(\xx_j) \Vert_2
    + \Vert \vect{g}(\xx_i) - \vect{g}(\xx_j) \Vert_2 \\
    &< 2\varepsilon_0\Vert \vect{g}\Vert_{\text{lip}} + L \Vert \mS(\xx_i) - \mS(\xx_j) \Vert_2 \\
    &\leq 2\varepsilon_0\Vert \vect{g}\Vert_{\text{lip}} + L\Vert \mS(\xx) - \mS(\xx_i) \Vert_2 + L \Vert \mS(\xx') - \mS(\xx_j) \Vert_2 \\
    & \hspace{2cm} + L\Vert \mS(\xx) - \mS(\xx') \Vert_2 \\
    &< 2\varepsilon_0\Vert \vect{g}\Vert_{\text{lip}} + 2 L \varepsilon_0\Vert \mS \Vert_{\text{lip}} + L \Vert \mS(\xx) - \mS(\xx') \Vert_2.
\end{split}
\end{equation}
Gathering terms on $\varepsilon_0$ completes the proof.
\end{proof}

\begin{proof}[Proposition~\ref{prop:noisy_separation_guarantee}: Noisy Separation Guarantee]
Choose $\vect{x}_i, \vect{x}_j\in\calX_N$ and suppose that 
\begin{equation}
    \Vert \vect{g}(\xx_i) - \vect{g}(\xx_j) \Vert_2 \geq \varepsilon + 2\delta_{v}.
\end{equation}
Then we have
\begin{equation}
\begin{split}
    \Vert \left(\vect{g}(\xx_i) + \vect{v}_i\right) - \left(\vect{g}(\xx_j) + \vect{v}_j\right) \Vert_2 &\geq \Vert \vect{g}(\xx_i) - \vect{g}(\xx_j) \Vert_2 - \Vert \vect{v}_i \Vert - \Vert \vect{v}_j \Vert \\
    &\geq \Vert \vect{g}(\xx_i) - \vect{g}(\xx_j) \Vert_2 - 2\delta_{v} \\
    &\geq \varepsilon
\end{split}
\end{equation}
By our assumption, this implies 
\begin{equation}
    \Vert \left(\mS(\xx_i) + \vect{u}_{i,\scrS}\right) - \left(\mS(\xx_j) + \vect{u}_{j,\scrS}\right) \Vert_2 \geq \gamma,
\end{equation}
and so we have
\begin{equation}
\begin{split}
    \Vert \mS(\xx_i) - \mS(\xx_j) + \vect{u}_{j,\scrS} \Vert_2 &\geq 
    \Vert \left(\mS(\xx_i) + \vect{u}_{i,\scrS}\right) - \left(\mS(\xx_j) + \vect{u}_{j,\scrS}\right) \Vert_2 - \Vert \vect{u}_{i,\scrS} \Vert - \Vert \vect{u}_{j,\scrS} \Vert \\
    &\geq \gamma - 2\delta_{u}.
\end{split}
\end{equation}
Therefore, we have established that
\begin{multline}
    \forall \vect{x}_i, \vect{x}_j\in\calX_N \qquad 
    \Vert \vect{g}(\xx_i) - \vect{g}(\xx_j) \Vert_2 \geq \varepsilon + 2\delta_{v} \\
    \Rightarrow \quad \Vert \mS(\xx_i) - \mS(\xx_j) + \vect{u}_{j,\scrS} \Vert_2 \geq \gamma - 2\delta_{u}.
\end{multline}
The conclusion follows immediately by Proposition~\ref{prop:error_tol_separation}.
\end{proof}

\begin{proof}[Theorem~\ref{thm:minimal_sensing_sampled_amplification_tolerance}: Down-Sampled Amplification Guarantee]
For simplicity, we will drop $L$ from the subscript on our objective since the threshold $L$ for the Lipschitz constant remains fixed throughout the proof.
Let us begin by fixing a set $\scrS\subseteq\scrM$ and define the random variables
\begin{equation}
    Z_{\scrS}(\vect{b}_i) = \max_{\xx\in\calX_N} \bbone\big\lbrace \Vert \vect{g}(\vect{b}_i) - \vect{g}(\xx)\Vert_2 > L \Vert \mS(\vect{b}_i) - \mS(\xx)\Vert_2 \big\rbrace,
\end{equation}
for $i=1,\ldots,m$.
If $Z_{\scrS}(\vect{b}_i) = 0$ then every secant between $\vect{b}_i$ and points of $\calX_N$ satisfies the desired bound on the amplification.
Otherwise, there is some point $\xx\in\calX$ for which 
\begin{equation}
    \Vert \vect{g}(\vect{b}_i) - \vect{g}(\xx)\Vert_2 > L \Vert \mS(\vect{b}_i) - \mS(\xx)\Vert_2
\end{equation}
and so $Z_{\scrS}(\vect{b}_i) = 1$.
We observe that $Z_{\scrS}(\vect{b}_1),\ldots,Z_{\scrS}(\vect{b}_m)$ are independent, identically distributed Bernoulli random variables whose expectation
\begin{multline}
    \mathbb{E}[Z_{\scrS}(\vect{b}_i)] = \mu \big(\big\lbrace \xx\in\calX\ :\ \exists\xx_j\in\calX_N\quad \mbox{s.t.} \\
    \Vert \vect{g}(\xx) - \vect{g}(\xx_j)\Vert_2 > L \Vert \mS(\xx) - \mS(\xx_j)\Vert_2 \big\rbrace \big)
\end{multline}
is the $\mu$-measure of points in $\calX$ that are not adequately separated from points in the $\varepsilon_0$-net $\calX_N$ by the measurements $\mS$.
Suppose that for a fixed $\xx\in\calX$ we have
\begin{equation}
    \Vert \vect{g}(\xx) - \vect{g}(\xx_j)\Vert_2 \leq L \Vert \mS(\xx) - \mS(\xx_j)\Vert_2
\end{equation}
for every $\xx_j\in\calX_N$.
By definition of $\calX_N$, for any $\xx'\in\calX$, there is an $\xx_j\in\calX_N$ with $\Vert \xx' - \xx_j\Vert_2 < \varepsilon_0$ and so we have
\begin{equation}
\begin{split}
    \Vert \vect{g}(\xx) - \vect{g}(\xx')\Vert_2 &\leq \Vert \vect{g}(\xx) - \vect{g}(\xx_j)\Vert_2 + \Vert \vect{g}(\xx_j) - \vect{g}(\xx')\Vert_2 \\
    &< L\Vert \mS(\xx) - \mS(\xx_j)\Vert_2 + \varepsilon_0\Vert \vect{g}\Vert_{\text{lip}} \\
    &\leq L\Vert \mS(\xx) - \mS(\xx')\Vert_2 + L\Vert \mS(\xx') - \mS(\xx_j)\Vert_2 + \varepsilon_0\Vert \vect{g}\Vert_{\text{lip}} \\
    &< L\Vert \mS(\xx) - \mS(\xx')\Vert_2 + \left( \Vert \vect{g}\Vert_{\text{lip}} + L\Vert \mS \Vert_{\text{lip}}  \right) \varepsilon_0.
\end{split}
\end{equation}
It follows that $\mathbb{E}[Z_{\scrS}(\vect{b}_i)]$ is an upper bound on the $\mu$-measure of points in $\calX$ for which the relaxed amplification threshold is exceeded, that is,
\begin{multline}
    \mathbb{E}[Z_{\scrS}(\vect{b}_i)] \geq \mu \big(\big\lbrace \xx\in\calX\ :\ \exists\xx'\in\calX \quad \mbox{s.t.} \quad
    \Vert \vect{g}(\xx) - \vect{g}(\xx')\Vert_2  \\
    \geq L \Vert \mS(\xx) - \mS(\xx')\Vert_2 + \left( \Vert \vect{g}\Vert_{\text{lip}} + L\Vert \mS \Vert_{\text{lip}}  \right) \varepsilon_0 \big\rbrace \big).
    \label{eqn:thm_measure_of_points_with_large_amplification}
\end{multline}

By assumption, we have a set $\scrS\subseteq\scrM$ so that $Z_{\scrS}(\vect{b}_i) = 0$ for each $i=1,\ldots,m$.
And so it remains to bound the difference between the empirical and true expectation of $Z_{\scrS}(\vect{b}_i)$ uniformly over every subset $\scrS\subseteq\scrM$.
For fixed $\scrS$, the one-sided Hoeffding inequality gives
\begin{equation}
    \mathbb{P}\Big\lbrace \frac{1}{m}\sum_{i=1}^m\left( \mathbb{E}[Z_{\scrS}(\vect{b}_i)] - Z_{\scrS}(\vect{b}_i) \right) \geq \delta \Big\rbrace \leq e^{-2 m \delta^2}.
\end{equation}
Unfixing $\scrS$ via the union bound over all $\scrS\subseteq\scrM$ and applying our assumption about the number of base points $m$ yields
\begin{equation}
    \mathbb{P}\bigcup_{\scrS\subseteq\scrM}\Big\lbrace \frac{1}{m}\sum_{i=1}^m\left( \mathbb{E}[Z_{\scrS}(\vect{b}_i)] - Z_{\scrS}(\vect{b}_i) \right) \geq \delta \Big\rbrace
    \leq e^{ \#(\scrM)\ln{2} - 2 m \delta^2}
    \leq p.
\end{equation}
Since our assumed choice of $\scrS$ has $f_m(\scrS) = f_m(\scrM)$ it follows that all $Z_{\scrS}(\vect{b}_i) = 0$, $i=1,\ldots,m$, hence we have
\begin{equation}
    \mathbb{E}[Z_{\scrS}(\vect{b}_i)] < \delta
\end{equation}
with probability at least $1-p$.
Combining this with Eq.~\textbf{\ref{eqn:thm_measure_of_points_with_large_amplification}} completes the proof.
\end{proof}

\section{Description of the Accelerated Greedy Algorithm}
\label{subapp:AG_algorithm}
Since each objective function $f$ presented in Section~\ref{sec:secants} is submodular, it is possible to use an ``accelerated greedy'' (AG) algorithm to obtain the same solution as the naive greedy algorithm with a provably minimal number of objective function evaluations compared to a broad class of algorithms \cite{Minoux1978accelerated}.
Let the increase in the objective function obtained by adding the sensor $j$ to the set $\scrS$ be called $\Delta_j(\scrS) = f(\scrS\cup\lbrace j\rbrace) - f(\scrS)$.
Instead of evaluating $\Delta_j(\scrS_{k-1})$ for every measurement in $\scrM\setminus\scrS_{k-1}$, AG keeps track of an upper bound $\hat{\Delta}_j \geq \Delta_j(\scrS_{k-1})$ on the increments for each sensor.
Since submodularity of $f$ means that the increments $\Delta_j(\scrS)$ can only decrease as the size of $\scrS$ increases, it is sufficient to have the maximum upper bound $\hat{\Delta}_{j^*} \geq \hat{\Delta}_{j}$, $\forall j\in\scrM\setminus\scrS_{k-1}$ be tight $\hat{\Delta}_{j^*} = \Delta_{j^*}(\scrS_{k-1})$ in order to conclude that $\Delta_{j^*}(\scrS_{k-1})$ is the largest increment.
The rest of the upper bounds on the increments can remain loose since they are smaller than the tight maximum upper bound.
The AG algorithm finds largest upper bound $\hat{\Delta}_{j^*}$ and updates it so that it is tight.
If $\hat{\Delta}_{j^*}$ is still the greatest upper bound, then $j^* = j_k$ achieves the largest increment and is added to $\scrS_{k-1}$.
Otherwise if $\hat{\Delta}_{j^*}$ is no longer the largest upper bound, the new largest upper bound is selected and process repeated until a tight maximum upper bound is obtained.
